\documentclass[twoside,12pt]{article} 
\usepackage[margin=1.0in]{geometry}

\usepackage{amssymb,amsmath,amsthm}
\usepackage{mathtools}
\usepackage{bm,lscape}
\usepackage{bbm}
\usepackage{mathrsfs,dsfont}
\usepackage{wasysym}
\RequirePackage[
    colorlinks=true,
    linkcolor=blue,
    urlcolor=blue,
    citecolor=blue]{hyperref}
\RequirePackage{natbib}
\usepackage[]{algorithm2e}
\usepackage[utf8]{inputenc}
\usepackage[english]{babel}

\usepackage{graphicx}
\usepackage{color}
\usepackage{rotating}
\usepackage{authblk}
\usepackage{appendix}

\allowdisplaybreaks

\newcommand{\ud}{\mathrm{d}}

\newcommand{\F}{\mathsf{F}}
\newcommand{\G}{\mathsf{G}}
\newcommand{\Hg}{\mathsf{H}}
\newcommand{\K}{\mathsf{K}}

\newcommand{\ub}{\mathbf{u}}

\newcommand{\zb}{\mathbf{z}}
\newcommand{\xb}{\mathbf{x}}
\newcommand{\Xb}{\mathbf{X}}
\newcommand{\yb}{\mathbf{y}}
\newcommand{\Yb}{\mathbf{Y}}
\newcommand{\mb}{\mathbf{m}}
\newcommand{\Cb}{\mathbf{C}}
\newcommand{\tb}{\mathbf{t}}
\newcommand{\bbb}{\mathbf{b}}

\newcommand{\Abb}{\mathbf{A}}
\newcommand{\pb}{\mathbf{p}}
\newcommand{\Zb}{\mathbf{Z}}
\newcommand{\lb}{\mathbf{l}}
\newcommand{\nub}{\boldsymbol{\nu}}
\newcommand{\Abn}{\mathbf{A}_n}
\newcommand{\Lb}{\mathbf{L}}
\newcommand{\xib}{\boldsymbol{\xi}}
\newcommand{\alphab}{\boldsymbol{\alpha}}
\newcommand{\epsilonb}{\boldsymbol{\epsilon}}

\newcommand{\R}{\mathbb{R}}
\newcommand{\E}{\mathds{E}}
\newcommand{\N}{\mathbb{N}}

\newcommand{\ddr}{\mathrm{d}}

\newcommand{\ep}{\varepsilon}

\newcommand{\cepsm}{c_{\varepsilon,m}}
\usepackage{booktabs}

\newtheorem{theorem}{Theorem}
\newtheorem{proposition}[theorem]{Proposition}
\newtheorem{lem}[theorem]{Lemma}

\newtheorem{definition}[theorem]{Definition}
\newtheorem{remark}[theorem]{Remark}

\providecommand{\keywords}[1]
{
  \small	
  \textbf{\textit{Keywords:}} #1
}

\begin{document}

\title{The power of private likelihood-ratio tests for goodness-of-fit in frequency tables}

% \author{
% Emanuele Dolera\\
% \texttt{emanuele.dolera@unipv.it}\\
% Department of Mathematics\\ 
% University of Pavia, Italy\\
% \and
% Stefano Favaro\\
% \texttt{stefano.favaro@unito.it}\\
% Department of Economics and Statistics\\ 
% University of Torino and Collegio Carlo Alberto, Italy\\
% \and
% Stefano Peluchetti\\
% \texttt{speluchetti@cogent.co.jp}\\
% Cogent Labs, Tokyo, Japan
% }

\author[1]{Emanuele Dolera\thanks{emanuele.dolera@unipv.it}}
\author[2]{Stefano Favaro \thanks{stefano.favaro@unito.it}}
\affil[1]{\small{Department of Mathematics, University of Pavia, Italy}}
\affil[2]{\small{Department of Economics and Statistics, University of Torino and Collegio Carlo Alberto, Italy}}

\maketitle

\begin{abstract}
Privacy-protecting data analysis investigates statistical methods under privacy constraints. This is a rising challenge in modern statistics, as the achievement of confidentiality guarantees, which typically occurs through suitable perturbations of the data, may determine a loss in the statistical utility of the data. In this paper, we consider privacy-protecting tests for goodness-of-fit in frequency tables, this being arguably the most common form of releasing data, and present a rigorous analysis of the large sample behaviour of a private likelihood-ratio (LR) test. Under the framework of $(\varepsilon,\delta)$-differential privacy for perturbed data, our main contribution is the power analysis of the private LR test, which characterizes the trade-off between confidentiality, measured via the differential privacy parameters $(\varepsilon,\delta)$, and statistical utility, measured via the power of the test. This is obtained through a Bahadur-Rao large deviation expansion for the power of the private LR test, bringing out a critical quantity, as a function of the sample size, the dimension of the table and $(\varepsilon,\delta)$, that determines a loss in the power of the test. Such a result is then applied to characterize the impact of the sample size and the dimension of the table, in connection with the parameters $(\varepsilon,\delta)$, on the loss of the power of the private LR test. In particular, we determine the (sample) cost of $(\varepsilon,\delta)$-differential privacy in the private LR test, namely the additional sample size that is required to recover the power of the Multinomial LR test in the absence of perturbation. Our power analysis rely on a non-standard large deviation analysis for the LR, as well as the development of a novel (sharp) large deviation principle for sum of i.i.d. random vectors, which is of independent interest.
\end{abstract}


%%%%%%%%%%%%%%%%%%%%%%%%%%%%%%%%
%%%%%%%%%%%%%%%%%%%%%%%%%%%%%%%%
%%%%%%%%%%%%%%%%%%%%%%%%%%%%%%%%
%%%%%%%%%%%%%%%%%%%%%%%%%%%%%%%%

\section{Introduction}

Privacy-protecting data analysis is a rising subject in modern statistics, building upon the following challenge: for given data, say $D$, how to determine a transformation $\mathcal{M}$, called (perturbation) mechanism, such that if $\mathcal{M}(D)$ is released then confidentiality will be protected and also the value of $D$ for statistical analysis, called utility, will be preserved in $\mathcal{M}(D)$? Measuring utility is common in statistics, and the decrease of utility arising from releasing $\mathcal{M}(D)$ rather than $D$ may be measured as the loss in the accuracy of a statistical method applied to the undertaken data analysis. Measuring confidentiality has been attracting much attention in computer science, where differential privacy (DP) has been put forth as a mathematical framework to quantify privacy guarantees \citep{Dwo(06),Dwo(06b)}. Roughly speaking, DP requires that the distribution of $\mathcal{M}(D)$ remains almost unchanged when an individual is included or removed from $D$, thus ensuring that nothing can be learnt about individuals. In a recent work, \citet{Rin(17)} provided a comprehensive, and practically oriented, treatment of DP in the dissemination of frequency tables, this being arguably the most common form of releasing data. Consider data to be arranged in a list of $k>1$ cells $\mathbf{a}=(a_{1},\ldots,a_{k})$, with $\sum_{1\leq i\leq k}a_{i}=n>k$, where $a_{i}$ is the number of individuals taking the attribute values corresponding to cell $i$, for $i=1,\ldots,k$. Under the curator framework of DP, or global DP, \citet{Rin(17)} introduced a class of truncated exponential  mechanisms (EMs) and showed how they allow to increase utility of the perturbed list $\mathbf{b}=\mathcal{M}(\mathbf{a})$ at the cost of relaxing DP to the $(\varepsilon,\delta)$-DP \citep{Dwo(06b),Dwo(13)}. The parameters $\varepsilon$ and $\delta$ control the level of privacy against intruders:  privacy guarantees become more stringent as $\varepsilon$ and $\delta$ tend to zero, with the DP corresponding to $\delta=0$. Then, an empirical analysis of truncated EMs is presented for the problem of testing goodness-of-fit in frequency tables of dimension $k=2$, showing the effect of the perturbation in the power of Pearson's chi-squared and likelihood-ratio (LR) tests. See \citet{Wan(15)} and \citet{Kif(17)} for similar analyses under DP.

\subsection{Our contributions}
In this paper, we present a rigorous analysis of the large sample behaviour of the LR test for goodness-of-fit under the $(\varepsilon,\delta)$-DP framework of \citet{Rin(17)}. We focus on the popular truncated Laplace EM, though our results can be easily extended to a broad class of truncated EMs discussed in \citet{Rin(17)}. The perturbed list $\mathbf{b}$ is assumed to be modeled as the convolution between a Multinomial distribution with parameter $(n,\mathbf{p})$ and the distribution of i.i.d truncated (discrete) Laplace random variables on $\{-m,\ldots,m\}$ with location $0$ and scale $\varepsilon^{-1}$, in such a way that the sample size of $\mathbf{b}$ is $n$. Under such a model, which is referred to as the  ``true" model,  we consider the LR test to assess goodness-of-fit in the form: $H_{0}:\mathbf{p}=\mathbf{p}_{0}$, for a fixed $\mathbf{p}_{0}\in\Delta_{k-1}$, against $H_{1}:\mathbf{p}\neq\mathbf{p}_{0}$, where $\Delta_{k-1}=\{\mathbf{p}\in[0,1]^{k-1}\text{ : }|\mathbf{p}|\leq 1\}$ and $|\mathbf{p}|=\sum_{1\leq i\leq k-1}p_{i}$. First, we establish an Edgeworth expansion of the distribution the ``true" LR, providing a ``private", and more accurate, version of the classical chi-squared limit of the LR \citep{Wil(38)}. Then, our main contribution provides a quantitative characterization of the trade-off between confidentiality, measured via the DP parameters $\varepsilon$ and $m$, and utility, measured via the power of the LR test. In particular, we rely on non-standard large deviation analysis \citep{Pet(75),SS(91)} to establish a Bahadur-Rao large deviation expansion for the power of the ``true" LR test. This result provides a ``private" version of a theorem by \citet{Hoe(65),Hoe(67)} on the power of the Multinomial LR test. See also \citet{Bah(60),Rao(62),Bah(67),Efr(67),Efr(68)}. Our large deviation expansion brings out a critical quantity, as a function of $n$, $k$, $\varepsilon$ and $m$, which determines a loss in the power of the ``true" LR test. This leads to characterize the impact of $n$ and $k$, in connection with the parameters $(\varepsilon,\delta)$, on the loss of the power of the private LR test. Concretely, we determine the (sample) cost of $(\varepsilon,\delta)$-DP under the ``true" LR test, namely the additional sample size required to recover the power of the Multinomial LR test in the absence of perturbation.

As a complement to our power analysis of the ``true" LR test, we investigate the well-known problem of releasing negative values in frequency tables under the $(\varepsilon,\delta)$-DP. The ``true" model allows for negative values in the perturbed list $\mathbf{b}$, which may be questionable in the context of frequency tables. If publishing data with negative values is not acceptable for some reason, then the common policy is to post-process $\mathbf{b}$ by reporting negative values as zeros, which preserves $(\varepsilon,\delta)$-DP \citep{Dwo(13)}. However, as a matter of fact, releasing lists that have an appearance similar to that of original lists may lead to ignoring the perturbation and analyze data as if they were not perturbed, which is known empirically to provide unreliable conclusions \citep{Fie(10),Rin(17)}. Following \citet{Rin(17)}, we consider a ``na\"ive" model for the perturbed list $\mathbf{b}$, that is a statistical model that does not take the exponential EM into account. In particular, if $\mathbf{b}^{+}=(b_{1}^{+},\ldots,b_{k}^{+})$ denotes the post-processed list, which is obtained from the $(\varepsilon,\delta)$-DP list $\mathbf{b}$ by setting negative values to be equal to zeroes, then the ``na\"ive" model assumes $\mathbf{b}^{+}$ to be modeled as the Multinomial distribution with parameter $\sum_{1\leq i\leq k}b_{i}^{+}$ and $\mathbf{p}$.  Under the ``na\"ive" model, we consider the LR test to assess goodness-of-fit, and we establish an Edgeworth expansion of the distribution of the ``na\"ive" LR. This result allows to identify a critical quantity, as a function of $k$, $n$ and the variance of the truncated Laplace EM, which determines a loss in the statistical significance of the ``na\"ive" LR test with respect to the ``true" LR test. Our analysis thus shows the importance of taking the perturbation into account, and presents a first rigorous evidence endorsing the release of negative values when the $(\varepsilon,\delta)$-DP is adopted. 

\subsection{Related literature}
Under the curator framework of $(\varepsilon,\delta)$-DP, some recent works have considered the problem of privacy-protecting tests for goodness-of-fit in frequency tables. Our work is the first to adopt a LR approach for an arbitrary dimension $k\geq2$ and a two-sided alternative hypothesis. The work of \citet{Awa(18)} is closely related to ours, as they consider a truncated EM and adopt a LR approach. However, \citet{Awa(18)} assume $k=2$ and, motivated by the study of uniformly most powerful private tests, they consider pointwise and one-sided alternative hypotheses. The works of \citet{Gab(16)} and \citet{Kif(17)} are also related to our work, as they assume $k\geq2$ and they consider a two-sided alternative hypothesis. However, besides not considering truncated EMs, these works do not adopt a LR approach, introducing private tests through suitable perturbations of Pearson's chi-squared test. Other recent works, though less related to ours, consider a minimax analysis for a class of identity tests in the form $H_{0}:\mathbf{p}=\mathbf{p}_{0}$ against $H_{1}:d_{TV}(\mathbf{p},\mathbf{p}_{0})\geq\rho$, for some choice of $\rho>0$, with $d_{TV}$ being the total variation distance \citep{Cai(17),Ach(18),Ali(18),Cum(18),Can(19),Can(20)}. In general, to the best of our knowledge, our work is the first to make use of the power of the test to quantify the trade-off between confidentiality and utility. This is achieved through the use of non-standard large deviation analysis, which is known to be challenging in a setting such as ours, where the statistical model is discrete, multidimensional and not belonging to the exponential family. In particular, our Bahadur-Rao large deviation expansion relies on the development of a novel (sharp) large deviation principle for sum of i.i.d. random vectors, which is of independent interest.

\subsection{Organization of the paper}
The paper is structured as follows. In Section \ref{sec2} we recall the curator framework of $(\varepsilon,\delta)$-DP in the context of frequency tables, define the class of truncated EMs, and introduce some related terminology and notation. In Section \ref{sec3} we introduce the ``true" model and establish a Bahadur-Rao large deviation expansion for the power of the ``true" LR test, quantifying the trade-off between confidentiality and utility. In Section \ref{sec4} we introduce the ``na\"ive" model and establish an Edgeworth expansion for the distribution of the ``na\"ive" LR, quantifying the loss in the statistical significance of the ``na\"ive" LR test with respect to the ``true" LR test. Section \ref{sec5} contains some concluding remarks and directions for future work. Proofs are deferred to Appendix \ref{Appa} for $k=2$ and to Appendix \ref{Appb} for $k>2$.

%%%%%%%%%%%%%%%%%%%%%%%%%%%%%%%%
%%%%%%%%%%%%%%%%%%%%%%%%%%%%%%%%
%%%%%%%%%%%%%%%%%%%%%%%%%%%%%%%%
%%%%%%%%%%%%%%%%%%%%%%%%%%%%%%%%

\section{The curator framework of $(\varepsilon,\delta)$-DP}\label{sec2}

Before presenting our main results in Section \ref{sec3}, it is helpful to recall the definitions of $(\varepsilon,\delta)$-DP and truncated EM, and to introduce some related terminology and notation \citep{Dwo(06),Dwo(13),Rin(17)}. Under the curator framework of $(\varepsilon,\delta)$-DP for frequency tables, or global $(\varepsilon,\delta)$-DP, the list $(a_{1},\ldots,a_{k})$ is centrally stored and a trusted curator is responsible for its perturbation. This is different from local $(\varepsilon,\delta)$-DP, under which individual data points are perturbed \citep{Dwo(13)}. For notational convenience, let $\mathbf{a}=(a_{1},\ldots,a_{k-1})$ and $a_{k}=n-|\mathbf{a}|$, with $|\mathbf{a}|=\sum_{1\leq i\leq k-1}a_{i}$. We consider a class of (randomized) perturbation mechanisms $\mathcal{M}$ on a universe $\mathcal{A}$ and with range $\mathcal{B}$, and we denote by $\mathcal{B}(\mathbf{a})$ the range of the perturbed list $\mathbf{b}=\mathcal{M}(\mathbf{a})$ such that $\mathcal{B}(\mathbf{a})\subseteq\mathcal{B}$; here, it is assumed that $\mathcal{A}=\mathcal{B}$, that is $\mathbf{b}$ has the same structure as $\mathbf{a}$. As we have recalled in the introduction, a privacy loss occurs when an intruder can learn from the perturbed list $\mathbf{b}$ about an individual contributing to the original list $\mathbf{a}$. To quantify such a privacy loss, it is useful to consider two neighbouring lists $\mathbf{a},\mathbf{a}^{\prime}\in\mathcal{A}$, denoted by $\mathbf{a}\sim \mathbf{a}^{\prime}$, meaning that $\mathbf{a}^{\prime}$ can be obtained from $\mathbf{a}$ by adding or removing exactly one individual. Then, the $(\varepsilon,\delta)$-DP provides a suitable measure on how much can be learnt about any individual by taking the ratio between the likelihood of the perturbed list $\mathbf{b}=\mathcal{M}(\mathbf{a})$ and the likelihood of the neighbouring perturbed list $\mathbf{b}^{\prime}=\mathcal{M}(\mathbf{a}^{\prime})$ \citep{Dwo(06),Dwo(13)}. The LR may be alternatively viewed as a posterior odds ratio, or Bayes factor, from a Bayesian perspective. Placing an upper bound on such a LR motivates the definition of $(\varepsilon,\delta)$-DP, and then leads to the following definition of $(\varepsilon,\delta)$-DP mechanism.

\begin{definition}\label{defdiff_pri}
\citep{Dwo(13)} For every $\varepsilon,\delta\geq0$ we say that a mechanism $\mathcal{M}$ satisfies $(\varepsilon,\delta)$-DP if for all $\mathbf{a}, \mathbf{a}^{\prime}\in\mathcal{A}$ such that $\mathbf{a}\sim \mathbf{a}^{\prime}$ and all $S\subseteq\mathcal{B}$,
\begin{equation}\label{ediff_pri_2}
\mathrm{Pr}[\mathcal{M}(\mathbf{a})\in S]\leq \text{e}^{\varepsilon}\mathrm{Pr}[\mathcal{M}(\mathbf{a}^{\prime})\in S]+\delta.
\end{equation}
\end{definition}

The definition of $\varepsilon$-DP by \citet{Dwo(06)} arises from Definition \ref{defdiff_pri} by setting $\delta=0$. For small values of $\varepsilon$, the $\varepsilon$-DP guarantees that the distribution of the perturbed list $\mathbf{b}$ is not affected by the data of any single individual. This leads to protect individuals' confidentiality agains intruders, in the sense that the data of any single individual is not reflected in the released data $\mathbf{b}$. The definition of $(\varepsilon,\delta)$-DP has been proposed as a relaxation of $\varepsilon$-DP to reduce confidentiality protection in a controlled way, and hence to increase the utility of the released data \citep{Dwo(13)}. According to Definition \ref{defdiff_pri}, the parameter $\delta$ adds flexibility to $\varepsilon$-DP by allowing $\mathbf{b}$ to have a probability $\delta$ of having un undesirable LR with a higher associated disclosure risk. That is, $\delta$ may be interpreted as the probability of data accidentally being leaked, thus suggesting that $\delta$ should be small. An implication of \eqref{ediff_pri_2} is that with probability $\delta$ the data may be released unperturbed, though the $(\varepsilon,\delta)$-DP mechanism  described in this paper never releases the whole unperturbed data set. In general, the choice of $\varepsilon$ and $\delta$ should take into account a balance between confidentiality and utility of the released data. See \citet{Dwo(13)} and \citet{Ste(16)} for a discussion on the choice of $\delta$ in connection with the sample size $n$ and the utility of the data. We refer to \citet{Rin(17)} for a comprehensive discussion of $(\varepsilon,\delta)$-DP, as well as of other relaxations of $\varepsilon$-DP, in the context of frequency tables.

As a perturbation mechanism $\mathcal{M}$, in this paper we consider a truncated version of the Laplace EM, which is arguably the most popular EM \citep{Rin(17)}. However, our results can be easily extended the broad class of truncated convolutional-type EMs. For $\mathbf{a}\in\mathcal{A}$ and $\mathbf{b}\in\mathcal{B}(\mathbf{a})$ consider an additive utility function of the form $u(\mathbf{a},\mathbf{b})=\sum_{1\leq i\leq k}v(a_{i},b_{i})$ for some function $v$, which enables us to specify a mechanism that perturbs the cells of a list $\mathbf{a}$ independently, and impose that $|a_{i}-b_{i}|\leq m$ for a truncation level $m\in\mathbb{N}_{0}$, for any $i=1,\ldots,k-1$. If $w(\mathbf{b};\mathbf{a})$ is the conditional probability that the list $\mathbf{a}$ is perturbed to the list $\mathbf{b}$, then a truncated EM is defined as follows
\begin{equation}\label{exp_mec}
w(\cdot;\,\mathbf{a})\propto\exp\left\{\eta\frac{u(\mathbf{a},\,\cdot)}{\Delta u(\mathbf{a})}\right\},
\end{equation}
 where $\eta$ is a value that depends on $\varepsilon$, and $\Delta u$ is defined as $\Delta u(\mathbf{a})= \max_{\mathbf{b}\in\mathcal{B}(\mathbf{a}^{\prime})}\max_{\mathbf{a}\sim \mathbf{a}^{\prime}\in\mathcal{A}} |u(\mathbf{a},\mathbf{b})-u(\mathbf{a}^{\prime},\mathbf{b})|$. According to \eqref{exp_mec}, a truncated EM attaches higher probability to perturbed lists with higher utility. Truncated convolutional-type EMs correspond to the choice $v(a_{i},b_{i})=g(a_{i}-b_{i})$, for some function $g$. In particular, the truncated Laplace EM is a truncated convolutional-type EM for the choice $v(a_{i},b_{i})=-|a_{i}-b_{i}|$. The next theorem states that truncated convolutional-type truncated EMs are $(\varepsilon,\delta)$-DP mechanism, with $\delta>0$ depending on $m$ and the utility function. We refer \citet[Section 4]{Rin(17)} for details on the calculation of $\delta$ as a function of $\varepsilon$ and $m$.
  
 \begin{theorem}\label{teo22}
\citep{Rin(17)} Let $u$ be an utility function of the form $g(\mathbf{a}-\mathbf{b})$ for some function $g$, and let $\mathcal{M}$ be a mechanism such that $\text{Pr}[\mathcal{M}(\mathbf{a})=\mathbf{b}]\propto\exp\{\varepsilon u(\mathbf{a},\mathbf{b})/\Delta u(\mathbf{a})\}$ for all lists $\mathbf{a}\in\mathcal{A}$ and all $\mathbf{b}\in\mathcal{B}(\mathbf{a})$ such that $|a_{i}-b_{i}|\leq m\leq+\infty$ for any $i=1,\ldots,k-1$. Assume that for all $\mathbf{a},\mathbf{a}^{\prime}\in\mathcal{A}$ such that $\mathbf{a}\sim \mathbf{a}^{\prime}$ it holds that $\text{Pr}[\mathcal{M}(\mathbf{a}^{\prime})=\mathbf{b}]=0$ implies $\text{Pr}[\mathcal{M}(\mathbf{a})=\mathbf{b}]<\delta$. Then the mechanism $\mathcal{M}$ is $(\varepsilon,\delta)$-DP, with $\delta=0$ when $m=+\infty$.
\end{theorem}

%%%%%%%%%%%%%%%%%%%%%%%%%%%%%%%%
%%%%%%%%%%%%%%%%%%%%%%%%%%%%%%%%
%%%%%%%%%%%%%%%%%%%%%%%%%%%%%%%%
%%%%%%%%%%%%%%%%%%%%%%%%%%%%%%%%

\section{LR tests for goodness-of-fit under $(\varepsilon,\delta)$-DP}\label{sec3}

For any $n,k\in\mathbb{N}$ with $k< n$, let $(a_{1},\ldots,a_{k})$ be the list to be released, such that $n=\sum_{1\leq i\leq k}a_{i}$. We consider the Multinomial model for $\mathbf{a}$, that is the list $\mathbf{a}$ is assumed to be the realization of the random variable $\mathbf{A}_{n}=(A_{1,n},\ldots,A_{k-1,n})$ distributed as a Multinomial distribution with parameter $n\in\mathbb{N}$ and $\mathbf{p}=(p_{1},\ldots,p_{k-1})\in\Delta_{k-1}=\{\mathbf{p}\in[0,1]^{k-1}\text{ : }|\mathbf{p}|\leq 1\}$, where $|\mathbf{p}|=\sum_{1\leq i\leq k-1}p_{i}$. Then,
\begin{displaymath}
\text{Pr}[\mathbf{A}_{n}=\mathbf{a}]=\frac{n!}{(n-|\mathbf{a}|)!}\left(\prod_{i=1}^{k-1}\frac{p_{i}^{a_{i}}}{a_{i}!}\right)\left(1-|\mathbf{p}|\right)^{n-|\mathbf{a}|}\mathbbm{1}_{\mathcal{P}_{n,k-1}}(\mathbf{a}),
\end{displaymath}
where $\mathcal{P}_{n,k-1}:=\{\mathbf{a}\in\mathbb{N}_{0}^{k-1}\text{ : }|\mathbf{a}|\leq n\}$. We assume that $\mathbf{a}$ is perturbed by means of the truncated Laplace EM. More precisely, for $\varepsilon\geq0$ and $m\in\mathbb{N}_{0}$ the cells of $\mathbf{a}$ are perturbed independently through the conditional distribution
\begin{equation}\label{lap}
w(b_{i};a_{i})=\frac{1}{c_{\varepsilon,m}}\text{e}^{-\varepsilon|b_{i}-a_{i}|}\mathbbm{1}_{\{-m,\ldots,0,\ldots,m\}}(|b_{i}-a_{i}|)
\end{equation}
for $i=1,\ldots,k-1$, where $c_{\varepsilon,m}=\sum_{-m\leq l\leq m}\exp\{-\varepsilon|l|\}$, and $b_{k}=n-|\mathbf{b}|$. The resulting perturbed list $\mathbf{b}$ has the same sample size $n$ as the original list $\mathbf{a}$, i.e. $n=\sum_{1\leq i\leq k}a_{i}=\sum_{1\leq i\leq k}b_{i}$. According to Theorem \ref{teo22}, the truncated Laplace EM is $(\varepsilon,\delta)$-DP with respect to the utility function $u(\mathbf{a},\mathbf{b})=-\sum_{1\leq i\leq k-1}|a_{i}-b_{i}|$, with $\delta=c_{\varepsilon,m}^{-1}\exp\{-\varepsilon m\}$ \citep[Section 5]{Rin(17)}. Throughout this paper, we assume $m<n$, thus excluding the case $m=+\infty$ that corresponds to $\delta=0$, i.e. the $\varepsilon$-DP. See Lemma \ref{lemma_likelihood} for details. 

The ``true" or natural model for $\mathbf{b}$ is defined as a statistical model that takes into account the truncated Laplace EM. Because of the form of the distribution \eqref{lap}, the truncated Laplace EM may be viewed as adding, independently for each cell of $\mathbf{a}$, random variables that are i.i.d. as a truncated (discrete) Laplace distribution. In particular, let $\mathbf{L}=(L_{1},\ldots,L_{k-1})$ be a random variable independent of $\mathbf{A}_{n}$, and such that the $L_{i}$'s are i.i.d. according to
\begin{equation} \label{Laplace}
\text{Pr}[L=l]=\frac{1}{c_{\varepsilon,m}}\text{e}^{-\varepsilon|l|}\mathbbm{1}_{\{-m,\ldots,0,\ldots,m\}}(l),
\end{equation}
i.e. the $m$-truncated Laplace distribution with location $0$ and scale $\varepsilon^{-1}$. Then, the ``true" model assumes that $\mathbf{b}$ is the realization of a random variable whose distribution is the convolution between the distributions of $\mathbf{A}_{n}$ and $\mathbf{L}$, i.e. the distribution of $\mathbf{B}_{n}=(B_{1,n},\ldots,B_{k-1,n})$ with $B_{i,n}=A_{i,n}+L_{i}$ for $i=1,\ldots,k-1$, and $B_{k,n}=n-|\mathbf{B}_{n}|$. The corresponding likelihood function is
\begin{equation}\label{true_mod}
L_{n,T}(\mathbf{p};\mathbf{b})\propto\left(\prod_{i=1}^{k-1}\sum_{l_{i}=-m\vee(b_{i}-n)}^{m\wedge b_{i}}\frac{p_{i}^{b_{i}-l_{i}}\text{e}^{-\varepsilon|l_{i}|}}{(b_{i}-l_{i})!}\right)(1-|\mathbf{p}|)^{n-|\mathbf{b}-\mathbf{l}|}.
\end{equation}
Under the ``true" model, we make use of the LR test to assess goodness-of-fit in the form: $H_{0}\text{ : }\mathbf{p}=\mathbf{p}_{0}$, for a fixed $\mathbf{p}_{0}\in\Delta_{k-1}$, against $H_{1}\text{ : }\mathbf{p}\neq\mathbf{p}_{0}$. Without loss of generality, we assume that $\mathbf{p}_{0}$ belongs to the interior of $\Delta_{k-1}$. We present a careful large sample analysis, in terms of Edgeworth expansions, of the distribution of the LR and of the power of the corresponding LR test.  Such an analysis is new for likelihood functions in the convolutional form \eqref{true_mod}, and it is definitely a challenging task due to the fact that the ``true" model is discrete, multidimensional and not belonging to the exponential family.

While we focus on the popular truncated Laplace EM, our analysis and results can be easily extended to any truncated convolutional-type EMs. An example is the truncated Gaussian EM, such that for any $\varepsilon\geq0$ and $m\in\mathbb{N}_{0}$ the cells of the list $\mathbf{a}$ are perturbed independently through the conditional distribution
\begin{displaymath}
w(b_{i};a_{i})=\frac{1}{d_{\varepsilon,m}}\text{e}^{-\frac{\varepsilon}{2m+1}(b_{i}-a_{i})^{2}}\mathbbm{1}_{\{-m,\ldots,0,\ldots,m\}}(|b_{i}-a_{i}|)
\end{displaymath}
for $i=1,\ldots,k-1$, where $d_{\varepsilon,m}=\sum_{-m\leq l\leq m}\exp\{-\varepsilon l^{2}/(2m+1)\}$, and $b_{k}=n-|\mathbf{b}|$. According to Theorem \ref{teo22}, the truncated Gaussian EM is $(\varepsilon,\delta)$-DP with respect to the utility function $u(\mathbf{a},\mathbf{b})=-\sum_{1\leq i\leq k-1}(a_{i}-b_{i})^{2}$, with $\delta=d_{\varepsilon,m}^{-1}\exp\{-\varepsilon m^{2}/(2m+1)\}$. The truncated Gaussian EM corresponds to add, independently for each cell of $\mathbf{a}$, random variables that are i.i.d. as the $m$-truncated (discrete) Gaussian distribution with location $0$ and (squared) scale $(2m+1)/2\varepsilon$, i.e.
\begin{equation}\label{gauss_n}
\text{Pr}[G=g]=\frac{1}{d_{\varepsilon,m}}\text{e}^{-\frac{\varepsilon}{2m+1}g^{2}}\mathbbm{1}_{\{-m,\ldots,0,\ldots,m\}}(g).
\end{equation}
The resulting ``true" model for $\mathbf{b}$ has a likelihood function in a convolutional form similar to \eqref{true_mod}, and therefore our analysis and results under the truncated Laplace EM can be easily adapted to the truncated Gaussian EM. We refer to \citet{Rin(17)} for other examples of truncated convolutional-type EMs.

\subsection{The ``true" LR test, and a ``private" Wilks theorem}

For any parameter $\mathbf{p}\in\Delta_{k-1}$, we denote by $\hat{\mathbf{p}}_{n,T}$ be the maximum likelihood estimator of $\mathbf{p}$ under the ``true" model for the perturbed list $\mathbf{b}$. Such an estimator is not available in a closed-form expression. Then, we define the ``true" LR as follows
\begin{equation} \label{LR_multidim}
\Lambda_{n,T}(\mathbf{p}_{0})=2\log\left(\frac{L_{n,T}(\hat{\mathbf{p}}_{n,T};\mathbf{B}_{n})}{L_{n,T}(\mathbf{p}_{0};\mathbf{B}_{n})}\right),
\end{equation}
and denote by $\text{Pr}[\Lambda_{n,T}(\mathbf{p}_{0})\in\cdot;\,\mathbf{p}_{0}]$ the distribution of $\Lambda_{n,T}(\mathbf{p}_{0})$ computed with respect to $\mathbf{A}_{n}$ distributed as a Multinomial distribution with parameter $(n,\mathbf{p}_{0})$. For any $\mathbf{p}_{0}$ in the interior of $\Delta_{k-1}$, let $\Lambda_{n}(\mathbf{p}_{0})$ be the Multinomial LR, namely the LR in the absence of perturbation, which is obtained from $\Lambda_{n,T}(\mathbf{p}_{0})$ by setting $m=0$ and/or $\varepsilon\rightarrow+\infty$. If $\F_{n}(t)=\mathrm{Pr}[\Lambda_{n}(\mathbf{p}_{0})\leq t;\,\mathbf{p}_{0}]$ and $\K_{k}$ is the cumulative distribution function of a chi-squared distribution with $k\geq1$ degrees of freedom, then a classical result by \citet{Wil(38)} shows that for any $t>0$
\begin{equation}\label{class_wil}
\lim_{n\rightarrow+\infty}\F_{n}(t)=\K_{k-1}(t).
\end{equation}
See also \citet[Chapter 4]{Fer(02)}, and references therein, for a detailed account on Wilks' theorem and generalizations thereof. The next theorem establishes an Edgeworth expansion, with respect to  $n\rightarrow+\infty$, of the distribution of $\Lambda_{n,T}(\mathbf{p}_{0})$. Such a result provides a ``private" and refined version of \eqref{class_wil}.

\begin{theorem}\label{teo32}
Let $\F_{n, T}(t)=\mathrm{Pr}[\Lambda_{n,T}(\mathbf{p}_{0})\leq t;\,\mathbf{p}_{0}]$ be the cumulative distribution function of $\Lambda_{n,T}(\mathbf{p}_{0})$. For any $n,k\in\mathbb{N}$ and any $t>0$ it holds true that
\begin{equation} \label{BE_LRTk}
\F_{n,T}(t)= \K_{k-1}(t)+\frac{c_{1,k,\mathbf{p}_{0}}(t)}{n^{1/2}} +\frac{c_{\ast,k,\mathbf{p}_{0}}(t;m,\varepsilon)}{e^{t/2}\sqrt{2\pi n}} 
+\frac{c_{2,k,\mathbf{p}_{0}}(t)}{n} + R_{(k,n),\mathbf{p}_{0},T}(t),
\end{equation}
with 
\begin{displaymath}
|R_{(k,n),\mathbf{p}_{0},T}(t)| \leq \frac{C_{k,\mathbf{p}_{0},T}(t; m, \varepsilon)}{n^{3/2}},
\end{displaymath}
where the functions $c_{1,k,\mathbf{p}_{0}}$ and $c_{2,k,\mathbf{p}_{0}}$ are independent of $n$ and of the distribution of $\mathbf{L}$, i.e. independent of $\varepsilon$ and $m$, the function $c_{\ast,k,\mathbf{p}_{0}}$ is independent of $n$ and such that $|c_{\ast,k,\mathbf{p}_{0}}| \leq (k-1)$, and the function $C_{k,\mathbf{p}_{0},T}$ is independent of $n$.
\end{theorem}

See Appendix \ref{Appa} and Appendix \ref{Appb} for the proof of Theorem \ref{teo32} with $k=2$ and $k>2$, respectively. The functions $c_{1,k,\mathbf{p}_{0}}$, $c_{\ast,k,\mathbf{p}_{0}}$ and $c_{2,k,\mathbf{p}_{0}}$ in Theorem \ref{teo32}, as well as $C_{k,\mathbf{p}_{0},T}$, can be made explicit by gathering some equations in the proof. In particular,  $c_{1,k,\mathbf{p}_{0}}$ and $c_{2,k,\mathbf{p}_{0}}$ are the same that would appear in the Edgeworth expansion of the distribution of $\Lambda_{n}(\mathbf{p}_{0})$, which ensues as a corollary of \eqref{BE_LRTk} by setting $m=0$ and/or $\varepsilon \to +\infty$. For a fixed (reference) level of significance $\alpha\in(0,1)$, the ``true" LR test to assess goodness-of-fit has a rejection region $\{\Lambda_{n,T}>\lambda_{T}(\alpha)\}$, with the critical point $\lambda_{T}(\alpha)$ being determined in such a way that $\F_{n,T}(\lambda_{T}(\alpha))=1-\alpha$. According to Theorem \ref{teo32}, if $n$ increases then the distribution $\F_{n,T}$ of $\Lambda_{n,T}$ becomes close to $\K_{k-1}$, and for a fixed $n\geq1$ the parameters $\varepsilon$ and $m$ interact with $k$ and $n$ to control such a closeness. For a fixed $n\geq1$, Theorem \ref{teo32} allows us to determine $\lambda_{T}(\alpha)$ as a function of $\varepsilon$ and $m$, thus quantifying the contribution of the truncated Laplace EM, in terms $\varepsilon$ and $m$, to the rejection region of the ``true" LR test. From \eqref{BE_LRTk},  the truncated Laplace EM gives its most significant contribution in the term of order $n^{-3/2}$, whereas in the limit its contribution vanishes. Hence, we write
\begin{displaymath}
\lambda_{T}(\alpha) = \lambda(\alpha)+\frac{\tilde{c}_{1,k,\mathbf{p}_{0}}(\alpha)}{n^{1/2}} +\frac{\tilde{c}_{2,k,\mathbf{p}_{0}}(\alpha)}{n}+
\frac{\tilde{c}_{3,k,\mathbf{p}_{0}}(n;\alpha;\varepsilon,m)}{n^{3/2}},
\end{displaymath}
where $\lambda$, $\tilde{c}_{1,k,\mathbf{p}_{0}}$ and $\tilde{c}_{2,k,\mathbf{p}_{0}}$ are positive constants independent of $n, \varepsilon$ and $m$, while $\tilde{c}_{3,k,\mathbf{p}_{0}}(n;\alpha;\varepsilon,m)$ is, for fixed $\alpha, \varepsilon$ and $m$, a bounded function of $n$. Along the same lines of the proof of Theorem \ref{teo32}, it is easy to show that an analogous result holds true for the truncated Gaussian EM, and in general  for any truncated convolutional-type EMs. Theorem \ref{teo32} provides a preliminary result to the study of the power of the ``true" LR.  A further discussion of Theorem \ref{teo32} is deferred to Section \ref{sec4}, with respect to a comparison between the rejection regions defined under the ``true" model and the ``na\"ive" model, the latter being a statistical model that does not take into account the truncated Laplace EM.  

\subsection{The power of the ``true" LR test, and a ``private" Bahadur-Rao theorem} \label{sec3.2}

For a fixed $\alpha\in(0,1)$, Theorem \ref{teo32} provides the critical point $\lambda_{T}(\alpha)$ of the ``true" LR test. Then assuming $\mathbf{p}_{0}$ and $\mathbf p_{1}$ in the interior of $\Delta_{k-1}$, and such that $\mathbf p_{1} \neq \mathbf p_{0}$, the power of the ``true" LR test with respect to $\mathbf p_{1}$ 
is defined as
\begin{displaymath}
\beta_{n,T}(\mathbf p_{1};\alpha)=\text{Pr}[\Lambda_{n,T}(\mathbf{p}_{0})>\lambda_{T}(\alpha);\,\mathbf{p}_{1}],
\end{displaymath}
with $\text{Pr}[\Lambda_{n,T}(\mathbf{p}_{0})\in\cdot;\,\mathbf{p}_{1}]$ being computed with respect to $\mathbf{A}_{n}$ distributed as a Multinomial distribution with parameter $(n,\mathbf{p}_{1})$. The Kullback-Leibler divergence between two Multinomial distributions of parameters $(1, \mathbf p)$ and $(1, \mathbf q)$ is
\begin{displaymath}
\mathcal D_{KL}(\mathbf p\,\|\, \mathbf q) =\E_{\mathbf p}\left[\log\left(\prod_{i=1}^{k-1}\left(\frac{p_{i}}{q_{i}}\right)^{M_{i}}\right)+\log\left(\frac{p_{k}}{q_{k}} \right)^{M_{k}}\right] 
= \sum_{i=1}^k p_{i} \log\frac{p_{i}}{q_{i}},
\end{displaymath}
where $p_{k} = 1-|\mathbf{p}|$, $q_{k} = 1-|\mathbf{q}|$ and $M_{k}=1-\sum_{1\leq i\leq k-1}M_{i}$. We denote by $\beta_{n}(\mathbf{p}_{1};\alpha)$ the power of the Multinomial LR test with rejection region $\{\Lambda_{n}(\mathbf p_{0})>\lambda(\alpha)\}$. A classical result from \cite{Rao(62)} and \citet{Bah(67)} shows that
\begin{equation}\label{as_unp}
\lim_{n\rightarrow+\infty} -\frac{1}{n}\log(1-\beta_n(\mathbf p_{1};\alpha)) = \mathcal D_{KL}(\mathbf p_0\,\|\, \mathbf p_1).
\end{equation}
See \citet{Hoe(65),Hoe(67)} for refinements of \eqref{as_unp} in terms of Bahadur-Rao large deviation expansions of $-n^{-1}\log(1-\beta_n(\mathbf p_{1};\alpha))$, which introduce the dependence on $\alpha$ in the right-hand side of \eqref{as_unp} \citep{Efr(67),Efr(68)}.

We show that the large sample behaviour \eqref{as_unp} holds true under the ``true" model for $\mathbf{b}$. Precisely, for a fixed $\mathbf{p}_{1}\in \Delta_{k-1}$ such that $\mathbf p_{1} \neq \mathbf p_{0}$ we show that
\begin{equation} \label{lim_powe}
\lim_{n\rightarrow+\infty} -\frac{1}{n}\log(1-\beta_{n,T}(\mathbf p_{1};\alpha)) = \mathcal D_{KL}(\mathbf p_0\,\|\, \mathbf p_1). 
\end{equation}
Note that Equation \eqref{lim_powe} does not show explicitly the contribution of the truncated Laplace EM to the power $\beta_{n,T}(\mathbf p_{1};\alpha))$ of the ``true" LR test, hiding the parameters $\varepsilon$ and $m$, as well as it does not show any contribution of the (reference) level of significance $\alpha$. In other terms, the right-hand side of \eqref{lim_powe} does not depend on $\varepsilon$, $m$ and $\alpha$. To bring out the contribution of the truncated Laplace EM to $\beta_{n,T}(\mathbf p_{1};\alpha))$, we introduce a refinement of \eqref{lim_powe} in the sense of \citet{Hoe(65),Hoe(67)}. In particular, in the next theorem we rely on non-standard large deviation analysis \citep{Pet(75),SS(91),vBa(67)} in order to establish a Bahadur-Rao large deviation expansion of $n^{-1}\log(1-\beta_{n,T}(\mathbf p_{1};\alpha))$. Our result thus provides a ``private" and refined version of \eqref{as_unp}. This is a critical tool, as it leads to a quantitative characterization of the tradeoff between confidentiality, measured via $\varepsilon$ and $m$, and utility of the data, measured via the power of the ``true" LR test. As a corollary, the next theorem allows us to define the (sample) cost of $(\varepsilon,\delta)$-DP for the ``true" LR test, that is the additional sample size that is required in order to recover the power of the Multinomial LR test.

\begin{theorem}\label{teo33}
For a fixed (reference) level of significance $\alpha$, let $\lambda(\alpha)$ be the critical point such that $\K_{k-1}(\lambda(\alpha))=1-\alpha$. For any $k\in\mathbb{N}$, it holds true that
\begin{align}\label{LD1}
&-\frac{1}{n}\log(1-\beta_{n,T}(\mathbf p_{1};\alpha))\\
&\notag\quad= \mathcal D_{KL}(\mathbf p_0\,\|\, \mathbf p_{1}) + \frac{c_{1,k,(\mathbf p_0, \mathbf p_{1})}(\lambda(\alpha))}{\sqrt{n}} + \frac{k}{4}\left(\frac{\log n}{n}\right) +
\frac{c_{2,k,(\mathbf p_0, \mathbf p_{1})}(\lambda(\alpha))}{n} \nonumber \\
&\notag\quad\quad-\frac{1}{n}\log \mathfrak M_L\left(\nabla_{\mathbf p_0}\mathcal D_{KL}(\mathbf p_0\,\|\, \mathbf p_{1})\right)
+o\left(\frac{1}{n}\right)
\end{align}
as $n \to +\infty$, where $\mathfrak M_L(\zb)= \E[\exp\{\mathbf L \zb\}] \geq 1$, i.e. the moment generating function of $\mathbf{L}$ evaluated in $\mathbf{z}$, and $c_{1,k,(\mathbf p_0, \mathbf p_{1})}$ and $c_{2,k,(\mathbf p_0, \mathbf p_{1})}$ are constants independent of $n$.
\end{theorem}

See Appendix \ref{Appa} and Appendix \ref{Appb} for the proof of Theorem \ref{teo33} with $k=2$ and $k>2$, respectively. The constants $c_{1,k,(\mathbf p_0, \mathbf p_{1})}$ and $c_{2,k,(\mathbf p_0, \mathbf p_{1})}$ in Theorem \ref{teo33} can be made explicit by gathering equations in the proof. The proof of \eqref{LD1} relies on a novel (sharp) large deviation principle for sum of i.i.d. random vectors, which is of independent interest. In particular, consider a $C^2$-regular compact and convex subset $\mathscr E \subset \R^d$, such that there exists a function $h \in C^2(\R^d)$ for which: i) $\mathscr E = \{\xb \in \R^d\ |\ h(\xb) \leq 0\}$; ii) $\partial\mathscr E = \{ \xb \in \R^d\ |\ h(\xb) = 0\}$ is a $C^2$-hypersurface; iii) $\nabla h$ does not vanish on $\partial\mathscr E$; iv) the hypersurface $\partial\mathscr E$ is oriented by the normal field $|\nabla h(\xb)|^{-1} \nabla h(\xb)$. Let $\{\mathbf X_i\}_{i \geq 1}$ be a sequence of i.i.d. $d$-dimensional random variables such that $\E[e^{\mathbf{z}\mathbf X_1}] < +\infty$ for all $\mathbf z$ with $|\mathbf z| < H$ for some $H >0$, and such that $\E[\mathbf X_1] = \mathbf 0$. Set $L(\mathbf z) := \log \E[e^{\mathbf z\mathbf X_1}]$. If, 
for any given a vector $\boldsymbol \xi \in \R^d$, there exists a unique $\hat{\mathbf z} \in \R^d$ with $|\hat{\mathbf z}| < H$ for which $\nabla L(\hat{\mathbf z}) = \boldsymbol \xi$, 
then it holds true that
\begin{align}\label{main_ld}
&\mathrm{Pr} \left[\frac{1}{\sqrt{n}} \sum_{i=1}^n \mathbf X_i \in \sqrt{n} \boldsymbol \xi + \mathscr E\right]\\
&\notag= \exp\left\{-n[\hat{\mathbf z} \nabla L(\hat{\mathbf z}) - L(\hat{\mathbf z})] + \sqrt{n} 
\min_{\mathbf v \in \mathscr E} \hat{\mathbf z} \mathbf v\right\} n^{-(d+1)/4}
[1 + o(1)],
\end{align}
as $n \rightarrow +\infty$, where $d$ denotes the dimension of the $\mathbf X_i$'s. See Appendix B for the proof of \eqref{main_ld}. Analogous large deviation principles, though not useful in our specific context, are in \cite{Ale(83)}, \cite{Osi(82)}, \cite{Sau(83)}, \cite{vBa(67)}. See also \citet{SS(91)} and references therein.

\begin{remark}
Under the Multinomial LR test, Theorem \ref{teo33} provides a refinement of \eqref{as_unp}. This is obtained from \eqref{LD1} by setting $\mathbf{L}=(0,\ldots,0)$, which implies
\begin{displaymath}
\log \mathfrak M_L\left(\nabla_{\mathbf p_0}\mathcal D_{KL}(\mathbf p_0\,\|\, \mathbf p_{1})\right)=0.
\end{displaymath}
The resulting expression is a novel Bahadur-Rao large deviation expansion of $n^{-1}\log(1-\beta_{n}(\mathbf p_{1};\alpha))$, which improves on the results of \citet{Hoe(65),Hoe(67)}.
\end{remark}

Theorem \ref{teo33} shows that, as $n$ increases, $-n^{-1}\log(1-\beta_{n,T}(\mathbf p_{1};\alpha))$ becomes close to $\mathcal D_{KL}(\mathbf p_{0}\,\|\, \mathbf p_{1})$, and, for a fixed $n\geq1$, the parameters $\varepsilon$ and $m$ interact with $k$ and $n$ in order to control such a closeness. Empirical analyses in \citet{Rin(17)} show that, for a fixed $n\geq1$ and $\alpha\in(0,1)$, the power $\beta_{n,T}(\mathbf p_{1};\alpha))$ decreases as $\varepsilon$ decreases and/or $m$ increases; such a behaviour agrees with intuition, as decreasing $\varepsilon$ and/or increasing $m$ leads to increase the perturbation in the data. Theorem \ref{teo33} provides a theoretical guarantee to the analyses in \citet{Rin(17)} by quantifying the contribution of the truncated Laplace EM, in terms of $\varepsilon$ and $m$, to the power of the ``true" LR test. According to \eqref{LD1}, the truncated Laplace EM  gives its most significant contribution in the term of order $n^{-1}$, whereas in the limit its contribution vanishes. That is,
\begin{displaymath}
\mathfrak L_{(\mathbf p_0, \mathbf p_{1})}(\ep,m) = \log \mathfrak M_L\left(\nabla_{\mathbf p_0}\mathcal D_{KL}(\mathbf p_0\,\|\, \mathbf p_{1})\right)
\end{displaymath}
determines the loss in the power of the ``true" LR test. Since $\mathfrak L_{(\mathbf p_0, \mathbf p_{1})}(\ep,m)$ increases as $\varepsilon$ decreases and/or $m$ increases, then the power of the  ``true" LR test decreases  under such a behavior of $\varepsilon$ and $m$. See Figure \ref{fig1} for an illustration of such a behaviour in the problem of testing a Uniform (U) $\mathbf{p}_{0}$, with $k=2$ and $k=4$, versus some alternatives for $\mathbf{p}_{1}$'s. Along the same lines of the proof of Theorem \ref{teo33}, it is easy to show that an analogous result holds true for the truncated Gaussian EM, with $\mathbf{L}$ being replaced by $\mathbf{G}=(G_{1},\ldots,G_{k-1})$ such that the $G_{i}$'s are random variables i.i.d. according to \eqref{gauss_n}. For fixed values of $\varepsilon$ and $\delta$, that is fixed levels of $(\varepsilon,\delta)$-DP, for the same tests considered in Figure \ref{fig1}, Table \ref{tab1} shows that the truncated Laplace EM produces a smaller decrease in the power than the truncated Gaussian EM.

\begin{figure}[!htb]
    \centering
    \includegraphics[width=0.95\linewidth]{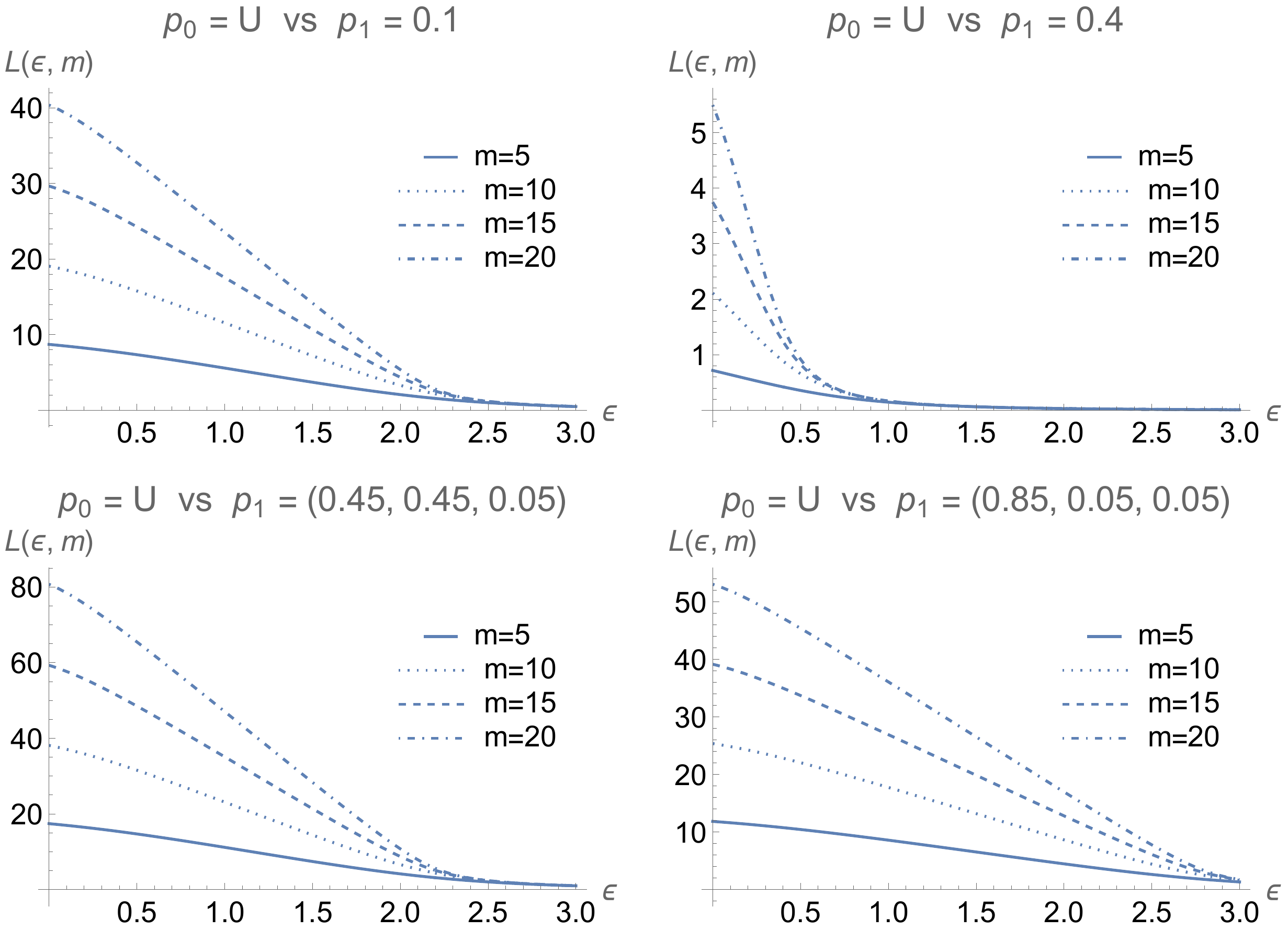}
    \caption{\small{Values of $\mathfrak L_{(\mathbf p_0, \mathbf p_{1})}(\ep,m)$ as a function of $\varepsilon\geq0$ for $m=5,\,10,\,15,\,20$: $H_{0}\text{ : }p=0.5$ vs $H_{1}\text{ : }p=0.1$ (top-left panel); $H_{0}\text{ : }p=0.5$ vs $H_{1}\text{ : }p=0.4$ (top-right panel); $H_{0}\text{ : }\mathbf{p}=(0.25,\,0,25,\,0.25)$ vs $H_{1}\text{ : }\mathbf{p}=(0.45,\,0.45,\,0.05)$ (bottom-left panel); $H_{0}\text{ : }\mathbf{p}=(0.25,\,0.25,\,0.25)$ vs $H_{1}\text{ : }\mathbf{p}=(0.85,\,0.05,\,0.05)$ (bottom-right panel).}}
    \label{fig1}
\end{figure}

\begin{remark}
Theorem \ref{teo33} shows that the contribution of the truncated Laplace EM to the power of the ``true" LR test vanishes in the limit $n\rightarrow+\infty$, leading to \eqref{lim_powe}. In particular, given a constant $C>0$, a value of $(\varepsilon,m)$ that solve
\begin{displaymath}
\mathfrak L_{(\mathbf p_0, \mathbf p_{1})}(\ep,m) = n C,
\end{displaymath}
say $(\varepsilon_{n},m_{n})$, provides the correct scaling for the parameters $\varepsilon$ and $m$, with respect to the sample size $n$, so that the contribution of the truncate Laplace EM to the power of the ``true" LR test does no longer vanish in the large $n$ limit.
\end{remark}

\begin{table}[!htb]
    \centering
    \resizebox{1\textwidth}{!}{\begin{tabular}{lrrrrrrrrrr}
        \toprule
        \multicolumn{1}{l}{} & \multicolumn{2}{c}{$m$}&\multicolumn{2}{c}{$\mathfrak L_{(\text{U}, 0.1)}(\ep,m) $} & \multicolumn{2}{c}{$\mathfrak L_{(\text{U},0.4)}(\ep,m) $} & \multicolumn{2}{c}{$\mathfrak L_{(\text{U},(0.45,\,0.45,\,0.05)}(\ep,m) $} & \multicolumn{2}{c}{$\mathfrak L_{(\text{U},(0.85,\,0.05,\,0.05))}(\ep,m) $} \\
        \cmidrule(r){2-3} \cmidrule(r){4-5} \cmidrule(r) {6-7} \cmidrule(r){8-9} \cmidrule(r){10-11}
        $(\varepsilon,\,\delta)$ & L & G & L & G & L & G& L & G & L & G \\[0.05cm]
        \midrule
            (0.025,\,0.09) & 5&   5& 8.652&   8.674 &  0.698 &   0.707 &  17.303 & 17.349  &   11.773 &   11.796  \\[0.1cm]
             (0.025,\,0.04) & 10&   10& 18.927 &   18.972 &  2.037 &   2.069 &  37.854 &  37.944  &   25.228 &   25.274  \\[0.1cm]
            (0.05,\,0.08) & 5&   5& 8.596 &  8.642 &   0.6789 &   0.697&   17.191 &   17.284  &  11.715 &  11.762 \\[0.1cm]
             (0.05,\,0.04) & 10&   10& 18.802 &  18.899 &  1.962 &  2.029&   37.605 &   37.795  &  25.101 &  25.197 \\[0.1cm]
            (0.075,\,0.08) &5&   5& 8.538 &  8.610 &   0.660 &  0.687 &   17.076 &   17.219  &  11.656 & 11.729  \\[0.1cm]
              (0.075,\,0.03) & 10&   11& 18.672 &  20.902 &   1.885 &  2.281 &   37.345 &   41.803  &  24.970 & 27.836  \\[0.1cm]
           (0.1,\,0.07) & 5&   6& 8.479 &  10.573 &  0.641 & 0.903 &   16.958 &   21.147  &  11.595 & 14.327  \\[0.1cm]
             (0.1,\,0.03) & 10&   12& 18.536 &  22.893 &  1.807 &  2.525 &   37.073 &   45.786  &  24.833 & 30.462  \\[0.1cm]
           (0.125,\,0.07) & 5&   6& 8.419 &  10.531 &   0.622 &  0.888 &   16.837 &   21.063  &  11.533 & 14.283  \\[0.1cm]
            (0.125,\,0.02) & 10&   12& 18.396 &  22.795 &   1.727 &  2.470 &   36.792 &   45.591  &  24.690 & 30.362  \\[0.1cm]
           (0.150,\,0.06) & 5&   6& 8.357 &  10.489 &  0.603 &  0.874 &   16.713 &   20.978  &  11.469 & 14.240  \\[0.1cm]
            (0.150,\,0.02) & 10&   12& 18.250 &  22.696 &  1.646 &  2.414 &   36.499 &   45.391  &  24.542 & 30.261  \\[0.1cm]
           (0.175,\,0.06) & 5&   6& 8.293 &  10.446 &   0.584 &  0.859 &   16.587 &   20.892  &  11.404 & 14.195  \\[0.1cm]
           (0.175,\,0.02)& 10&   12 & 18.098 &  22.595 &   1.565 &  2.358 &   36.197 &   45.189  &  24.389 & 30.158  \\[0.1cm]
           (0.2,\,0.05) & 5&   6& 8.228 &  10.403 &  0.565 &  0.845 &   16.457 &   20.801  &  11.337 & 14.150  \\[0.1cm]
            (0.2,\,0.02) & 10&   13& 17.942 &  24.536 &  1.483 &  2.566 &   35.884 &   49.072  &  24.231 & 32.734  \\[0.1cm]
          \bottomrule
    \end{tabular}}
    \caption{\small{Values of $\mathfrak L_{(\mathbf p_0, \mathbf p_{1})}(\ep,m)$ for fixed levels of $(\varepsilon,\delta)$-DP under the truncated Laplace EM (L) and the truncated Gaussian EM (G): $H_{0}\text{ : }p=0.5$ vs $H_{1}\text{ : }p=0.1$ (third column); $H_{0}\text{ : }p=0.5$ vs $H_{1}\text{ : }p=0.4$ (fourth column); $H_{0}\text{ : }\mathbf{p}=(0.25,\,0,25,\,0.25)$ vs $H_{1}\text{ : }\mathbf{p}=(0.45,\,0.45,\,0.05)$ (fifth column); $H_{0}\text{ : }\mathbf{p}=(0.25,\,0.25,\,0.25)$ vs $H_{1}\text{ : }\mathbf{p}=(0.85,\,0.05,\,0.05)$ (sixth column).}}
    \label{tab1}
\end{table}

Theorem \ref{teo33} allows us to define the loss of power of the ``true" LR test, which admits  a natural formulation as the additional sample size that is required in the ``true" LR test in order to recover the power of the Multinomial LR test. According to Theorem \ref{teo33}, for any fixed value $\beta > 0$ of the power, there exist: i) an integer $\bar{n}_{(\mathbf p_{0}, \mathbf p_{1})}(\alpha, \beta)$ such that, for $n\geq \bar{n}_{(\mathbf p_{0}, \mathbf p_{1})}(\alpha, \beta)$, the LR test in the absence of perturbation has eventually power not less than $\beta$; ii) an integer $\bar{n}_{(\mathbf p_{0}, \mathbf p_{1}),T}(\alpha, \beta)$ such that, for $n\geq \bar{n}_{(\mathbf p_{0}, \mathbf p_{1}),T}(\alpha, \beta)$, the ``true"  LR test has eventually power not less than $\beta$. That is, $\bar{n}_{(\mathbf p_{0}, \mathbf p_{1})}(\alpha, \beta)$ and $\bar{n}_{(\mathbf p_{0}, \mathbf p_{1}),T}(\alpha, \beta)$ provide lower bounds for the sample size $n$ required to obtain a power (at least equal to) $\beta$. Accordingly, for any $\mathbf p_{1} \in \Delta_{k-1}$ such that $\mathbf p_{1} \neq \mathbf p_{0}$, we can make use of the large $n$ asymptotic expansions of $\beta_{n}(\mathbf p_{1},\alpha)$ and $\beta_{n, T}(\mathbf p_{1},\alpha)$ in Theorem \ref{teo33}  to quantify $\bar{n}_{(\mathbf p_{0}, \mathbf p_{1})}(\alpha, \beta)$ and $\bar{n}_{(\mathbf p_{0}, \mathbf p_{1}),T}(\alpha, \beta)$ and, above all, the gap between these lower bounds. In particular, from \eqref{LD1} we can write
\begin{displaymath}
\bar{n}_{(\mathbf p_{0}, \mathbf p_{1}),T}(\alpha, \beta) \simeq \bar{n}_{(\mathbf p_{0}, \mathbf p_{1})}(\alpha, \beta) + \mathcal{R}_{(\mathbf p_{1}, \mathbf p_{0})}(\varepsilon, m)
\end{displaymath}
with
\begin{displaymath}
\mathcal{R}_{(\mathbf p_{1}, \mathbf p_{0})}(\varepsilon, m) = \bar{n}_{(\mathbf p_{0}, \mathbf p_{1})}(\alpha, \beta) \frac{\mathfrak L_{(\mathbf p_0, \mathbf p_{1})}(\ep,m)}{\mathcal D_{KL}(\mathbf p_0\,\|\, \mathbf p_{1})}, 
\end{displaymath}
where $\simeq$ means that $\bar{n}_{(\mathbf p_{0}, \mathbf p_{1}),T}(\alpha, \beta)$ and $\bar{n}_{(\mathbf p_{0}, \mathbf p_{1})}(\alpha, \beta)$ are obtained after dropping the $o(n^{-1})$-term on the right-hand side of \eqref{LD1}. Therefore, the gap $\mathcal{R}_{(\mathbf p_{1}, \mathbf p_{0})}(\varepsilon, m)$ between $\bar{n}_{(\mathbf p_{0}, \mathbf p_{1}),T}(\alpha, \beta)$ and $ \bar{n}_{(\mathbf p_{0}, \mathbf p_{1})}(\alpha, \beta) $ takes on the interpretation of the (sample) cost of $(\varepsilon,\delta)$-DP under the ``true" LR test. That is, the additional sample size required to recover the power of the Multinomial LR test.

We conclude with a brief discussion on the effect of the dimension $k>1$ on the power of the ``true"  LR test. Theorem \ref{teo33} shows how the sample size $n$ interacts with the truncated Laplace EM. However, according to Theorem \ref{teo33}, also $k$ interacts with the truncated Laplace EM, still in terms of the parameters $\varepsilon$ and $m$, thus having a direct effect on the power of the ``true" LR test. By maintaining the sample size $n$ large enough, we assume that $k$ increases moderately. Differently from the sample size $n$, any increase of the dimension $k$ determines a corresponding change of $\mathbf p_{0}$ and $\mathbf p_{1}$, as these probabilities are elements of $\Delta_{k-1}$. Therefore, instead of the sole $k$, it is critical to consider a suitable function of $k$, $\mathbf p_{0}$ and $\mathbf p_{1}$, say $\nu_{k,(\mathbf p_{0}, \mathbf p_{1})}$, and then show how such a function interacts with the truncated Laplace EM. From \eqref{LD1}, it is clear that the interaction between $\nu_{k,(\mathbf p_{0}, \mathbf p_{1})}$ and the parameters $\varepsilon$ and $m$ is determined by $-n^{-1}\mathfrak L_{(\mathbf p_0, \mathbf p_{1})}(\ep,m)$. The behaviour of such a term depends on the specific form of the moment generating function of the Laplace distribution \eqref{Laplace}, which turns out to be a convex, even function with a global minimum at $0$. The larger $m$ the flatter the $U$-shape of this function around its minimum. Since $\mathfrak L_{(\mathbf p_0, \mathbf p_{1})}(\ep,m)$ is the sum of $k$ logarithmic terms, it is critical to count the number of terms in the sum that are not of the form $\log(1+\omega)$, for some small $\omega$. That is, the number of ``large-scale terms'':
\begin{displaymath}
\nu_{k,(\mathbf p_{0}, \mathbf p_{1})}(\varepsilon, \eta) := \#\left\{i = 1, \dots, k\ :\ \left|\frac{\partial}{\partial p_{0,i}} \mathcal D_{KL}(\mathbf p_{0}\,\|\, \mathbf p_{1})\right| \geq \varepsilon(1+\eta) \right\}.
\end{displaymath}
By a suitable tuning of $\eta$, the quantity $\nu_{k,(\mathbf p_{0}, \mathbf p_{1})}(\varepsilon, \eta)$ gives the number of large-scale terms. The next proposition shows how $\nu_{k,(\mathbf p_{0}, \mathbf p_{1})}(\varepsilon, \eta)$ interacts with the truncated Laplace EM, and its impact on the power of the ``true"  LR test.

\begin{proposition} \label{prop:large_k}
Under the assumptions of Theorem \ref{teo33}, for a given $\eta > 0$, it holds that
\begin{itemize}
\item[i)] if $\nu_{k,(\mathbf p_{0}, \mathbf p_{1})}(\varepsilon, \eta) > 0$ then the contribution of $\mathfrak L_{(\mathbf p_0, \mathbf p_{1})}(\ep,m)$ to the power of the ``true"  LR test decreases at least proportionally to the quantity
\begin{displaymath}
\nu_{k,(\mathbf p_{0}, \mathbf p_{1})}(\varepsilon, \eta) \log \E[\exp\{L_1 \varepsilon(1+\eta)\}];
\end{displaymath}
\item[ii)] if $\nu_{k,(\mathbf p_{0}, \mathbf p_{1})}(\varepsilon, \eta) = 0$ and $\varepsilon^2(1+\eta)^2\mathsf{Var}[L] < 1$ then the contribution of $\mathfrak L_{(\mathbf p_0, \mathbf p_{1})}(\ep,m)$ to the power of the ``true"  LR decreases proportionally to 
\begin{displaymath}
k \varepsilon^2(1+\eta)^2\mathsf{Var}[L].
\end{displaymath}
\end{itemize}
\end{proposition}

See Appendix \ref{Appb} for the proof of Proposition \ref{prop:large_k}. According to Proposition \ref{prop:large_k}, the choice of the parameter $\eta >0$ determines a ``large scale'' and a ``small scale'' for the terms that are involved in the sum defined through $\mathfrak L_{(\mathbf p_0, \mathbf p_{1})}(\ep,m)$. Under the large scale behaviour of $\mathfrak L_{(\mathbf p_0, \mathbf p_{1})}(\ep,m)$, i.e. $\nu_{k,(\mathbf p_{0}, \mathbf p_{1})}(\varepsilon, \eta)>0$, Proposition \ref{prop:large_k} shows that there exists an interaction between $\nu_{k,(\mathbf p_{0}, \mathbf p_{1})}(\varepsilon, \eta)$ and the parameters $\varepsilon$ and $m$, and that such an interaction is with respect to $\log \E[\exp\{L_1 \varepsilon(1+\eta)\}]$. Under the small scale behaviour of $\mathfrak L_{(\mathbf p_0, \mathbf p_{1})}(\ep,m)$, i.e. $\nu_{k,(\mathbf p_{0}, \mathbf p_{1})}(\varepsilon, \eta)=0$, Proposition \ref{prop:large_k} shows that the exists an interaction between $k$ and the parameters $\varepsilon$ and $m$, and that such an interaction is with respect to $\mathsf{Var}[L]$. The small scale behaviour thus leads to an interaction directly in terms of the dimension $k$, thus showing clearly the effect of $k$ on the power of the ``true"  LR test. In particular, under the small scale behaviour, the contribution of $\mathfrak L_{(\mathbf p_0, \mathbf p_{1})}(\ep,m)$ to the power of the ``true"  LR test decreases proportionally to $k \varepsilon^2(1+\eta)^2\mathsf{Var}[L]$. That is, the smaller  $k$ and the smaller $\mathsf{Var}[L]$, i.e. the larger $\varepsilon$ and/or the smaller $m$, the smaller is the contribution of $-n^{-1} \mathfrak L_{(\mathbf p_0, \mathbf p_{1})}(\ep,m)$ to the power of the ``true"  LR test.

%%%%%%%%%%%%%%%%%%%%%%%%%%%%%%%%
%%%%%%%%%%%%%%%%%%%%%%%%%%%%%%%%
%%%%%%%%%%%%%%%%%%%%%%%%%%%%%%%%
%%%%%%%%%%%%%%%%%%%%%%%%%%%%%%%%

\section{A ``na\"ive" LR test for goodness-of-fit under $(\varepsilon,\delta)$-DP} \label{sec4}

As in Section \ref{sec3}, we assume that $\mathbf{a}=(a_{1},\ldots,a_{k})$ is perturbed by means of the truncated Laplace EM, in such a way that the resulting perturbed list $\mathbf{b}$ has the same sample size $n$ as the original list $\mathbf{a}$. Then, to avoid the release of negative values, we assume that $\mathbf{b}$ is post-processed in such a way that the negative values of $\mathbf{b}$ are replaced by zeros \citep[Section 6]{Rin(17)}. That is, we define $\mathbf{b}^{+}=(b_{1}^{+},\ldots,b_{k}^{+})$, where $b_{i}^{+}=\max\{0,\,b_{i}\}\geq0$ for $i=1,\ldots,k-1$ and $b_{k,n}=\max\{0,\,n-\sum_{1\leq i\leq k-1}b_{i}^{+}\}$. Note that the sample size $|\mathbf{b}^{+}|$ of $\mathbf{b}^{+}$ may be different from the sample size $n$ of $\mathbf{a}$. Assuming that $\mathbf{b}$ is modeled as the ``true" model described in Section \ref{sec3}, i.e. $\mathbf{b}$ is the realization of the random variable $\mathbf{B}=(B_{1,n},\ldots,B_{k-1,n})$ and $B_{k,n}=n-|\mathbf{B}_{n}|$, it is natural to model $\mathbf{b}^{+}$ as the realization of the random variable $\mathbf{B}^{+}=(B_{1,n}^{+},\ldots,B_{k,n}^{+})$ with $B_{i,n}=\max\{0,\,B_{i,n}\}\geq0$ for $i=1,\ldots,k-1$ and $B_{k,n}^{+}=\max\{0,n-\sum_{1\leq i\leq k-1}B_{i,n}^{+}\}$. Therefore, differently from $\mathbf{B}$, the sample size $|\mathbf{B}^{+}|$ of $\mathbf{B}^{+}$ is random. Following the work of \citet{Rin(17)}, we consider a ``na\"ive" model for the perturbed list $\mathbf{b}^{+}$, that is a statistical model that does not takes into account the truncated Laplace EM. The ``na\"ive" model assumes that $\mathbf{b}^{+}$ is modeled as the distribution of the random variable $\mathbf{B}^{+}=(B_{1,n}^{+},\ldots,B_{k,n}^{+})$ such that the conditional distribution of $\mathbf{B}^{+}$ given $|\mathbf{B^{+}}|=n^{+}$ is a Multinomial distribution with parameter $(\mathbf{p},n^{+})$. Under the ``na\"ive" model for $\mathbf{b}^{+}$, we make use of the LR test to assess goodness-of-fit in the form: $H_{0}\text{ : }\mathbf{p}=\mathbf{p}_{0}$, for a fixed $\mathbf{p}_{0}\in\Delta_{k-1}$, against $H_{1}\text{ : }\mathbf{p}\neq\mathbf{p}_{0}$. As in Section \ref{sec3}, without loss of generality, we assume that $\mathbf{p}_{0}$ belongs to the interior of $\Delta_{k-1}$.

\begin{remark}
A straightforward application of classical concentration inequalities for the Multinomial distribution entails that, for any fixed $n\geq1$, the probability of the event $\{\mathbf{B}^{+} \neq \mathbf{B}\}$ can be bounded by an exponential term of the form $e^{-\tau n}$, for a suitable choice of $\tau > 0$. This result shows that the additional randomness carried by the sample size $|\mathbf{B}^{+}|$ does not contribute in any significant way to the Edgeworth expansions of the distribution of the ``na\"ive" LR  in Theorem \ref{teo31} and Theorem \ref{teo32n} below, where only terms of $O(n^{-1/2})$ and $O(n^{-1})$-type are relevant. We refer to Remark \ref{rmk:pezza} for further details. 
\end{remark}

Under the ``na\"ive" model for the post-processed list $\mathbf{b}^{+}$, the likelihood function is
\begin{equation}\label{naiv_mod}
L_{n^{+},N}(\mathbf{p};\mathbf{b}^{+})\propto\left(\prod_{i=1}^{k-1}\frac{p_{i}^{b_{i}^{+}}}{b_{i}^{+}!}\right)\left(1-|\mathbf{p}|\right)^{b^{+}_{k}}.
\end{equation}

Now, for the sake of simplicity, we start developing an asymptotic analysis under the assumption that $k=2$. That is, $A_{n}$ has a Binomial distribution with parameter $(n,p)$, $L$ is independent of $A_{n}$ and distributed according to \eqref{Laplace}, and $B_{n}=A_{n}+L$. Under the ``na\"ive" model with $k=2$, we make use of the LR test to assess goodness-of-fit in the form: $H_{0}\text{ : }p=p_{0}$, for a fixed $p_{0}\in(0,1)$, against $H_{1}\text{ : }p\neq p_{0}$. Recall that the LR under the ``true" model  is given by \eqref{LR_multidim} with $k=2$. We denote by $\hat{p}_{n^{+},N}$ the maximum likelihood estimator of $p$ under the ``na\"ive" model for $\mathbf{b}^{+}$, that is the estimator obtained through the maximization of the likelihood function \eqref{naiv_mod} with $k=2$. Then, we define
\begin{displaymath}
\Lambda_{n^{+},N}(p_{0})=2\log\left(\frac{L_{n^{+},N}(\hat{p}_{n^{+},N};B_{n})}{L_{n^{+},N}(p_{0};B_{n})}\right),
\end{displaymath}
which is referred to as the ``na\"ive" LR, and we denote by $\text{Pr}[\Lambda_{n^{+},N}(p_{0})\in\cdot;\,p_{0}]$ the distribution of $\Lambda_{n^{+},N}(p_{0})$ computed with respect to $A_{n}$ distributed as a Binomial distribution with parameter $(n,p_{0})$. The next theorem establishes an Edgeworth expansion, with respect to $n\rightarrow+\infty$, for the distribution of $\Lambda_{n^{+},N}(\mathbf{p}_{0})$.  To emphasize the comparison with respect to the ``true" model, we also include the Edgeworth expansion of the distribution of $\Lambda_{n,T}(p_{0})$, which follows from Theorem \ref{teo32} with $k=2$. We denote by $\K$ the the cumulative distribution function of a chi-squared distribution with $1$ degree of freedom.

\begin{theorem}\label{teo31}
Let $\F_{n,T}(t)=\mathrm{Pr}[\Lambda_{n,T}(p_{0})\leq t;\,p_{0}]$ be the cumulative distribution function of $\Lambda_{n,T}(p_{0})$, and let $\F_{n^{+},N}(t)=\mathrm{Pr}[\Lambda_{n^{+},N}(p_{0})\leq t;\,p_{0}]$ be the cumulative distribution function of $\Lambda_{n^{+},N}(p_{0})$. For any $n\in\mathbb{N}$ and any $t>0$
\begin{equation}\label{exp_t}
\F_{n,T}(t)= \K(t)+\frac{c_{1,p_{0}}(t)}{\sqrt{n}}+\frac{c_{\ast,p_{0}}(t;m,\varepsilon)}{e^{t/2}\sqrt{2\pi n}} +\frac{c_{2,p_{0}}(t)}{n} + R_{n,p_{0},T}(t)
\end{equation}
with 
\begin{displaymath}
|R_{n,p_{0},T}(t)| \leq \frac{C_{p_{0},N}(t;m, \varepsilon)}{n^{3/2}},
\end{displaymath}
and
\begin{equation}\label{exp_n}
\F_{n^{+},N}(t)= \K(t)+\frac{c_{1,p_{0}}(t)}{\sqrt{n}}+\frac{c_{2,p_{0}}(t)}{n}-\left(\frac{t\text{e}^{-t}}{2\pi}\right)^{1/2}\frac{\mathsf{Var}[L]}{np_{0}(1-p_{0})} + R_{n,p_{0},N}(t),
\end{equation}
with
\begin{displaymath}
|R_{n,p_{0},N}(t)| \leq \frac{C_{p_{0},N}(t; m, \varepsilon)}{n^{3/2}},
\end{displaymath}
where: the functions $c_{1,p_{0}}$ and $c_{2}$ are independent of $n$ and of the distribution of $L$, that is independent of $\varepsilon$ and $m$; the function $c_{\ast,p_{0}}$ is independent of $n$ and such that $|c_{\ast,p_{0}}|\leq 1$; the functions $C_{p_{0},T}$ and $C_{p_{0},N}$ are independent of $n$.
\end{theorem}

See Appendix \ref{Appa} for the proof of Theorem \ref{teo31}. The functions $c_{1,p_{0}}$, $c_{\ast,p_{0}}$ and $c_{2,p_{0}}$ in Theorem \ref{teo31}, as well as the functions $C_{p_{0},T}$ and $C_{p_{0},N}$, can be made explicit by gathering some equations in the proof of the theorem. To assess goodness-of-fit by means of $\Lambda_{n,T}(p_{0})$ and $\Lambda_{n^{+},N}(p_{0})$, we fix a (reference) level of significance $\alpha\in(0,1)$ and find $\lambda_{T}(\alpha)> 0$ and $\lambda_{N}(\alpha)> 0$ such that $\K(\lambda_{T}(\alpha))=1-\alpha$ and $\K(\lambda_{N}(\alpha))=1-\alpha$, respectively. Then, we define the rejection regions $\{\Lambda_{n,T}(p_{0}) > \lambda_{T}(\alpha)\}$ and $\{\Lambda_{n^{+},N}(p_{0}) > \lambda_{N}(\alpha)\}$. Some empirical analyses in \citet{Rin(17)} show that the ``true" LR test has statistical significance at level $\alpha$, with a power that varies with $\varepsilon$ and $m$, whereas the ``na\"ive" LR test has no statistical significance at the same level $\alpha$. Theorem \ref{teo31} provides a theoretical guarantee to such empirical analyses, showing that the term $(np_{0}(1-p_{0}))^{-1}\mathsf{Var}[L]$ determines the loss in the statistical significance of the ``na\"ive" LR test. For a fixed $n\geq1$, the variance of the truncated Laplace EM is critical to control  the closeness between $\F_{n^{+},N}$ and $\K$: the larger $(np_{0}(1-p_{0}))^{-1}\mathsf{Var}[L]$ the less close $\F_{n^{+},N}$ to $K$, and hence the larger $(np_{0}(1-p_{0}))^{-1}\mathsf{Var}[L]$ the less the statistical significance of the ``na\"ive" LR test. Note that $\mathsf{Var}[L]=c_{\varepsilon,m}^{-1}\sum_{-m\leq l\leq m}l^{2}\text{e}^{-\varepsilon|l|}$ increases as $\varepsilon$ decreases and/or $m$ increases. Then, the loss in the statistical significance of the ``na\"ive" LR test is driven by $\varepsilon$ and $m$, through $\mathsf{Var}[L]$, in combination with $n$. Along the same lines of the proof of Theorem \ref{teo32}, it is easy to show that an analogous result holds true for any convolutional-type EM. In particular, for the truncated Gaussian EM, Theorem \ref{teo32} holds true with the variance of $L$ replaced by the variance of $G$ with distribution \eqref{gauss_n}. For some fixed values of $\varepsilon$ and $\delta$, that is fixed levels of $(\varepsilon,\delta)$-DP, Table \ref{tab2} shows that $\mathsf{Var}[L]\leq \mathsf{Var}[G]$, and hence the truncated Laplace EM produces a smaller loss of statistical significance in the ``na\"ive" LR test than the truncated Gaussian EM.

\begin{table}[!htb]
    \centering
    \resizebox{0.4\textwidth}{!}{\begin{tabular}{lrrrr}
        \toprule
        \multicolumn{1}{l}{} & \multicolumn{2}{c}{$m$}&\multicolumn{2}{c}{} \\
        \cmidrule(r){2-3} \cmidrule(r){4-5} 
        $(\varepsilon,\,\delta)$ & L & G & $\mathsf{Var}(L)$ & $\mathsf{Var}(G)$  \\[0.05cm]
        \midrule
            (0.025,\,0.09) & 5&   5& 9.660&   9.824   \\[0.1cm]
             (0.025,\,0.04) & 10&   10& 34.288 &   35.411  \\[0.1cm]
            (0.05,\,0.08) & 5&   5& 9.650 &  9.324  \\[0.1cm]
             (0.05,\,0.04) & 10&   10& 31.965 &  34.187\\[0.1cm]
            (0.075,\,0.08) &5&   5& 8.991 &  9.478   \\[0.1cm]
              (0.075,\,0.03) & 10&   11& 29.711 &  39.189   \\[0.1cm]
           (0.1,\,0.07) & 5&   6& 8.662 &  12.850  \\[0.1cm]
             (0.1,\,0.03) & 10&   12& 27.541 &  43.925   \\[0.1cm]
           (0.125,\,0.07) & 5&   6& 8.337 &  12.574  \\[0.1cm]
            (0.125,\,0.02) & 10&   12& 25.465 &  42.084  \\[0.1cm]
           (0.150,\,0.06) & 5&   6& 8.019 &  12.303  \\[0.1cm]
            (0.150,\,0.02) & 10&   12& 23.493 &  40.317   \\[0.1cm]
           (0.175,\,0.06) & 5&   6& 7.706 &  12.037  \\[0.1cm]
           (0.175,\,0.02)& 10&   12 & 21.631 &  38.624   \\[0.1cm]
           (0.2,\,0.05) & 5&   6& 7.399 & 11.776   \\[0.1cm]
            (0.2,\,0.02) & 10&   13& 19.884 &  42.001  \\[0.1cm]
          \bottomrule
    \end{tabular}}
    \caption{\small{Variance of the truncated Laplace EM (L) and the truncated Gaussian EM for fixed levels of $(\varepsilon,\delta)$-DP.}}
    \label{tab2}
\end{table}

We conclude by extending Theorem \ref{teo31} to an arbitrary dimension $k\geq2$. Note that it is sufficient to present such an extension under the ``na\"ive" mode, as the extension under the ``true" model is precisely Theorem \ref{teo32}. Recall that $\mathbf{A}_{n}$ is distributed as a Multinomial distribution with parameter $(n,\mathbf{p})$, $\mathbf{L}=(L_{1},\ldots,L_{k-1})$ is independent of $\mathbf{A}_{n}$ and with the $L_{i}$'s being random variables i.i.d. according to \eqref{Laplace}, and $\mathbf{B}_{n}=(B_{1,n},\ldots,B_{k-1,n})$ with $B_{i,n}=A_{i,n}+L_{i}$ for $i=1,\ldots,k-1$ and $B_{k,n}=n-|\mathbf{B}_{n}|$. Under the ``na\"ive" model, we make use of the LR test to assess goodness-of-fit in the form: $H_{0}\text{ : }\mathbf{p}=\mathbf{p}_{0}$, for a fixed $\mathbf{p}_{0}\in\Delta_{k-1}$, against $H_{1}\text{ : }\mathbf{p}\neq\mathbf{p}_{0}$. The LR under the ``true" model  is given by \eqref{LR_multidim}. We denote by $\hat{\mathbf{p}}_{n^{+},N}$ be the maximum likelihood estimator of $\mathbf{p}$ under the ``na\"ive" model for $\mathbf{b}^{+}$, that is the estimator obtained through the maximization of the likelihood function \eqref{naiv_mod}. Then, we define the LR
\begin{displaymath}
\Lambda_{n^{+},N}(\mathbf{p}_{0})=2\log\left(\frac{L_{n^{+},N}(\hat{\mathbf{p}}_{n^{+},N};\mathbf{B}_{n})}{L_{n^{+},N}(\mathbf{p}_{0};\mathbf{B}_{n})}\right),
\end{displaymath}
which is referred to as the ``na\"ive" LR, and we denote by $\text{Pr}[\Lambda_{n^{+},N}(\mathbf{p}_{0})\in\cdot;\,p_{0}]$ the distribution of $\Lambda_{n^{+},N}(\mathbf{p}_{0})$ computed with respect to $\mathbf{A}_{n}$ distributed as a Multinomial distribution with parameter $(n,\mathbf{p}_{0})$. As $n\rightarrow+\infty$, the theorem establishes an Edgeworth expansion for the distribution of ``na\"ive" LR, 
$\Lambda_{n^{+},N}(\mathbf{p}_{0})$. With respect to $k=2$, there is a new element to take into account, that is the  $(k-1)\times(k-1)$ covariance matrix $\mathbf{\Sigma}(\mathbf{p})$ whose entry $\sigma_{r,s}$ is $\sigma_{r,s}=-p_r p_s$ for $1\leq r\neq s\leq k-1$ and by $\sigma_{r,r} = p_r (1-p_r)$ for $1\leq r\leq k-1$.

\begin{theorem}\label{teo32n}
Let $\mathbb I(\pb)$ denote the Fisher information matrix of the Multinomial model with parameter $(n,\mathbf{p})$, and let $F_{n^{+},N}(t)=\mathrm{Pr}[\Lambda_{n^{+},N}(\mathbf{p}_{0})\leq t;\,\mathbf{p}_{0}]$ be the cumulative distribution function of $\Lambda_{n^{+},N}(\mathbf{p}_{0})$. For any $n,k\in\mathbb{N}$  and any $t>0$
\begin{align}\label{BE_LRNk}
&\F_{n^{+},N}(t)\\
&\notag\quad= \K_{k-1}(t)+\frac{c_{1,k,\mathbf{p}_{0}}(t)}{n^{1/2}}+\frac{c_{2,k,\mathbf{p}_{0}}(t)}{n}  \\
&\notag\quad\quad- \frac{\left(\frac 12\right)^{(k+1)/2}\text{e}^{-t/2} t^{(k-1)/2}}{\Gamma\left(\frac{k+1}{2}\right)}\frac{\mathsf{Var}[L_1]}{n} \mathrm{tr}(\mathbb I(\pb_0))+  R_{(k,n),\mathbf{p}_{0},N}(t), \nonumber 
\end{align}
with
\begin{displaymath}
|R_{(k,n),\mathbf{p}_{0},N}(t)| \leq \frac{C_{k,\mathbf{p}_{0},N}(t; m, \varepsilon)}{n^{3/2}},
\end{displaymath}
where the functions $c_{1,k,\mathbf{p}_{0}}$ and $c_{2,k,\mathbf{p}_{0}}$ are independent of $n$ and of the distribution of $\mathbf{L}$, i.e. independent of $\varepsilon$ and $m$, and the function $C_{k,\mathbf{p}_{0},N}$ is independent of $n$.
\end{theorem}

See Appendix \ref{Appb} for the proof of Theorem \ref{teo32n}. Theorem \ref{teo32n} leads to the same conclusions as Theorem \ref{teo31}, showing the critical role of the Laplace perturbation mechanism. In general, Theorem \ref{teo31} and Theorem \ref{teo32n} highlight the importance of taking the perturbation into account in privacy-protecting LR tests for goodness-of-fit, and therefore the importance of negative values in the released list $\mathbf{b}$.  Official statistical agencies are typically reluctant to disseminate perturbed tables with negative frequencies, and hence a common policy to preserve $(\varepsilon,\delta)$-DP consists in reporting negative values as zeros. However, as a matter of fact, releasing tables that have an appearance similar to that of original tables may lead to ignoring the perturbation and to analyzing data as they were not perturbed, i.e. under the ``na\"ive" model  \citep{Rin(17)}. Our analysis shows the importance of taking the perturbation into account, i.e. the importance of the ``true" model versus the ``na\"ive", showing a loss in the statistical significance of the test when perturbed data are treated as they were not perturbed. Such a loss provides an evidence of the importance of taking the perturbation into account in the statistical model, thus endorsing the release of negative values if $(\varepsilon,\delta)$-DP is adopted.

%%%%%%%%%%%%%%%%%%%%%%%%%%%%%%%%
%%%%%%%%%%%%%%%%%%%%%%%%%%%%%%%%
%%%%%%%%%%%%%%%%%%%%%%%%%%%%%%%%
%%%%%%%%%%%%%%%%%%%%%%%%%%%%%%%%

\section{Discussion}\label{sec5}

Under the framework of $(\varepsilon,\delta)$-DP for frequency table, we developed a rigorous analysis of the large sample behaviour of the ``true" LR test. Our main contributions are with respect to the power analysis of the ``true" LR test, and they built upon a Bahadur-Rao large deviation expansion for the power of the ``true" LR test. By relying on a novel (sharp) large deviation principle or sum of i.i.d. random vectors, such an expansion brought out the critical quantity $\mathfrak L_{(\mathbf p_0, \mathbf p_{1})}(\ep,m)$, which determines a loss in the power of the ``true" LR test. This result has then been applied to characterize the impact of the sample size $n$ and the dimension $k$ of the table, in connection with the parameters $(\varepsilon,\delta)$, on the loss of the power of the private LR test. In particular, we determined the (sample) cost of $(\varepsilon,\delta)$-DP under the ``true" LR test, namely the additional sample size required to recover the power of the Multinomial LR test in the absence of perturbation. As a complement to our power analysis of the ``true" LR test, we investigated the well-known problem of releasing negative values in frequency tables under the $(\varepsilon,\delta)$-DP. By comparing the Edgworth expansions for the distribution of the ``true" LR test and the ``na\"ive" LR test, we showed the importance of taking the perturbation into account in private LR tests for goodness-of-fit, thus providing the first rigorous evidence to endorse the release of negative values when $(\varepsilon,\delta)$-DP is adopted. Our work provides the first rigorous treatment of privacy-protecting LR tests for goodness-of-fit in frequency tables and, in particular, it is the first work to make use of the power of the test to quantify the trade-off between confidentiality and utility. This is achieved through a non-standard large deviation analysis of the LR test, which is known to be challenging in a setting such as ours, where the statistical model is multidimensional, discrete, and not belonging to the exponential family.

Our power analysis of the ``true" LR test can be easily extended to any truncated convolutional-type EM. As an example, we considered the truncated Gaussian EM, comparing it with the truncated Laplace EM. A numerical comparison showed how the truncated Laplace EM produces a smaller decrease in the power of the ``true" LR test than the truncated Gaussian EM. Such a finding leads to the natural problem of identifying the optimal truncated convolutional-type EM, that is the truncated convolutional-type EM that leads to the smallest decrease in the power of the ``true" LR test. In particular, is the truncated Laplace EM the optimal convolutional-type EM? This is an interesting open problem, whose rigorous solution requires an extension of Theorem \ref{teo33} to deal with truncated EMs such that $v(a_{i},b_{i})=g(a_{i}-b_{i})$ for a general function $g$. Based on the proof of Theorem \ref{teo33}, we conjecture that the resulting Bahadur-Rao large deviation expansion is still in the form \eqref{LD1}, with a critical term at the order $n^{-1}$ that depends on $g$. Given that, one has to deal with a challenging optimization problem with respect to $g$, that is finding $g$ that leads to the smallest decrease in the power of the ``true" LR test. Along similar lines, one may consider the more general problem of  identifying the optimal truncated EM, thus removing the assumption that $v(a_{i},b_{i})=g(a_{i})-b_{i})$. However, for a general EM, it is difficult to conjecture the corresponding generalization of Theorem \ref{teo33}. Still related to Theorem \ref{teo33}, an interesting open problem is to consider a scaling of the parameter $(\varepsilon,\delta)$ with respect to the sample size $n$, say $(\varepsilon_{n},\delta_{n})$, and then identify the right scaling for which the contribution of the truncated Laplace EM to the power of the ``true" LR test does no longer vanish in the large $n$ limit. 

Among other directions for future research, the study of optimal properties of the ``true" LR is of special interest. In particular, for a fixed (reference) level of significance $\alpha$, let $\beta^{\ast}_{n}(\mathbf p_{1};\alpha)$ and $\beta_{n}(\mathbf p_{1};\alpha)$ denote the powers of an arbitrary test statistic and of the LR test, respectively, for testing testing $H_{0}\text{ : } \mathbf p = \mathbf p_0$ against $H_{1}\text{ : } \mathbf p=\mathbf{p}_{1}$, for fixed $\mathbf{p}_0,\,\mathbf{p}_{1}\in\Delta_{k-1}$. It is known from \citet{Bah(60)} that 
\begin{equation}\label{eq:low}
\liminf_{n\rightarrow+\infty}\frac{1}{n}\log(1-\beta^{\ast}_{n}(\mathbf p_{1};\alpha))\geq -\mathcal D_{KL}(\mathbf p_{1}\,\|\, \mathbf p_{0})
\end{equation}
and 
\begin{displaymath}
\lim_{n\rightarrow+\infty}\frac{1}{n}\log(1-\beta_{n}(\mathbf p_{1};\alpha))= -\mathcal D_{KL}(\mathbf p_{1}\,\|\, \mathbf p_{0}).
\end{displaymath}
That is, the LR test attains the lower bound \eqref{eq:low}. This is referred to as Bahadur efficiency of the LR test, and it provides a well-known optimal property of the LR test \citep{Bah(60),Bah(67)}. An open problem emerging from our work is to establish a ``private" version of the lower bound \eqref{eq:low}, and  investigate the Bahadur efficiency of the ``true" LR with respect to such a lower bound. Establishing such a property of optimality for the ``true" LR test would provide a framework to compare goodness-of-fit tests under $(\varepsilon,\delta)$-DP.

%%%%%%%%%%%%%%%%%%%%%%%%%%%%%%%%
%%%%%%%%%%%%%%%%%%%%%%%%%%%%%%%%
%%%%%%%%%%%%%%%%%%%%%%%%%%%%%%%%
%%%%%%%%%%%%%%%%%%%%%%%%%%%%%%%%
\appendix

\section{Proofs for $k=2$}\label{Appa}

In this section we will prove Theorems \ref{teo31} and \ref{teo33} in the case that our data set is a table with two cells, i.e. when $k=2$. Recall that, in this case, the statement of Theorem 
\ref{teo32} is already included in that of Theorem \ref{teo31}. 
Thus, the observable variable reduces to the counting $b$ contained in the
first cell. At the beginning, some preparatory steps are needed to analyze the likelihood, the MLE and the LR relative to the null hypothesis $H_0 : p=p_0$, under the true model \eqref{true_mod}. 

\subsection{Preparatory steps}

First of all, let us state a proposition that fixes the exact expression of the likelihood under the true model \eqref{true_mod}. 
\begin{lem} \label{lemma_likelihood}
Let $m \in \N$ and $\varepsilon > 0$ be fixed to define the Laplace distribution \eqref{Laplace}. Then, if $n > m$ and $b \in R(n,m) := \{m, m+1, \dots, n-m\}$, we have
\begin{equation} \label{LikTrueReg} 
L_{n,T}(p; b) = \frac{1}{\cepsm} \sum_{l=-m}^m e^{-\ep |l|} \binom{n}{b-l} p^{b-l} (1-p)^{n-b+l} \qquad (p \in [0,1]).
\end{equation}
Moreover, under the same assumption, for any $p \in (0,1)$ we can write
\begin{equation} \label{LikTrueDec} 
L_{n,T}(p; b) = L_n^{(0)}(p; b) \cdot H_{n,T}^{(\ep,m)}(p; b)
\end{equation}
with
\begin{align}
L_n^{(0)}(p; b) &:= \binom{n}{b} p^{b} (1-p)^{n-b} \label{Definition_Ln0} \\
H_{n,T}^{(\ep,m)}(p; b) &:= \frac{1}{\cepsm} \sum_{l=-m}^m e^{-\ep |l|} \rho(n,b,l) \left(\frac{1-p}{p}\right)^l \label{Definition_HnT} \\
\rho(n,b,l) &:= \frac{b! (n-b)!}{(b-l)! (n-b+l)!} \label{Definition_rho_nbl} \ . 
\end{align}
\end{lem}

\begin{proof}
Start from
\begin{equation} \label{ApiuL} 
B_n = A_n + L
\end{equation}
where the random variables $A_{n}$ and $L$, defined on the probability space $(\Omega,\mathscr{F},\text{Pr})$, are independent, $A_n \sim Bin(n,p)$ and $L$ has the Laplace distribution \eqref{Laplace}. 
By definition $L_{n,T}(p; b) := \textrm{Pr}[B_n = b]$, so that, under the assumption of the Lemma, the independence of $A_{n}$ and $L$ entails
$$
\textrm{Pr}[B_n = b] = \sum_{l=-m}^m \textrm{Pr}[A_n = b-l] \cdot \textrm{Pr}[L = l]
$$  
proving \eqref{LikTrueReg}. Finally, the decomposition \eqref{LikTrueDec} ensues from straightforward algebraic manipulations.
\end{proof}

The next step aims at providing a large $n$ asymptotic expansion of the true likelihood. This expansion can be obtained by considering the observable quantity $b$ as itself dependent by $n$, according to the 
following
\begin{lem} \label{lemma_rho}
Let $\xi \in (0,1)$ be a fixed number. Under the assumption that $b = n(\xi + \epsilon_n)$ with $\lim_{n\rightarrow 0} \epsilon_n =0$, there exists $n_0 = n_0(\xi; \ep,m)$ such that 
\begin{align} \label{exp_rhon} 
\rho_n(\xi,l) &:= \rho(n, b, l) = \left(\frac{\xi}{1-\xi}\right)^l \times \\
&\times \left\{1 + \alpha_1(\xi,l) \epsilon_n + \alpha_2(\xi,l) \epsilon_n^2 + \alpha_3(\xi,l) \frac 1n + \alpha_4(\xi,l) \frac{\epsilon_n}{n} + R_n(\xi,l) \right\} \nonumber 
\end{align}
holds for any $n \geq n_0$ and $l \in \{-m, \dots, m\}$, where
\begin{align*}
\alpha_1(\xi,l) &:= \frac{l}{\xi(1-\xi)} \\
\alpha_2(\xi,l) &:= \frac{l(l-1+2\xi)}{2\xi^2(1-\xi)^2} \\
\alpha_3(\xi,l) &:= -\frac{l(l-1+2\xi)}{2\xi(1-\xi)} \\
\alpha_4(\xi,l) &:= -\frac{l(l^2 - 2l +1 + 2\xi^2 + 4l\xi - 2\xi)}{2\xi^2(1-\xi)^2} \\
|R_n(\xi,l)| &\leq C(\xi,l) \left[|\epsilon_n|^3 + \frac{1}{n^2} + \frac{\epsilon_n^2}{n} \right]
\end{align*}
for some constant $C(\xi,l)$. Therefore, under the same assumption, there holds
\begin{align}
H_{n,T}^{(\ep,m)}(p; b) &= \mathcal H_0^{(\ep,m)}(p; \xi) + \mathcal H_1^{(\ep,m)}(p; \xi) \epsilon_n + \mathcal H_2^{(\ep,m)}(p; \xi) \epsilon_n^2 \label{HnT_ep_m}\\
&+ \mathcal H_3^{(\ep,m)}(p; \xi) \frac 1n + \mathcal H_4^{(\ep,m)}(p; \xi) \frac{\epsilon_n}{n} + \mathcal R_n(p;\xi) \nonumber
\end{align}
for any $n \geq n_0$, where
\begin{align*}
\mathcal H_0^{(\ep,m)}(p; \xi) &:= \frac{1}{\cepsm} \sum_{l=-m}^m e^{-\ep |l|} \left(\frac{\xi(1-p)}{p(1-\xi)}\right)^l \\
\mathcal H_i^{(\ep,m)}(p; \xi) &:= \frac{1}{\cepsm} \sum_{l=-m}^m e^{-\ep |l|} \alpha_i(\xi,l) \left(\frac{\xi(1-p)}{p(1-\xi)}\right)^l \qquad (i= 1,2,3,4)\\
|\mathcal R_n(p; \xi)| &\leq \mathcal C(p;\xi) \left[|\epsilon_n|^3 + \frac{1}{n^2} + \frac{\epsilon_n^2}{n} \right]
\end{align*}
with $\mathcal C(p;\xi) := \frac{1}{\cepsm} \sum_{l=-m}^m e^{-\ep |l|} C(\xi,l) \left(\frac{\xi(1-p)}{p(1-\xi)}\right)^l$. 
\end{lem}

\begin{proof}
We start by dealing with \eqref{exp_rhon}. First, we find $n_0 = n_0(\xi; \ep,m)$ in such a way that the assumptions of Lemma \ref{lemma_likelihood} are fulfilled for any $n \geq n_0$. Now, if $l=0$, then 
$\rho_n(\xi,0) = 1$ and the thesis follows trivially. Also, if $l=1$, then $\rho_n(\xi,1) = (\xi + \epsilon_n)/(1-\xi- \epsilon_n + \frac 1n)$ and the validity of the thesis can be checked by direct computation.  
Then, if $l \in \{2, \dots, m\}$, we use \eqref{Definition_rho_nbl} to get
$$
\rho(n, b, l) = \frac{(b)_{\downarrow l}}{(n-b+1)_{\uparrow l}} = \frac{(b)_{\downarrow l} (n-b)}{(n-b)_{\uparrow (l+1)}} = \frac{\sum_{k=1}^l \mathfrak s(l,k) b^k}{\sum_{k=1}^{l+1} |\mathfrak s(l+1,k)| (n-b)^{k-1}} 
$$
where $\mathfrak s(l,k)$ denotes the Stirling number of first kind. Whence, 
$$
\rho_n(\xi,l) = \frac{\sum_{k=1}^l \mathfrak s(l,k) n^{k-l}(\xi + \epsilon_n)^k}{\sum_{k=1}^{l+1} |\mathfrak s(l+1,k)| n^{k-l-1} (1- \xi - \epsilon_n)^{k-1}} = \frac{\mathfrak N_n}{\mathfrak D_n}\ .
$$
At this stage, recalling that $\mathfrak s(l,l) = 1$ and $\mathfrak s(l,l-1) = -\binom{l}{2}$, and exploiting the binomial formula, we have
\begin{align*}
\mathfrak{N}_n &:= \xi^l \left\{1 + l \frac{\epsilon_n}{\xi} + \frac{l(l-1)}{2} \frac{\epsilon_n^2}{\xi^2} - \frac{l(l-1)}{2n} \frac{1}{\xi} - \frac{l(l-1)^2}{2n} \frac{\epsilon_n}{\xi^2} + R_n(\xi,l)\right\} \\
\mathfrak{D}_n &:= (1-\xi)^l  \left\{1 - l \frac{\epsilon_n}{1-\xi} + \frac{l(l-1)}{2} \frac{\epsilon_n^2}{(1-\xi)^2} + \frac{l(l+1)}{2n} \frac{1}{1-\xi} \right. \\
&\left. - \frac{l(l-1)(l+1)}{2n} \frac{\epsilon_n}{(1-\xi)^2} + R_n(\xi,l)\right\} 
\end{align*}
for suitable expressions of $R_n(\xi,l)$ (possibly different from line to line) satisfying, in any case, the relation 
$$
|R_n(\xi,l)| \leq C(\xi,l) \left[|\epsilon_n|^3 + \frac{1}{n^2} + \frac{\epsilon_n^2}{n} \right] \ .
$$ 
To proceed further, we exploit that $\frac{1}{1+t} = 1 - t + t^2 + o(t^2)$ as $t\rightarrow 0$, to obtain
\begin{align*}
\mathfrak{D}_n^{-1} &:= \frac{1}{(1-\xi)^l} \left\{1 + l \frac{\epsilon_n}{1-\xi} + \frac{l(l+1)}{2} \frac{\epsilon_n^2}{(1-\xi)^2} - \frac{l(l+1)}{2n} \frac{1}{1-\xi} \right. \\
&\left. - \frac{l(l+1)^2}{2n} \frac{\epsilon_n}{(1-\xi)^2} + R_n(\xi,l)\right\} \ .
\end{align*}
The thesis now follows by multiplying the last expression by that of $\mathfrak{N}_n$, neglecting all the terms which are comparable with $R_n(\xi,l)$. Thus, \eqref{exp_rhon} is proved also for all 
$l \in \{2, \dots, m\}$.

For $l \in \{-m, \dots, -1\}$ the argument can be reduced to the previous case. In fact, since $\rho(n, b, -l) = \rho(n, n-b, l)$, we can put $h := -l > 0$, $\eta := 1-\xi$ and $\delta_n := -\epsilon_n$ 
to obtain
\begin{align*}
\rho_n(\xi,l) &= \left(\frac{\eta}{1-\eta}\right)^h \left\{1 + \alpha_1(\eta,h) \delta_n + \alpha_2(\eta,h) \delta_n^2 + \alpha_3(\eta,h) \frac 1n \right. \\ 
& \left. + \alpha_4(\eta,h) \frac{\delta_n}{n} + R_n(\eta,h) \right\} \ .
\end{align*}
This completes the proof of \eqref{exp_rhon} since $\left(\frac{\eta}{1-\eta}\right)^h = \left(\frac{\xi}{1-\xi}\right)^l$ and $-\alpha_1(\eta,h) = \alpha_1(\xi,l)$, $\alpha_2(\eta,h) = \alpha_2(\xi,l)$,
$\alpha_3(\eta,h) = \alpha_3(\xi,l)$ and $-\alpha_4(\eta,h) = \alpha_4(\xi,l)$.

Finally, \eqref{HnT_ep_m} follows immediately from the combination of \eqref{Definition_HnT} with \eqref{exp_rhon}.
\end{proof}

We can now provide a large $n$ asymptotic expansion for the MLE relative to the true likelihood $L_{n,T}$, contained in the following
\begin{lem} \label{lem:MLE}
Under the same assumption of Lemma \ref{lemma_rho}, there holds
\begin{equation} \label{Expansion_MLE_true}
\hat{p}_{n,T} = \xi + \epsilon_n + O\left(\frac{\epsilon_n^3}{n}\right) \ .
%\Gamma_1(\xi) \frac{\epsilon_n}{n} + 
%\Gamma_2(\xi) \frac{\epsilon_n^2}{n} + \Delta_n(\xi; \epsilon_n)
\end{equation}
%where
%\begin{align}
%\Gamma_2(\xi) &:= \label{Gamma2} \\
%\Delta_n(\xi; \epsilon_n) &\leq C(\xi) \frac{|\epsilon_n|^3}{n} \nonumber
%\end{align}
%for some constant $C(\xi)$. 
\end{lem}

\begin{proof}
First, we get a large $n$ expansion for the $\log$-likelihood, as follows
\begin{align*}
\ell_{n,T}(p;b) := \log L_{n,T}(p;b) &= \log \binom{n}{b} + n(\xi + \epsilon_n) \log p + n(1- \xi - \epsilon_n) \log(1- p) \\
&+ h_0^{(\ep,m)}(p; \xi) + h_1^{(\ep,m)}(p; \xi) \epsilon_n + h_2^{(\ep,m)}(p; \xi) \epsilon_n^2 \\
&+ h_3^{(\ep,m)}(p; \xi) \frac 1n + h_4^{(\ep,m)}(p; \xi) \frac{\epsilon_n}{n} + r_n(p;\xi)
\end{align*}
where
\begin{align*}
h_0^{(\ep,m)}(p; \xi) &:= \log \mathcal H_0^{(\ep,m)}(p; \xi) \\
h_1^{(\ep,m)}(p; \xi) &:= \frac{\mathcal H_1^{(\ep,m)}(p; \xi)}{\mathcal H_0^{(\ep,m)}(p; \xi)} \\
h_2^{(\ep,m)}(p; \xi) &:= \frac{\mathcal H_2^{(\ep,m)}(p; \xi)}{\mathcal H_0^{(\ep,m)}(p; \xi)} - \frac 12 \left(\frac{\mathcal H_1^{(\ep,m)}(p; \xi)}{\mathcal H_0^{(\ep,m)}(p; \xi)}\right)^2 \\
h_3^{(\ep,m)}(p; \xi) &:= \frac{\mathcal H_3^{(\ep,m)}(p; \xi)}{\mathcal H_0^{(\ep,m)}(p; \xi)} \\
h_4^{(\ep,m)}(p; \xi) &:= \frac{\mathcal H_4^{(\ep,m)}(p; \xi)}{\mathcal H_0^{(\ep,m)}(p; \xi)} - \frac{\mathcal H_1^{(\ep,m)}(p; \xi)}{\mathcal H_0^{(\ep,m)}(p; \xi)}
\frac{\mathcal H_3^{(\ep,m)}(p; \xi)}{\mathcal H_0^{(\ep,m)}(p; \xi)} \\
|r_n(p;\xi)| &\leq C(p;\xi) \left[|\epsilon_n|^3 + \frac{1}{n^2} + \frac{\epsilon_n^2}{n} \right]
\end{align*}
for some constant $C(p;\xi)$. Then, to find the maximum point of the likelihood, we study the equation $\frac{\ddr}{\ddr p} \ell_{n,T}(p;b) = 0$, which reads
\begin{align*}
& n \frac{p -\xi -\epsilon_n}{p(1-p)} = [\partial_p h_0^{(\ep,m)}(p; \xi)] + [\partial_p h_1^{(\ep,m)}(p; \xi)] \epsilon_n + [\partial_p h_2^{(\ep,m)}(p; \xi)] \epsilon_n^2 \\
&+ [\partial_p h_3^{(\ep,m)}(p; \xi)] \frac 1n + [\partial_p h_4^{(\ep,m)}(p; \xi)] \frac{\epsilon_n}{n} + r_n^{'}(p;\xi) \ . 
\end{align*}
The solution of such an equation can be obtained by inserting the expression $\xi + \epsilon_n + \Gamma_0(\xi) \frac{1}{n} + \Gamma_1(\xi) \frac{\epsilon_n}{n} + \Gamma_2(\xi) \frac{\epsilon_n^2}{n} + 
\Delta_n(\xi; \epsilon_n)$ in the place of $p$, and then expanding both members. For the left-hand side we get
\begin{align*}
\frac{1}{\xi(1-\xi)} &\left\{\Gamma_0(\xi) + \left(\Gamma_1(\xi) + \frac{2\xi-1}{\xi(1-\xi)}\Gamma_0(\xi)\right) \epsilon_n \right. \\
&\left. +  \left(\Gamma_2(\xi) + \frac{2\xi-1}{\xi(1-\xi)}\Gamma_1(\xi) + \frac{1-4\xi+4\xi^2}{\xi^2(1-\xi)^2}\Gamma_0(\xi)\right) \epsilon_n^2 + \omega_n\right\}
\end{align*}
where $\omega_n = O(|\epsilon_n|^3 + \frac 1n)$. On the other hand, a Taylor expansion around $p=\xi$ yields for the right-hand side 
\begin{align*}
&[\partial_p h_0^{(\ep,m)}(p; \xi)]_{|p=\xi} + \left([\partial_p^2 h_0^{(\ep,m)}(p; \xi)]_{|p=\xi} + [\partial_p h_1^{(\ep,m)}(p; \xi)]_{|p=\xi}\right) \epsilon_n \\
&+ \left(\frac 12 [\partial_p^3 h_0^{(\ep,m)}(p; \xi)]_{|p=\xi} + [\partial_p^2 h_1^{(\ep,m)}(p; \xi)]_{|p=\xi} + [\partial_p h_2^{(\ep,m)}(p; \xi)]_{|p=\xi}\right) \epsilon_n^2 + \omega_n\ .
\end{align*}
Here, it is important to notice that we have disregarded all the terms of $O(\frac 1n)$-type, considering them of lower order with respect to the terms of $O(\epsilon_n^2)$-type. This is due to the law of iterated 
logarithm, by which $\frac 1n A_n - p \sim \sqrt{(\log\log n)/n}$, when $A_n \sim Bin(n,p)$. Thus, requiring identity between the above expansions yields
\begin{align*}
\Gamma_0(\xi) &= \xi(1-\xi) [\partial_p h_0^{(\ep,m)}(p; \xi)]_{|p=\xi} \\
\Gamma_1(\xi) &= \xi(1-\xi) \left([\partial_p^2 h_0^{(\ep,m)}(p; \xi)]_{|p=\xi} + [\partial_p h_1^{(\ep,m)}(p; \xi)]_{|p=\xi}\right) + (1-2\xi) [\partial_p h_0^{(\ep,m)}(p; \xi)]_{|p=\xi} \\
\Gamma_2(\xi) &= \xi(1-\xi) \left(\frac 12 [\partial_p^3 h_0^{(\ep,m)}(p; \xi)]_{|p=\xi} + [\partial_p^2 h_1^{(\ep,m)}(p; \xi)]_{|p=\xi} + [\partial_p h_2^{(\ep,m)}(p; \xi)]_{|p=\xi}\right) \\
&- \frac{1-4\xi+4\xi^2}{\xi(1-\xi)}[\partial_p h_0^{(\ep,m)}(p; \xi)]_{|p=\xi} + (1-2\xi)\left([\partial_p^2 h_0^{(\ep,m)}(p; \xi)]_{|p=\xi} + [\partial_p h_1^{(\ep,m)}(p; \xi)]_{|p=\xi}\right)\\
&+ (1-2\xi)^2 [\partial_p h_0^{(\ep,m)}(p; \xi)]_{|p=\xi} \ .
\end{align*}
At this stage, to complete the proof it remains to evaluate the various terms $[\partial_p^j h_i^{(\ep,m)}(p; \xi)]_{|p=\xi}$. We start with the terms involving $h_0^{(\ep,m)}(p; \xi)$. First, we notice that
\begin{align*}
\partial_p h_0^{(\ep,m)}(p; \xi) &= \frac{\partial_p \mathcal H_0^{(\ep,m)}(p; \xi)}{\mathcal H_0^{(\ep,m)}(p; \xi)} \\
\partial_p^2 h_0^{(\ep,m)}(p; \xi) &=\frac{\mathcal H_0^{(\ep,m)}(p; \xi) \partial_p^2 \mathcal H_0^{(\ep,m)}(p; \xi) - [\partial_p \mathcal H_0^{(\ep,m)}(p; \xi)]^2}{[\mathcal H_0^{(\ep,m)}(p; \xi)]^2} \\
\partial_p^3 h_0^{(\ep,m)}(p; \xi) &= \frac{[\mathcal H_0^{(\ep,m)}(p; \xi)]^2 \partial_p^3 \mathcal H_0^{(\ep,m)}(p; \xi) + 2 [\partial_p \mathcal H_0^{(\ep,m)}(p; \xi)]^3}
{[\mathcal H_0^{(\ep,m)}(p; \xi)]^3} \\
&- \frac{3\partial_p \mathcal H_0^{(\ep,m)}(p; \xi)\partial_p^2 \mathcal H_0^{(\ep,m)}(p; \xi)}{[\mathcal H_0^{(\ep,m)}(p; \xi)]^2}
\end{align*}
After putting $\phi(p; \xi) := \log \frac{\xi(1-p)}{p(1-\xi)}$ and noticing that 
$$
\mathcal H_0^{(\ep,m)}(p; \xi) = \frac{1}{\cepsm} \sum_{l=-m}^m \exp\{-\ep |l| + l \phi(p; \xi)\}\ ,
$$
we get
\begin{align*}
\partial_p \mathcal H_0^{(\ep,m)}(p; \xi) &= \partial_p\phi(p; \xi) \frac{1}{\cepsm} \sum_{l=-m}^m l \exp\{-\ep |l| + l \phi(p; \xi)\} \\
\partial_p^2 \mathcal H_0^{(\ep,m)}(p; \xi) &= \partial_p^2\phi(p; \xi) \frac{1}{\cepsm} \sum_{l=-m}^m l \exp\{-\ep |l| + l \phi(p; \xi)\} \\
&+ [\partial_p\phi(p; \xi)]^2 \frac{1}{\cepsm} \sum_{l=-m}^m l^2 \exp\{-\ep |l| + l \phi(p; \xi)\}\\
\partial_p^3 \mathcal H_0^{(\ep,m)}(p; \xi) &= \partial_p^3\phi(p; \xi) \frac{1}{\cepsm} \sum_{l=-m}^m l \exp\{-\ep |l| + l \phi(p; \xi)\} \\
&+ 3\partial_p\phi(p; \xi) \partial_p^2\phi(p; \xi)\frac{1}{\cepsm} \sum_{l=-m}^m l^2 \exp\{-\ep |l| + l \phi(p; \xi)\} \\
&+ [\partial_p\phi(p; \xi)]^3 \frac{1}{\cepsm} \sum_{l=-m}^m l^3 \exp\{-\ep |l| + l \phi(p; \xi)\}\ .
\end{align*}
Now, when $p=\xi$, we have 
\begin{align*}
\phi(\xi; \xi) &= 0 \\
[\partial_p\phi(p; \xi)]_{|p=\xi} &= -\frac{1}{\xi(1-\xi)} \\
[\partial_p^2\phi(p; \xi)]_{|p=\xi} &= \frac{1-2\xi}{\xi^2(1-\xi)^2}
\end{align*}
and, hence,
\begin{align*}
\mathcal H_0^{(\ep,m)}(p; \xi)_{|p=\xi} &= 1 \\
[\partial_p \mathcal H_0^{(\ep,m)}(p; \xi)]_{|p=\xi} &= 0 \\
[\partial_p^2 \mathcal H_0^{(\ep,m)}(p; \xi)]_{|p=\xi} &= \frac{\mathsf{Var}(L)}{\xi^2(1-\xi)^2} \\
[\partial_p^3 \mathcal H_0^{(\ep,m)}(p; \xi)]_{|p=\xi} &= -3\frac{(1-2\xi)\mathsf{Var}(L)}{\xi^3(1-\xi)^3}\ . 
\end{align*}
Then, we consider the terms containing $h_1^{(\ep,m)}(p; \xi)$. We start from the identities
\begin{align*}
\partial_p h_1^{(\ep,m)}(p; \xi) &= \frac{\mathcal H_0^{(\ep,m)}(p; \xi) \partial_p \mathcal H_1^{(\ep,m)}(p; \xi) - \mathcal H_1^{(\ep,m)}(p; \xi) \partial_p \mathcal H_0^{(\ep,m)}(p; \xi)}
{[\mathcal H_0^{(\ep,m)}(p; \xi)]^2} \\
\partial_p^2 h_1^{(\ep,m)}(p; \xi) & = \frac{\mathcal H_0^{(\ep,m)}(p; \xi) \partial_p^2 \mathcal H_1^{(\ep,m)}(p; \xi) - \mathcal H_1^{(\ep,m)}(p; \xi) \partial_p^2 \mathcal H_0^{(\ep,m)}(p; \xi)}
{[\mathcal H_0^{(\ep,m)}(p; \xi)]^2} \\
&-2 \frac{\partial_p \mathcal H_0^{(\ep,m)}(p; \xi)[\mathcal H_0^{(\ep,m)}(p; \xi) \partial_p \mathcal H_1^{(\ep,m)}(p; \xi) - \mathcal H_1^{(\ep,m)}(p; \xi) \partial_p \mathcal H_0^{(\ep,m)}(p; \xi)]}
{[\mathcal H_0^{(\ep,m)}(p; \xi)]^3}
\end{align*}
and
\begin{align*}
\partial_p \mathcal H_1^{(\ep,m)}(p; \xi) &= \frac{\partial_p\phi(p; \xi)}{\xi(1-\xi)} \frac{1}{\cepsm} \sum_{l=-m}^m l^2 \exp\{-\ep |l| + l \phi(p; \xi)\} \\
\partial_p^2 \mathcal H_1^{(\ep,m)}(p; \xi) &= \frac{\partial_p^2\phi(p; \xi)}{\xi(1-\xi)} \frac{1}{\cepsm} \sum_{l=-m}^m l^2 \exp\{-\ep |l| + l \phi(p; \xi)\} \\
&+ \frac{[\partial_p\phi(p; \xi)]^2}{\xi(1-\xi)} \frac{1}{\cepsm} \sum_{l=-m}^m l^3 \exp\{-\ep |l| + l \phi(p; \xi)\}\ . 
\end{align*}
Now, when $p=\xi$, we have 
\begin{align*}
\mathcal H_1^{(\ep,m)}(p; \xi)_{|p=\xi} &= 0 \\
[\partial_p \mathcal H_1^{(\ep,m)}(p; \xi)]_{|p=\xi} &= -\frac{\mathsf{Var}(L)}{\xi^2(1-\xi)^2} \\
[\partial_p^2 \mathcal H_1^{(\ep,m)}(p; \xi)]_{|p=\xi} &= \frac{(1-2\xi)\mathsf{Var}(L)}{\xi^3(1-\xi)^3} \ .
\end{align*}
Finally, we consider the terms containing $h_2^{(\ep,m)}(p; \xi)$. We start from the identities
\begin{align*}
\partial_p h_2^{(\ep,m)}(p; \xi) &= \frac{\mathcal H_0^{(\ep,m)}(p; \xi) \partial_p \mathcal H_2^{(\ep,m)}(p; \xi) - \mathcal H_2^{(\ep,m)}(p; \xi) \partial_p \mathcal H_0^{(\ep,m)}(p; \xi)}
{[\mathcal H_0^{(\ep,m)}(p; \xi)]^2} \\
&- \frac{\mathcal H_1^{(\ep,m)}(p; \xi)[\mathcal H_0^{(\ep,m)}(p; \xi) \partial_p \mathcal H_1^{(\ep,m)}(p; \xi) -  \mathcal H_1^{(\ep,m)}(p; \xi) \partial_p \mathcal H_0^{(\ep,m)}(p; \xi)]}
{[\mathcal H_0^{(\ep,m)}(p; \xi)]^3}
\end{align*}
and
\begin{align*}
\partial_p \mathcal H_2^{(\ep,m)}(p; \xi) &= \frac{\partial_p\phi(p; \xi)}{2\xi^2(1-\xi)^2} \frac{1}{\cepsm} \sum_{l=-m}^m l^3 \exp\{-\ep |l| + l \phi(p; \xi)\} \\
&- \frac{(1-2\xi)\partial_p\phi(p; \xi)}{2\xi^2(1-\xi)^2} \frac{1}{\cepsm} \sum_{l=-m}^m l^2 \exp\{-\ep |l| + l \phi(p; \xi)\} \ .
\end{align*}
Now, when $p=\xi$, we have 
\begin{align*}
\mathcal H_2^{(\ep,m)}(p; \xi)_{|p=\xi} &= \frac{\mathsf{Var}(L)}{2\xi^2(1-\xi)^2} \\
[\partial_p \mathcal H_2^{(\ep,m)}(p; \xi)]_{|p=\xi} &= \frac{(1-2\xi)\mathsf{Var}(L)}{2\xi^3(1-\xi)^3}  \ .
\end{align*}
At the end of all these computation, we are in a position to conclude that $\Gamma_0(\xi) = \Gamma_1(\xi) = \Gamma_2(\xi) = 0$ proving \eqref{Expansion_MLE_true}.
\end{proof}

The way is now paved to analyze the LR relative to the null hypothesis $H_0 : p = p_0$, by means of the following
\begin{lem} \label{lem_gothic}
Under the same assumption of Lemma \ref{lemma_rho} with $\xi = p_0$, there hold
\begin{align}
\Lambda_{n,T}(p_0; b) &:= -2\log\left( \frac{L_{n,T}(p_0;b)}{L_{n,T}(\hat{p}_{n,T}; b)} \right) \nonumber \\
&= 2\log\left( \frac{L_n^{(0)}(\hat{p}_{n,T}; b)}{L_n^{(0)}(p_0; b)} \right) + 2\log\left( \frac{H_n^{(\ep,m)}(\hat{p}_{n,T}; b)}{H_n^{(\ep,m)}(p_0; b)} \right) \label{Dec_LambdanT}
\end{align}
along with
\begin{align}
2\log\left( \frac{L_n^{(0)}(\hat{p}_{n,T}; b)}{L_n^{(0)}(p_0; b)} \right) = &n\left\{\frac{1}{p_0(1-p_0)}\epsilon_n^2 +  \frac{2p_0 - 1}{3p_0^2(1-p_0)^2}\epsilon_n^3 \right. \label{Dec_Lambdan0} \\
&\left. + \frac{1 - 3p_0 + 3p_0^2}{6p_0^3(1-p_0)^3}\epsilon_n^4\right\} + O(\epsilon_n^5) \nonumber 
\end{align}
and
\begin{equation} \label{Dec_HnT}
2\log\left( \frac{H_n^{(\ep,m)}(\hat{p}_{n,T}; b)}{H_n^{(\ep,m)}(p_0; b)} \right) = -\frac{\mathsf{Var}(L)}{p_0^2(1-p_0)^2}  \epsilon_n^2 + \frac{(1-2p_0)\mathsf{Var}(L)}{p_0^2(1-p_0)^2}  \frac{\epsilon_n}{n}
+ O(\epsilon_n^3)\ .
\end{equation}
Therefore, putting $\zeta_n := \sqrt{n} \epsilon_n$, we get
\begin{equation} \label{final_expansion_LambdanT}
\Lambda_{n,T}(p_0; b) = \mathfrak{A}_n^{(\ep,m)}(p_0) \zeta_n^2 + \mathfrak{B}_n^{(\ep,m)}(p_0) \frac{\zeta_n^3}{\sqrt{n}} + \mathfrak{C}_n^{(\ep,m)}(p_0) \frac{\zeta_n^4}{n} + O(n^{-3/2})
\end{equation}
with
\begin{align}
\mathfrak{A}_n^{(\ep,m)}(p_0) &= \frac{1}{p_0(1-p_0)} - \frac{\mathsf{Var}(L)}{np_0^2(1-p_0)^2} + O\left(\frac{1}{n^2}\right) \label{A_gothic} \\
\mathfrak{B}_n^{(\ep,m)}(p_0) &= \frac{2p_0 - 1}{3p_0^2(1-p_0)^2} + \frac{\mathfrak{B}_{\ast}^{(\ep,m)}(p_0)}{n} + O\left(\frac{1}{n^2}\right) \label{B_gothic} \\
\mathfrak{C}_n^{(\ep,m)}(p_0) &= \frac{1 - 3p_0 + 3p_0^2}{6p_0^3(1-p_0)^3} + \frac{\mathfrak{C}_{\ast}^{(\ep,m)}(p_0)}{n} + O\left(\frac{1}{n^2}\right) \label{C_gothic}
\end{align}
for suitable terms $\mathfrak{B}_{\ast}^{(\ep,m)}(p_0)$ and $\mathfrak{C}_{\ast}^{(\ep,m)}(p_0)$.
\end{lem}

\begin{remark} \label{rmk:Taylor}
Notice that the expression inside the brackets in \eqref{Dec_Lambdan0} coincides with the Taylor polynomial of order 4 of the map $\epsilon_n \mapsto \mathcal D_{KL}(p_0 + \epsilon_n\| p_0)$ where
$$
\mathcal D_{KL}(p\| p_0) := p \log\left(\frac{p}{p_0}\right) + (1-p) \log\left(\frac{1-p}{1-p_0}\right)
$$ 
denotes the Kullback-Leibler divergence relative to the Bernoulli model. In particular, we have that
$$
\frac{1}{p_0(1-p_0)} = \frac{\partial^2}{\partial p^2} \mathcal D_{KL}(p\| p_0)_{| p = p_0}
$$
coincides with the Fisher information of the Bernoulli model. Finally, it is worth noticing that, in the Bernoulli model, the Fisher information just coincides with the inverse of the variance.
\end{remark}

\begin{proof}
Since 
$$
\Lambda_{n,T}(p_0; b) := -2\log\left(\frac{\sup_{p\in \{p_0\}} L_{n,T}(p;b)}{\sup_{p\in [0,1]} L_{n,T}(p; b)} \right) = -2\log\left( \frac{L_{n,T}(p_0;b)}{L_{n,T}(\hat{p}_{n,T}; b)} \right)
$$ 
holds by definition, identity \eqref{Dec_LambdanT} follows immediately from \eqref{LikTrueDec}. Next, we derive \eqref{Dec_Lambdan0} by combining \eqref{Definition_Ln0} with 
\eqref{Expansion_MLE_true}. In fact, we have
\begin{align*}
2\log\left( \frac{L_n^{(0)}(\hat{p}_{n,T}; b)}{L_n^{(0)}(p_0; b)} \right) = 2n &\left\{(p_0 + \epsilon_n) \log\left(1 + \frac{\epsilon_n}{p_0} + \frac{\Delta_n(p_0)}{p_0}\right) \right. \\
& \left. + (1 - p_0 - \epsilon_n) \log\left(1 - \frac{\epsilon_n}{1-p_0} - \frac{\Delta_n(p_0)}{1- p_0}\right)  \right\}
\end{align*}
where $\Delta_n(p_0) = \frac 1n[\Gamma_3(p_0) \epsilon_n^3 + \Gamma_4(p_0) \epsilon_n^4 + O(\epsilon_n^5)]$, according to \eqref{Expansion_MLE_true}. Therefore, using the Taylor expansion of the function 
$z \mapsto \log(1+z)$, we can further specialize the last expression as
\begin{align*}
2n &\left\{ (p_0 + \epsilon_n) \left[ \frac{\epsilon_n + \Delta_n(p_0)}{p_0} - \frac{(\epsilon_n + \Delta_n(p_0))^2}{2p_0^2} + \frac{(\epsilon_n + \Delta_n(p_0))^3}{3p_0^3} 
- \frac{(\epsilon_n + \Delta_n(p_0))^4}{4p_0^4} \right] \right. \\
& \left. - (1 - p_0 - \epsilon_n)\left[\frac{\epsilon_n + \Delta_n(p_0)}{1-p_0} + \frac{(\epsilon_n + \Delta_n(p_0))^2}{2(1-p_0)^2} + \frac{(\epsilon_n + \Delta_n(p_0))^3}{3(1-p_0)^3} 
+ \frac{(\epsilon_n + \Delta_n(p_0))^4}{4(1-p_0)^4} \right] \right\} \\
&+ O(\epsilon_n^5) \ . 
\end{align*}
At this stage, careful algebraic computations based on the Newton binomial formula lead to \eqref{Dec_Lambdan0}. Incidentally, it is interesting to notice that the terms $\Gamma_3(p_0)$ and $\Gamma_4(p_0)$,
depending on the perturbation, do not appear in the expansion \eqref{Dec_Lambdan0}.

Coming to \eqref{Dec_HnT}, we start by noticing that
$$
\frac{1-\hat{p}_{n,T}}{\hat{p}_{n,T}} = \frac{1-p_0}{p_0} \left[1 - \frac{\epsilon_n}{p_0(1-p_0)}  + \frac{\epsilon_n^2}{p_0^2(1-p_0)} + O(\epsilon_n^3)\right]
$$
which entails
\begin{align*}
\left(\frac{1-\hat{p}_{n,T}}{\hat{p}_{n,T}}\right)^l = \left(\frac{1-p_0}{p_0}\right)^l &\left[1 - \frac{l}{p_0(1-p_0)}\epsilon_n  + \frac{l}{p_0^2(1-p_0)} \epsilon_n^2 \right. \\
&\left. + \frac{l(l-1)}{2p_0^2(1-p_0)^2} \epsilon_n^2 + O(\epsilon_n^3)\right]\ .
\end{align*}
This identity, combined with \eqref{Definition_HnT} and \eqref{exp_rhon}, yields
$$
H_n^{(\ep,m)}(\hat{p}_{n,T}; b) = 1 - \frac{\mathsf{Var}(L)}{2np_0(1-p_0)} + \frac{(1-2p_0)\mathsf{Var}(L)}{2p_0^2(1-p_0)^2}\frac{\epsilon_n}{n} + O(\epsilon_n^3)\ .
$$ 
Moreover, another combination of \eqref{Definition_HnT} and \eqref{exp_rhon} gives
$$
H_n^{(\ep,m)}(p_0; b) = 1 + \frac{\mathsf{Var}(L)}{2p_0^2(1-p_0)^2} \epsilon_n^2 - \frac{\mathsf{Var}(L)}{2np_0(1-p_0)} + \frac{(1-2p_0)\mathsf{Var}(L)}{p_0^2(1-p_0)^2}\frac{\epsilon_n}{n} + O(\epsilon_n^3)\ .
$$ 
Then, \eqref{Dec_HnT} is a straightforward consequence of the last two identities.

Finally, identities \eqref{final_expansion_LambdanT}-\eqref{A_gothic}-\eqref{B_gothic}-\eqref{C_gothic} follows immediately from \eqref{Dec_LambdanT}-\eqref{Dec_Lambdan0}-\eqref{Dec_HnT}.
\end{proof}

The last preparatory result provides an equivalent reformulation of the event 
\begin{equation} \label{event_reject_H0}
\{b : H_0\ \text{is\ rejected}\} = \{\Lambda_{n,T}(p_0; b) > \lambda_{\alpha}\}
\end{equation}
for some $\lambda_{\alpha} > 0$, to be determined after assessing the level $\alpha$ of the test.  

\begin{lem}
Under the same assumption of Lemma \ref{lemma_rho} with $\xi = p_0$, the event \eqref{event_reject_H0} is equivalent to the event
\begin{equation} \label{event_reject_H0_Z}
\{b : \zeta_n > \mathfrak{Z}_{n,+}^{(\ep,m)}(p_0; \lambda_{\alpha})\} \cup \{b : \zeta_n < \mathfrak{Z}_{n,-}^{(\ep,m)}(p_0; \lambda_{\alpha})\} 
\end{equation}
where
\begin{align}
\mathfrak{Z}_{n,+}^{(\ep,m)}(p_0; \lambda_{\alpha}) &:= \sqrt{\lambda_{\alpha}p_0(1-p_0)} + \frac{(1 - 2p_0) \lambda_{\alpha}}{6\sqrt{n}} \nonumber \\
&+ \frac{\sqrt{\lambda_{\alpha}} }{2n\sqrt{p_0(1-p_0)}} \left\{\mathsf{Var}(L) - \lambda_{\alpha}(1 + 2p_0 - 2p_0^2)\right\}
+O(n^{-3/2}) \label{Z+_gothic}  \\
\mathfrak{Z}_{n,-}^{(\ep,m)}(p_0; \lambda_{\alpha}) &:= -\sqrt{\lambda_{\alpha}p_0(1-p_0)} + \frac{(1 - 2p_0) \lambda_{\alpha}}{6\sqrt{n}} \nonumber \\
&- \frac{\sqrt{\lambda_{\alpha}} }{2n\sqrt{p_0(1-p_0)}} \left\{\mathsf{Var}(L) - \lambda_{\alpha}(1 + 2p_0 - 2p_0^2)\right\} +O(n^{-3/2}) \label{Z-_gothic} \ .
\end{align}
\end{lem}

\begin{proof}
Since $\mathfrak{A}_n^{(\ep,m)}(p_0) > 0$ and $\mathfrak{C}_n^{(\ep,m)}(p_0) > 0$ eventually, we first notice that, always eventually, the graphic of the function
$$
\Xi_n : x \mapsto \mathfrak{A}_n^{(\ep,m)}(p_0) x^2 + \mathfrak{B}_n^{(\ep,m)}(p_0) \frac{x^3}{\sqrt{n}} + \mathfrak{C}_n^{(\ep,m)}(p_0) \frac{x^4}{n}
$$
goes to $+\infty$ as $x \to \pm \infty$, decreases for $x<0$, increases for $x>0$, and has a unique absolute minimum at $x=0$. Thus, for any $\lambda_{\alpha} > 0$ (to be determined later on), the equation
$\Xi_n(x) = \lambda_{\alpha}$ admits two real solutions, that we just denote as $\mathfrak{Z}_{n,+}^{(\ep,m)}(p_0; \alpha)$ and $\mathfrak{Z}_{n,-}^{(\ep,m)}(p_0; \alpha)$. To evaluate them, we put
\begin{align*}
\mathfrak{Z}_{n,+}^{(\ep,m)}(p_0; \lambda_{\alpha}) &:= \sqrt{\frac{\lambda_{\alpha}}{\mathfrak{A}_n^{(\ep,m)}(p_0)}} + \frac{\xi_{+}^{(\ep,m)}(p_0;\lambda_{\alpha})}{\sqrt{n}} + 
\frac{\eta_{+}^{(\ep,m)}(p_0; \lambda_{\alpha})}{n} + O(n^{-3/2})\\
\mathfrak{Z}_{n,-}^{(\ep,m)}(p_0;\lambda_{\alpha}) &:= -\sqrt{\frac{\lambda_{\alpha}}{\mathfrak{A}_n^{(\ep,m)}(p_0)}} + \frac{\xi_{-}^{(\ep,m)}(p_0; \lambda_{\alpha})}{\sqrt{n}} + 
\frac{\eta_{-}^{(\ep,m)}(p_0; \lambda_{\alpha})}{n} + O(n^{-3/2})\ .
\end{align*}
After inserting these expression into the function $\Xi_n$, we set equal to zero the coefficients of $\frac{1}{\sqrt{n}}$ and $\frac{1}{n}$, obtaining
\begin{align*}
\xi_{+}^{(\ep,m)}(p_0; \lambda_{\alpha}) = \xi_{-}^{(\ep,m)}(p_0; \alpha) &= -\frac{\mathfrak{B}_n^{(\ep,m)}(p_0) \lambda_{\alpha}}{2 [\mathfrak{A}_n^{(\ep,m)}(p_0)]^2} \\
\eta_{+}^{(\ep,m)}(p_0; \lambda_{\alpha}) = -\eta_{-}^{(\ep,m)}(p_0; \alpha) 
&= \frac{\lambda_{\alpha}^{3/2}}{2 [\mathfrak{A}_n^{(\ep,m)}(p_0)]^{5/2}}\left[\frac{5[\mathfrak{B}_n^{(\ep,m)}(p_0)]^2}{4\mathfrak{A}_n^{(\ep,m)}(p_0)} - 
\mathfrak{C}_n^{(\ep,m)}(p_0)\right]\ .
\end{align*}
Finally, to prove \eqref{Z+_gothic}-\eqref{Z-_gothic}, it is enough to use \eqref{A_gothic}-\eqref{B_gothic}-\eqref{C_gothic} of Lemma \ref{lem_gothic}, showing that
$$
\sqrt{\frac{\lambda_{\alpha}}{\mathfrak{A}_n^{(\ep,m)}(p_0)}} =  \sqrt{\lambda_{\alpha}p_0(1-p_0)} \left(1 + \frac{\mathsf{Var}(L)}{2np_0(1-p_0)}\right) + O\left(\frac{1}{n^2}\right) 
$$
and
\begin{align*}
-\frac{\mathfrak{B}_n^{(\ep,m)}(p_0) \lambda_{\alpha}}{2 [\mathfrak{A}_n^{(\ep,m)}(p_0)]^2} &= -\frac{\lambda_{\alpha}}{2} [p_0(1-p_0)]^2 \frac{2p_0 - 1}{3p_0^2(1-p_0)^2}\\
\frac{\lambda_{\alpha}^{3/2}}{2 [\mathfrak{A}_n^{(\ep,m)}(p_0)]^{5/2}}\left[\frac{5[\mathfrak{B}_n^{(\ep,m)}(p_0)]^2}{4\mathfrak{A}_n^{(\ep,m)}(p_0)} - 
\mathfrak{C}_n^{(\ep,m)}(p_0)\right] &= \frac{\lambda_{\alpha}^{3/2}}{2} [p_0(1-p_0)]^{5/2} \left[ \frac{5}{4}p_0(1-p_0) \times \right. \\
&\left. \times \left(\frac{2p_0 - 1}{3p_0^2(1-p_0)^2}\right)^2 - \frac{1 - 3p_0 + 3p_0^2}{6p_0^3(1-p_0)^3}\right]\ , 
\end{align*}
which concludes the proof.
\end{proof}

%%%%%%%%%%%%%%%%%%%%%%%%%%%%%%%%%%%%%%%%%%%%%%%%%%%%%%%%%%%%%%%%%%%%%%%
%%%%%%%%%%%%%%%%%%%%%%%%%%%%%%%%%%%%%%%%%%%%%%%%%%%%%%%%%%%%%%%%%%%%%%%%%
%%%%%%%%%%%%%%%%%%%%%%%%%%%%%%%%%%%%%%%%%%%%%%%%%%%%%%%%%%%%%%%%%%%%%%%%

\subsection{Proof of Theorem \ref{teo31}} \label{sect:proof_teo31}
We start again from \eqref{ApiuL} where the random variables $A_{n}$ and $L$, defined on the probability space $(\Omega,\mathscr{F},\text{Pr})$, are independent, $A_n \sim Bin(n,p_0)$ and $L$ has the Laplace distribution \eqref{Laplace}. Thus, under the validity of $H_0$, we define the event $E_{0,n}\subset\Omega$ as
\begin{equation}\label{event}
E_{0,n}:=\left\{\omega\in\Omega\text{ $:$ }\left|\frac{A_{n}(\omega)}{n}-p_0\right|\leq c(p_0)\left(\frac{\log n}{n}\right)^{1/2}\right\}
\end{equation}
for some $c(p_0)>0$. By a standard large deviation argument (see, e.g., Theorem 1 of \citet{Oka(59)}), there holds
$$
\mathrm{Pr}[E_{0,n}]\geq 1-2n^{-\frac{c(p_0)^{2}}{2}},
$$
which implies that the probability $E_{0,n}^c$ can be made arbitrarily small after a suitable choice of $c(p_0)$. See also Proposition 2.1 of \citet{DoRe(19)} for more refined bounds.
In particular, it can be chosen so that
\begin{equation}\label{bound1}
\mathrm{Pr}[E_{0,n}^c] \leq 2n^{-3/2}.
\end{equation}
In this way, for $\ell=T,N$, we can write
$$
\F_{n,\ell}(t)= \text{Pr}[\Lambda_{n,\ell}(p_{0}) \leq t, E_{0,n}] + \text{Pr}[\Lambda_{n,\ell}(p_{0}) \leq t, E_{0,n}^c]
$$
with
$$
\mathrm{Pr}[\Lambda_{n,\ell}(p_{0}) \leq t, E_{0,n}^c] \leq 2n^{-3/2}.
$$
The advantage of such a preliminary step is that, on $E_{0,n}$, the assumptions of all the Lemmata stated in the previous subsection are fulfilled and, by resorting to \eqref{event_reject_H0_Z}, we can write
$$
\left\{\Lambda_{n,T}(p_{0}) \leq t, E_{0,n}\right\} = \left\{\mathfrak{Z}_{n,-}^{(\ep,m)}(p_0; t) \leq Z_n \leq \mathfrak{Z}_{n,+}^{(\ep,m)}(p_0; t), E_{0,n}\right\}
$$
where
\begin{equation} \label{eq:Vn}
Z_n := \frac{B_n - np_0}{\sqrt{n}} = \sqrt{p_0(1-p_0)} W_n + \frac{1}{\sqrt{n}} L
\end{equation}
with
$$
W_n := \frac{A_n - np_0}{\sqrt{np_0(1-p_0)}}\ . 
$$

\begin{remark}\label{rmk:pezza}
This very same argument, based on straightforward application of concentration inequalities for the Binomial distribution, makes clear that the additional randomness that emerges in the ``na\"ive" model, due to the presence of the random sample size $n^{+}$, is irrelevant in our asymptotic analysis. Indeed, after fixing $p_0 \in (0,1)$, we get that the probability that $B_n \neq B_n^{+}$ (that is, the first cell contains a negative value) is bounded by a term of the form $e^{-\tau n}$ for some $\tau > 0$. Since terms of this form are irrelevant in the Edgeworth expansion \eqref{exp_n}, we can work from the beginning under the assumption that $B_n = B_n^{+}$ (that is, the first cell does not contain a negative value) which entails, in particular, that $n^{+} = n$. An analogous consideration holds in larger dimensions, that is when $k > 2$, if $\mathbf{p}_0$ is fixed in the interior of $\Delta_{k-1}$.
\end{remark}

Therefore, in view of \eqref{bound1}, we can conclude that
\begin{align}
\F_{n,T}(t) &= \mathrm{Pr}\left[\mathfrak{Z}_{n,-}^{(\ep,m)}(p_0; t) \leq Z_n \leq \mathfrak{Z}_{n,+}^{(\ep,m)}(p_0; t)\right] + O(n^{-3/2})\nonumber \\
&= \G_n\left(\mathfrak{Z}_{n,+}^{(\ep,m)}(p_0; t)\right) - \lim_{\tau \to \mathfrak{Z}_{n,-}^{(\ep,m)}(p_0; t)^-} \G_n(\tau) + O(n^{-3/2}) \label{FDR_true}
\end{align}
where $\G_n$ denotes the distribution function of $Z_n$. Moreover, the same analysis developed in the previous section shows that  
\begin{align}
\F_{n,N}(t) &= \mathrm{Pr}\left[\mathfrak{V}_{n,-}^{(\ep,m)}(p_0; t) \leq Z_n \leq \mathfrak{V}_{n,+}^{(\ep,m)}(p_0; t)\right] + O(n^{-3/2})\nonumber \\
&= \G_n\left(\mathfrak{V}_{n,+}^{(\ep,m)}(p_0; t)\right) - \lim_{\tau \to \mathfrak{V}_{n,-}^{(\ep,m)}(p_0; t)^-} \G_n(\tau) + O(n^{-3/2}) \label{FDR_naive}
\end{align}
with
\begin{align}
\mathfrak{V}_{n,+}^{(\ep,m)}(p_0; \lambda_{\alpha}) &:= \sqrt{\lambda_{\alpha}p_0(1-p_0)} + \frac{(1 - 2p_0) \lambda_{\alpha}}{6\sqrt{n}} \nonumber \\
&- \frac{\lambda_{\alpha}^{3/2} (1 + 2p_0 - 2p_0^2)}{2n\sqrt{p_0(1-p_0)}} +O(n^{-3/2}) \label{V+_gothic}  \\
\mathfrak{V}_{n,-}^{(\ep,m)}(p_0; \lambda_{\alpha}) &:= -\sqrt{\lambda_{\alpha}p_0(1-p_0)} + \frac{(1 - 2p_0) \lambda_{\alpha}}{6\sqrt{n}} \nonumber \\
&+ \frac{\lambda_{\alpha}^{3/2} (1 + 2p_0 - 2p_0^2)}{2n\sqrt{p_0(1-p_0)}} +O(n^{-3/2}) \label{V-_gothic} \ .
\end{align}
Therefore, it is evident from the previous identities that the proof of Theorem \ref{teo31} can be carried out after providing an explicit expansion of Berry-Esseen type of  $\G_n$, 
which is contained in the following
\begin{proposition} \label{prp:BE}
Let $\Phi$ denote the cumulative distribution function of the standard Normal distribution, and let, for any $m \in \N$, 
$$
S_{2m+1}(x) := 2\sum_{k=1}^{\infty} \frac{\sin(2\pi kx)}{(2\pi k)^{2m+1}} \quad \text{and} \quad S_{2m}(x) := 2\sum_{k=1}^{\infty} \frac{\cos(2\pi kx)}{(2\pi k)^{2m}}\ .
$$ 
In addition, for any $j \in \N$, let $Q_j$ denote the function
$$
Q_j(x) := -\frac{1}{\sqrt{2\pi}} e^{-\frac{x^2}{2}} \sum_{(\ast)} H_{j+2s-1}(x) \prod_{m=1}^j \frac{1}{k_m!} \left(\frac{\gamma_m}{(m+2)! p^{\frac{m}{2}+1} (1-p)^{\frac{m}{2}+1}}\right)^{k_m}
$$
where $(\ast)$ means that the sum ranges over all the $j^{th}$-uples $(k_1, \dots, k_j) \in \N_0^j$ such that $k_1 + 2k_2 + \dots + jk_j= j$, with $s = k_1 + k_2 + \dots + k_j$, $H_m$ denotes the $m^{th}$ Chebyshev-Hermite polynomial, i.e. 
$$
H_m(x) := (-1)^m e^{\frac{x^2}{2}} \left(\frac{\ddr^m}{\ddr x^m} e^{-\frac{x^2}{2}} \right)
$$
and $\gamma_m$ stands for the $m^{th}$ cumulant of the Bernoulli distribution of parameter $p$. Then, for any $x \in \R$, there holds
\begin{align*}
\G_n(x) &= \Phi\left(\frac{x}{\sqrt{p_0(1-p_0)}}\right) + \sum_{j=1}^2 \left(\frac{1}{\sqrt{n}}\right)^j Q_j\left(\frac{x}{\sqrt{p_0(1-p_0)}}\right) \\
&+ \frac{\mathsf{Var}(L)}{2np_0(1-p_0)} \Phi''\left(\frac{x}{\sqrt{p_0(1-p_0)}}\right) \\
&+ \frac{1}{\sqrt{np_0(1-p_0)}} S_1(np_0 + \sqrt{n}x) \left[ \Phi'\left(\frac{x}{\sqrt{p_0(1-p_0)}}\right) + \frac{1}{\sqrt{n}} Q'_1\left(\frac{x}{\sqrt{p_0(1-p_0)}}\right) \right] \\
&+ \frac{1}{np_0(1-p_0)} S_2(np_0 + \sqrt{n}x) \Phi''\left(\frac{x}{\sqrt{p_0(1-p_0)}}\right) + R_n(x)
\end{align*}
with the remainder term $R_n$ satisfying an inequality like 
$$
\sup_{x \in \R} \{ |R_n(x)| (1+|x|)^5\} \leq \frac{C(p_0, m, \varepsilon)}{n^{3/2}}
$$ 
for some suitable constant $C(p_0, m, \varepsilon)$ independent of $n$. 
\end{proposition}

\begin{proof}
After denoting by $\F_n$ the distribution function of $W_n$, invoke well-known results (see, e.g., Theorem 6 in Chapter VI of Petrov \citet{Pet(75)}, or \citet{DoFa(20)}) 
to write (in the same notation adopted by Petrov)
\begin{align*}
\F_n(y) &= U_5(y) + \frac{1}{\sqrt{np_0(1-p_0)}} S_1(np_0 + y\sqrt{np_0(1-p_0)}) U'_5(y) \\
&+ \frac{1}{np_0(1-p_0)} S_2(np_0 + y\sqrt{np_0(1-p_0)}) U''_5(y) + T_n(y)
\end{align*}
where $U_5(y) := \Phi(y) + \sum_{j=1}^3 \left(\frac{1}{\sqrt{n}}\right)^j Q_j(y)$ and 
$$
|T_n(y)| \leq \frac{J(p_0)}{n^{3/2}(1+|y|)^5}
$$
for some positive constant $J(p_0)$. Now, in view of the definition of $V_n$, its distribution function $\G_n$ is given by
$$
\G_n(x) = \sum_{l = -m}^m \frac{1}{c_{\varepsilon,m}} e^{-\epsilon |l|}\ \F_n\left(\frac{x - l/\sqrt{n}}{\sqrt{p_0(1-p_0)}} \right) \ .
$$
To obtain an expansion for $\G_n$, notice that, for $y = \frac{x - l/\sqrt{n}}{\sqrt{p_0(1-p_0)}}$, there holds
$$
S_j(np_0 + y\sqrt{np_0(1-p_0)}) = S_j(np_0 + x\sqrt{n} - l) = S_j(np_0 + x\sqrt{n})
$$
thanks to the periodic character of $S_j$. Hence, the effect of putting $y = \frac{x - l/\sqrt{n}}{\sqrt{p_0(1-p_0)}}$ is mainly observed in the terms $U_5$, $U'_5$, and $U''_5$: by resorting to the Taylor expansion, for any $k \in \N_0$, write
\begin{align*}
U^{(k)}_5\left(\frac{x - l/\sqrt{n}}{\sqrt{p_0(1-p_0)}}\right) &= U^{(k)}_5\left(\frac{x}{\sqrt{p_0(1-p_0)}}\right) -\frac{l}{\sqrt{np_0(1-p_0)}} U^{(k+1)}_5\left(\frac{x}{\sqrt{p_0(1-p_0)}}\right) \\
&+ \frac{l^2}{2np_0(1-p_0)} U^{(k+2)}_5\left(\frac{x}{\sqrt{p_0(1-p_0)}}\right) + T_{n,k}(x)
\end{align*}
where the remainder term satisfies
$$
|T_{n,k}(x)| \leq \frac{J_k(p_0) |l|^3}{n^{3/2}(1+|x|)^5}
$$
for some positive constant $J_k(p_0)$. Thus, exploiting the symmetry of the Laplace distribution, obtain
\begin{align*}
\sum_{l = -m}^m \frac{1}{c_{\varepsilon,m}} e^{-\epsilon |l|}\ U^{(k)}_5\left(\frac{x - l/\sqrt{n}}{\sqrt{p_0(1-p_0)}}\right) &= U^{(k)}_5\left(\frac{x}{\sqrt{p_0(1-p_0)}}\right) \\
&+ \frac{\mathsf{Var}(L)}{2np_0(1-p_0)} U^{(k+2)}_5\left(\frac{x}{\sqrt{p_0(1-p_0)}}\right) + T_{n,k}^{\ast}(x)
\end{align*}
where the remainder term satisfies
$$
|T_{n,k}^{\ast}(x)| \leq \frac{C_k(p_0, m, \varepsilon)}{n^{3/2}(1+|x|)^5}
$$
for some suitable constant $C_k(p_0, m, \varepsilon)$ independent of $n$. Finally, gathering the previous identities, conclude that
\begin{align*}
\G_n(x) &= U_5\left(\frac{x}{\sqrt{p_0(1-p_0)}}\right) + \frac{\mathsf{Var}(L)}{2np_0(1-p_0)} U^{''}_5\left(\frac{x}{\sqrt{p_0(1-p_0)}}\right) + T_{n,0}^{\ast}(x) \\
& + \frac{1}{\sqrt{np_0(1-p_0)}} S_1(np_0 + x\sqrt{n}) \times \Big[ U^{'}_5\left(\frac{x}{\sqrt{p_0(1-p_0)}}\right) \\
&+ \frac{\mathsf{Var}(L)}{2np_0(1-p_0)} U^{'''}_5\left(\frac{x}{\sqrt{p_0(1-p_0)}}\right) + T_{n,1}^{\ast}(x) \Big] \\
&+ \frac{1}{np_0(1-p_0)} S_2(np_0 + x\sqrt{n}) \times  \Big[ U^{''}_5\left(\frac{x}{\sqrt{p_0(1-p_0)}}\right) \\
&+ \frac{\mathsf{Var}(L)}{2np_0(1-p_0)} U^{''''}_5\left(\frac{x}{\sqrt{p_0(1-p_0)}}\right) + T_{n,2}^{\ast}(x) \Big] \\
&+ \sum_{l = -m}^m \frac{1}{c_{\varepsilon,m}} e^{-\epsilon |l|}\ T_n\left(\frac{x - l/\sqrt{n}}{\sqrt{p_0(1-p_0)}} \right) \ .
\end{align*}
Incorporating all the terms of type $O(n^{-3/2})$ in the remainder yields the thesis to be proved.
\end{proof}

The way is now paved to complete the proof of Theorem \ref{teo31}. Indeed, \eqref{exp_t} follows by combining the thesis of Proposition \ref{prp:BE} with \eqref{FDR_true} and \eqref{Z+_gothic}-\eqref{Z-_gothic}.
To see this, we start from the analysis of the quantity
\begin{align*}
&\Phi\left(\frac{\mathfrak{Z}_{n,+}^{(\ep,m)}(p_0; t)}{\sqrt{p_0(1-p_0)}}\right) + \sum_{j=1}^2 \left(\frac{1}{\sqrt{n}}\right)^j Q_j\left(\frac{\mathfrak{Z}_{n,+}^{(\ep,m)}(p_0; t)}{\sqrt{p_0(1-p_0)}}\right) \\
&+ \frac{\mathsf{Var}(L)}{2np_0(1-p_0)} \Phi''\left(\frac{\mathfrak{Z}_{n,+}^{(\ep,m)}(p_0; t)}{\sqrt{p_0(1-p_0)}}\right) - \Phi\left(\frac{\mathfrak{Z}_{n,-}^{(\ep,m)}(p_0; t)}{\sqrt{p_0(1-p_0)}}\right) \\
&- \sum_{j=1}^2 \left(\frac{1}{\sqrt{n}}\right)^j Q_j\left(\frac{\mathfrak{Z}_{n,-}^{(\ep,m)}(p_0; t)}{\sqrt{p_0(1-p_0)}}\right) - \frac{\mathsf{Var}(L)}{2np_0(1-p_0)} 
\Phi''\left(\frac{\mathfrak{Z}_{n,-}^{(\ep,m)}(p_0; t)}{\sqrt{p_0(1-p_0)}}\right)
\end{align*}
which represents the ``regular'' part of the expression of \eqref{FDR_true}. Taking account of \eqref{Z+_gothic}-\eqref{Z-_gothic}, a straightforward Taylor expansion shows that the above quantity is equal to
\begin{align*}
&\Phi(\sqrt{t}) + \Phi'(\sqrt{t}) \left[\frac{(1 - 2p_0) t}{6\sqrt{np_0(1-p_0)}} + \frac{\sqrt{t} }{2np_0(1-p_0)} \left\{\mathsf{Var}(L) - t(1 + 2p_0 - 2p_0^2)\right\}\right] \\
&+ \frac 12 \Phi''(\sqrt{t}) \frac{(1 - 2p_0)^2 t^2}{36np_0(1-p_0)} + Q_1(\sqrt{t}) \frac{1}{\sqrt{n}} + Q'_1(\sqrt{t}) \frac{(1 - 2p_0) t}{6n\sqrt{p_0(1-p_0)}} + Q_2(\sqrt{t}) \frac{1}{n}\\
&+ \frac{\mathsf{Var}(L)}{2np_0(1-p_0)} \Phi''(\sqrt{t}) - \Phi(-\sqrt{t}) - \Phi'(-\sqrt{t}) \left[\frac{(1 - 2p_0) t}{6\sqrt{np_0(1-p_0)}} \right. \\
&\left. - \frac{\sqrt{t} }{2np_0(1-p_0)} \left\{\mathsf{Var}(L) - t(1 + 2p_0 - 2p_0^2)\right\}\right] - \frac 12 \Phi''(-\sqrt{t}) \frac{(1 - 2p_0)^2 t^2}{36np_0(1-p_0)} \\
&- Q_1(-\sqrt{t}) \frac{1}{\sqrt{n}} - Q'_1(-\sqrt{t}) \frac{(1 - 2p_0) t}{6n\sqrt{p_0(1-p_0)}} - Q_2(-\sqrt{t}) \frac{1}{n}
- \frac{\mathsf{Var}(L)}{2np_0(1-p_0)} \Phi''(-\sqrt{t}) \ .
\end{align*}
Exploiting that $x\Phi'(x) + \Phi''(x) = 0$, we show that the terms containing $\mathsf{Var}(L)$ cancel out. Moreover, $\Phi(\sqrt{t}) - \Phi(-\sqrt{t}) = \K(t)$ and $\Phi'(\sqrt{t}) = \Phi'(-\sqrt{t})$
hold for any $t \geq 0$. These considerations lead to the following equivalent reformulation of the above term
$$
\K(t) - \Phi'(\sqrt{t}) \frac{t^{3/2} (1 + 2p_0 - 2p_0^2)}{np_0(1-p_0)} + \Phi''(\sqrt{t}) \frac{(1 - 2p_0)^2 t^2}{36np_0(1-p_0)} + Q_1(\sqrt{t}) \frac{2}{\sqrt{n}}\ .
$$
This is the main part of the proof. Then, we deal with the ``irregular'' part of the expression of \eqref{FDR_true}. The argument is essentially the same as above, even if we have now to
take care of the expressions
\begin{align*}
S_i\left(np_0 + \sqrt{n}\mathfrak{Z}_{n,\pm}^{(\ep,m)}(p_0; t)\right) &= S_i\left(np_0 \pm \sqrt{n t p_0(1-p_0)} + \frac{(1 - 2p_0) t}{6} \right. \nonumber \\
&\left. \pm \frac{\sqrt{t} }{2\sqrt{np_0(1-p_0)}} \left\{\mathsf{Var}(L) - t(1 + 2p_0 - 2p_0^2)\right\}\right)
\end{align*}
for $i=1,2$, which are not smooth functions. Anyway, for any fixed $t$ such that the quantity $np_0 \pm \sqrt{n t p_0(1-p_0)} + \frac{(1 - 2p_0) t}{6}$
is not a singularity of $S_i$, we can actually apply the Taylor formula, since the $S_i$ are smooth away from their singularities. With this trick we show that the ``irregular'' part of \eqref{FDR_true}
provides the quantity $\frac{1}{n}I_n(t;p_0)$, where
\begin{align*}
I_n(t;p_0) := &\frac{1}{\sqrt{p_0(1-p_0)}} S_1\left(np_0 + \sqrt{n t p_0(1-p_0)} + \frac{(1 - 2p_0) t}{6}\right) \times \\
&\times \left[ \Phi'(\sqrt{t})\frac{(1 - 2p_0) t}{6\sqrt{p_0(1-p_0)}} + Q'_1(\sqrt{t}) \right] \\
&+ \frac{1}{p_0(1-p_0)} S_2\left(np_0 + \sqrt{n t p_0(1-p_0)} + \frac{(1 - 2p_0) t}{6}\right) \Phi''(\sqrt{t}) \\
&-\frac{1}{\sqrt{p_0(1-p_0)}} S_1\left(np_0 - \sqrt{n t p_0(1-p_0)} + \frac{(1 - 2p_0) t}{6}\right) \times \\
&\times \left[ \Phi'(-\sqrt{t})\frac{(1 - 2p_0) t}{6\sqrt{p_0(1-p_0)}} + Q'_1(-\sqrt{t}) \right] \\
&- \frac{1}{p_0(1-p_0)} S_2\left(np_0 - \sqrt{n t p_0(1-p_0)} + \frac{(1 - 2p_0) t}{6}\right) \Phi''(-\sqrt{t})\ .
\end{align*}
The only term which is not caught by this technique is 
$$
\frac{\Phi'(\sqrt{t})}{\sqrt{np_0(1-p_0)}} \left[S_1(np_0 + \sqrt{n}\mathfrak{Z}_{n,+}^{(\ep,m)}(p_0; t)) - S_1(np_0 + \sqrt{n}\mathfrak{Z}_{n,-}^{(\ep,m)}(p_0; t)) \right]
$$
which corresponds to the terms $\frac{c_{\ast}(t;p_{0};\epsilon,m)}{\sqrt{2\pi n}} e^{-t/2}$ in \eqref{exp_t}. Therefore, we can finally set
\begin{align*}
c_1(t; p_0) &:= 2Q_1(\sqrt{t}) \\
c_2(t; p_0) &:= - \Phi'(\sqrt{t}) \frac{t^{3/2} (1 + 2p_0 - 2p_0^2)}{p_0(1-p_0)} + \Phi''(\sqrt{t}) \frac{(1 - 2p_0)^2 t^2}{36p_0(1-p_0)} + I_n(t;p_0)\ . 
\end{align*}

Then, we pass to the analysis of the quantity
\begin{align*}
&\Phi\left(\frac{\mathfrak{V}_{n,+}^{(\ep,m)}(p_0; t)}{\sqrt{p_0(1-p_0)}}\right) + \sum_{j=1}^2 \left(\frac{1}{\sqrt{n}}\right)^j Q_j\left(\frac{\mathfrak{V}_{n,+}^{(\ep,m)}(p_0; t)}{\sqrt{p_0(1-p_0)}}\right) \\
&+ \frac{\mathsf{Var}(L)}{2np_0(1-p_0)} \Phi''\left(\frac{\mathfrak{V}_{n,+}^{(\ep,m)}(p_0; t)}{\sqrt{p_0(1-p_0)}}\right) - \Phi\left(\frac{\mathfrak{V}_{n,-}^{(\ep,m)}(p_0; t)}{\sqrt{p_0(1-p_0)}}\right) \\
&- \sum_{j=1}^2 \left(\frac{1}{\sqrt{n}}\right)^j Q_j\left(\frac{\mathfrak{V}_{n,-}^{(\ep,m)}(p_0; t)}{\sqrt{p_0(1-p_0)}}\right) - \frac{\mathsf{Var}(L)}{2np_0(1-p_0)} 
\Phi''\left(\frac{\mathfrak{V}_{n,-}^{(\ep,m)}(p_0; t)}{\sqrt{p_0(1-p_0)}}\right)
\end{align*}
which represents the ``regular'' part of the expression of \eqref{FDR_naive}. Again expanding by the Taylor formula, we get the equivalent expression 
\begin{align*}
&\K(t) - \Phi'(\sqrt{t}) \frac{t^{3/2} (1 + 2p_0 - 2p_0^2)}{np_0(1-p_0)} + \Phi''(\sqrt{t}) \frac{(1 - 2p_0)^2 t^2}{36np_0(1-p_0)} + Q_1(\sqrt{t}) \frac{2}{\sqrt{n}}\\
&+ \frac{\mathsf{Var}(L)}{np_0(1-p_0)} \Phi''(\sqrt{t}) \ . 
\end{align*}
Since $\Phi''(\sqrt{t}) = -\left(\frac{t\text{e}^{-t}}{2\pi}\right)^{1/2}$ and the ``irregular'' part contained in $I_n(t;p_0)$ is the same as above, we conclude the validity of \eqref{exp_n}, proving the theorem.

%%%%%%%%%%%%%%%%%%%%%%%%%%%%%%%%%%%%%%%%%%%%%%%%%%%%%%%%%%%%%
%%%%%%%%%%%%%%%%%%%%%%%%%%%%%%%%%%%%%%%%%%%%%%%%%%%%%%%%%%%%%
%%%%%%%%%%%%%%%%%%%%%%%%%%%%%%%%%%%%%%%%%%%%%%%%%%%%%%%%%%%%%%%%

\subsection{Proof of Theorem \ref{teo33} for $k=2$}

We start again from \eqref{ApiuL} where the random variables $A_{n}$ and $L$, defined on the probability space $(\Omega,\mathscr{F},\text{Pr})$, are independent, $A_n \sim Bin(n,p_1)$ and $L$ has the Laplace distribution \eqref{Laplace}. After fixing $\lambda_{\alpha}$ such that $\K(\lambda_{\alpha}) = 1-\alpha$, according to Theorem \ref{teo31}, we have that
$$
1-\beta_n(p_1;\alpha) = \mathrm{Pr}[\Lambda_{n,T}(p_0) \leq \lambda_{\alpha}]\ . 
$$
First, for $\delta \in (0, 1/2)$, we define the event $E_n(\delta)\subset\Omega$ as
\begin{equation}\label{event_delta}
E_n(\delta) :=\left\{\omega\in\Omega\text{ $:$ } \frac{A_{n}(\omega)}{n} \in [\delta, 1-\delta] \right\} \ .
\end{equation}
Then, recalling that 
$$
\mathcal D_{KL}(p_0\ \|\ p_1) := p_0\log\left(\frac{p_0}{p_1}\right) + (1-p_0) \log\left(\frac{1-p_0}{1-p_1}\right)
$$
is a fixed quantity, we resort once again on a large deviation argument to choose $\delta$ sufficiently small so that
\begin{equation}\label{bound_largedev}
\mathrm{Pr}[E_n(\delta)^c] \leq e^{-nC(\delta)}
\end{equation}
holds for some $C(\delta) > \mathcal D_{KL}(p_0\ \|\ p_1)$. In this way,  we can write
$$
1-\beta_n(p_1;\alpha) = \mathrm{Pr}[\Lambda_{n,T}(p_0) \leq \lambda_{\alpha}, E_n(\delta)] + \mathrm{Pr}[\Lambda_{n,T}(p_0) \leq \lambda_{\alpha}, E_n(\delta)^c]
$$
with
$$
\mathrm{Pr}[\Lambda_{n,T}(p_0) \leq \lambda_{\alpha}, E_n(\delta)^c] = o\left(e^{-n\mathcal D_{KL}(p_0\ \|\ p_1)}\right)\ .
$$
Then, for $M>0$, we define another event $E'_n(M)\subset\Omega$ as
\begin{equation}\label{event_M}
E'_n(M) :=\left\{\omega\in\Omega\text{ $:$ } \Big|\frac{A_{n}(\omega)}{n} - p_0\Big| \leq \frac{M}{\sqrt{n}}\right\} 
\end{equation}
and we write
\begin{align*}
\mathrm{Pr}[\Lambda_{n,T}(p_0) \leq \lambda_{\alpha}, E_n(\delta)] &= \mathrm{Pr}[\Lambda_{n,T}(p_0) \leq \lambda_{\alpha}, E_n(\delta) \cap E'_n(M)] \\
&+\mathrm{Pr}[\Lambda_{n,T}(p_0) \leq \lambda_{\alpha}, E_n(\delta) \cap E'_n(M)^c]\ .
\end{align*}
The advantage of such a preliminary step is that, on $E_n(\delta) \cap E'_n(M)$, the assumptions of all the Lemmata stated in the previous subsection are fulfilled and, by resorting to 
\eqref{event_reject_H0_Z}, we can write
\begin{align*}
&\left\{\Lambda_{n,T}(p_0) \leq \lambda_{\alpha}, E_n(\delta) \cap E'_n(M)\right\} \\
&= \left\{\mathfrak{Z}_{n,-}^{(\ep,m)}(p_0; \lambda_{\alpha}) \leq Z_n \leq \mathfrak{Z}_{n,+}^{(\ep,m)}(p_0; \lambda_{\alpha}), E_n(\delta) \cap E'_n(M)\right\}
\end{align*}
where $Z_n$ is the same random variable as in \eqref{eq:Vn}. Moreover, if $M > \lambda_{\alpha}$, as we will choose, we have eventually that
$$
\left\{\Lambda_{n,T}(p_0) \leq \lambda_{\alpha}, E_n(\delta) \cap E'_n(M)\right\} = \left\{\mathfrak{Z}_{n,-}^{(\ep,m)}(p_0; \lambda_{\alpha}) \leq Z_n \leq \mathfrak{Z}_{n,+}^{(\ep,m)}(p_0; \lambda_{\alpha}), 
E_n(\delta)\right\}\ .
$$
It remains to show that, for a suitable choice of $M$, we have
\begin{equation} \label{emptyset}
\left\{\Lambda_{n,T}(p_0) \leq \lambda_{\alpha}, E_n(\delta) \cap E'_n(M)^c\right\} = \emptyset
\end{equation}
eventually, yielding that 
\begin{equation} \label{Power_1}
1-\beta_n(p_1;\alpha) = \mathrm{Pr}[\mathfrak{Z}_{n,-}^{(\ep,m)}(p_0; \lambda_{\alpha}) \leq Z_n \leq \mathfrak{Z}_{n,+}^{(\ep,m)}(p_0; \lambda_{\alpha})] + o\left(e^{-n\mathcal D_{KL}(p_0\ \|\ p_1)}\right)\ .
\end{equation}
The proof of \eqref{emptyset}, along with the proper choice of $M$, is contained in the following 
\begin{lem} \label{lem:emptyset}
If $b/n \in [\delta, 1-\delta]$ and $|b/n - p_0| > \sqrt{(\lambda_{\alpha} + S(p_0, \delta; \epsilon, m))/n}$, with 
$$
S(p_0, \delta; \epsilon, m) := 2\log\left(\max_{x \in [\delta, 1-\delta]} \E\left[ \exp\left\{ \left(\frac{x(1-p_0)}{(1-x)p_0}\right)L \right\}\right]\right)\ ,
$$
then, eventually, there holds
$$
2\log\left(\frac{L_{n,T}(\hat{p}_{n,T}; b)}{L_{n,T}(p_0; b)}\right) > \lambda_{\alpha}\ .
$$
\end{lem}

\begin{proof}
Since $L_{n,T}(\hat{p}_{n,T}; b) \geq L_{n,T}(b/n; b)$, by definition of MLE estimator, we recall \eqref{LikTrueDec} to write
$$
2\log\left(\frac{L_{n,T}(\hat{p}_{n,T}; b)}{L_{n,T}(p_0; b)}\right) \geq 2\log\left(\frac{L_n^{(0)}(b/n; b)}{L_n^{(0)}(p_0; b)}\right) + 2\log\left(\frac{H_{n,T}^{(\ep,m)}(b/n; b)}{H_{n,T}^{(\ep,m)}(p_0; b)}\right)\ .
$$ 
Then, recalling \eqref{Definition_Ln0}, we have
$$
2\log\left(\frac{L_n^{(0)}(b/n; b)}{L_n^{(0)}(p_0; b)}\right) = 2n\mathcal D_{KL}(b/n\ \|\ p_0) \geq n(b/n - p_0)^2 > \lambda_{\alpha} + S(p_0, \delta; \epsilon, m)\ .
$$
Finally, exploiting that $\rho(n,b,l) \sim \left(\frac{b/n}{1-b/n}\right)^l$ and recalling \eqref{Definition_HnT}, we conclude that
$$
2\log\left(\frac{H_{n,T}^{(\ep,m)}(b/n; b)}{H_{n,T}^{(\ep,m)}(p_0; b)}\right) \geq - S(p_0, \delta; \epsilon, m) 
$$
holds eventually, completing the proof. 
\end{proof}

At this stage, we come back to \eqref{Power_1} by writing
$$
Z_n = \sqrt{p_1(1-p_1)} V_n + \frac{L}{\sqrt{n}} - \sqrt{n} \Delta
$$
with 
$$
V_n := \frac{A_n - np_1}{\sqrt{n p_1(1-p_1)}}
$$
and $\Delta := p_0 - p_1$. Thus, the event considered in \eqref{Power_1} can be rewritten in terms of the random variable $V_n$ as follows
\begin{align*}
&\left\{\mathfrak{Z}_{n,-}^{(\ep,m)}(p_0; \lambda_{\alpha}) \leq Z_n \leq \mathfrak{Z}_{n,+}^{(\ep,m)}(p_0; \lambda_{\alpha})\right\} \\
&= \left\{- \frac{L}{\sqrt{n}} + \sqrt{n} \Delta + \mathfrak{Z}_{n,-}^{(\ep,m)}(p_0; \lambda_{\alpha}) \leq \sqrt{p_1(1-p_1)} V_n \leq -\frac{L}{\sqrt{n}} + \sqrt{n} \Delta + \mathfrak{Z}_{n,+}^{(\ep,m)}(p_0; 
\lambda_{\alpha})\right\}\ .
\end{align*}
We now introduce the distribution function $\Hg_n$ of $V_n$, and we notice that \eqref{Power_1} becomes
\begin{align} 
1-\beta_n(p_1;\alpha) &= \frac{1}{c_{\varepsilon,m}} \sum_{l=-m}^m  e^{-\varepsilon |l|} \left[\Hg_n\left(\frac{\sqrt{n} \Delta + l/{\sqrt{n}} + \mathfrak{Z}_{n,+}^{(\ep,m)}(p_0; \lambda_{\alpha})}{\sqrt{p_1(1-p_1)} } \right) 
\right. \nonumber \\
& \left. - \Hg_n\left(\frac{\sqrt{n} \Delta + l/{\sqrt{n}} + \mathfrak{Z}_{n,-}^{(\ep,m)}(p_0; \lambda_{\alpha})}{\sqrt{p_1(1-p_1)} } \right)\right]
+o\left(e^{-n\mathcal D_{KL}(p_0\ \|\ p_1)}\right)\ . \label{Power_2}
\end{align}
At this stage, we are in a position to apply Theorem 10 in Chapter VIII of \citet{Pet(75)}. Supposing, for instance that $\Delta < 0$, we have that
\begin{align} 
&\Hg_n\left(\frac{\sqrt{n} \Delta + l/{\sqrt{n}} + \mathfrak{Z}_{n,\pm}^{(\ep,m)}(p_0; \lambda_{\alpha})}{\sqrt{p_1(1-p_1)} } \right) \sim 
\Phi\left(\frac{\sqrt{n} \Delta + l/{\sqrt{n}} + \mathfrak{Z}_{n,\pm}^{(\ep,m)}(p_0; \lambda_{\alpha})}{\sqrt{p_1(1-p_1)} } \right) \times \nonumber \\
&\times \exp\left\{ \frac{[\sqrt{n} \Delta + l/{\sqrt{n}} + \mathfrak{Z}_{n,\pm}^{(\ep,m)}(p_0; \lambda_{\alpha})]^3}{\sqrt{n} [p_1(1-p_1)]^{3/2}}
\mathcal L\left(\frac{\Delta + l/n + \mathfrak{Z}_{n,\pm}^{(\ep,m)}(p_0; \lambda_{\alpha})/\sqrt{n}}{\sqrt{p_1(1-p_1)} } \right) \right\} \label{Power_3}
\end{align}
where $\mathcal L$ stands for the so-called Cram\'er-Petrov series relative to the Bernoulli distribution of parameter $p_1$. Thus, putting
$$
\mathfrak{T}_{n,\pm} := \frac{\Delta + l/n + \mathfrak{Z}_{n,\pm}^{(\ep,m)}(p_0; \lambda_{\alpha})/\sqrt{n}}{\sqrt{p_1(1-p_1)}},
$$
we proceed by resorting to the well-known Mills approximation to write
$$
\Phi\left(\frac{\sqrt{n} \Delta + l/{\sqrt{n}} + \mathfrak{Z}_{n,\pm}^{(\ep,m)}(p_0; \lambda_{\alpha})}{\sqrt{p_1(1-p_1)} } \right) \sim \frac{1}{\sqrt{n}} \exp\{-\frac{n}{2} \mathfrak{T}_{n,\pm}^2\}\ .
$$
Therefore, \eqref{Power_3} can be simplified as follows
$$
\Hg_n\left(\frac{\sqrt{n} \Delta + l/{\sqrt{n}} + \mathfrak{Z}_{n,\pm}^{(\ep,m)}(p_0; \lambda_{\alpha})}{\sqrt{p_1(1-p_1)} } \right) \sim \frac{1}{\sqrt{n}} \exp\left\{-n \left(\frac 12\mathfrak{T}_{n,\pm}^2
-\mathfrak{T}_{n,\pm}^3 \mathcal L(\mathfrak{T}_{n,\pm})\right)\right\}\ .
$$

At this stage, to complete the proof, we need a technical result that characterizes the expression $\frac 12 t^2 - t^3 \mathcal L(t)$ in an exponential model parametrized by the mean.
\begin{lem} \label{lem:exp_fam}
Let $(\mathbb X, \mathcal X)$ be a measurable space, endowed with a $\sigma$-finite reference measure $\mu$. Let $\mathfrak t : \mathbb X \to \mathbb R^d$ be a measurable map such that the set
$$
\Gamma := \left\{ y \in \mathbb R^d\ \Big|\ \int_{\mathbb X} e^{y \cdot \mathfrak t(x)} \mu(\ddr x) < +\infty\right\}
$$
is open and convex. Putting $M(y) := \log\left(\int_{\mathbb X} e^{y \cdot \mathfrak t(x)} \mu(\ddr x)\right)$, $V(y) :=\nabla M(y)$ and $\Theta := V(\Gamma)$, we have that $\Theta$ is open and $V$ is 
a smooth diffeomorphism between $\Gamma$ and $\Theta$. Moreover, the family of $\mu$-densities $\{f_{\theta}\}_{\theta \in \Theta}$ given by
$$
f_{\theta}(x) := \exp\{V^{-1}(\theta) \cdot \mathfrak t(x) - M(V^{-1}(\theta))\} \qquad (x \in \mathbb X)
$$ 
defines a regular exponential family parametrized by the mean. Moreover, we have
\begin{align*}
\mathcal D_{KL}(\tau\ \|\ \theta) &:= \int_{\mathbb X} \log\left(\frac{f_{\tau}(x) }{f_{\theta}(x) }\right) f_{\tau}(x) \mu(\ddr x) \\
&= \tau \cdot [V^{-1}(\tau) - V^{-1}(\theta)] - [M(V^{-1}(\theta)) - M(V^{-1}(\tau))] \\
&= \Psi_{\theta}^{\ast}(\tau) := \sup_{y\in \Gamma}\{\tau \cdot y - \Psi_{\theta}(\tau)\}
\end{align*}
where $\Psi_{\theta}(\tau) := \log\left(\int_{\mathbb X} e^{y \cdot \mathfrak t(x)} f_{\theta}(x)\mu(\ddr x)\right)$. For $d=1$, putting 
$$
\sigma^2(\theta) := \int_{\mathbb X} [\mathfrak t(x) - \theta]^2 f_{\theta}(x)\mu(\ddr x)\ , 
$$
we have that 
\begin{equation} \label{kullback_cramer}
\frac 12 \left(\frac{\tau - \theta}{\sigma(\theta)} \right)^2 - \left(\frac{\tau - \theta}{\sigma(\theta)} \right)^3 \mathcal L\left(\frac{\tau - \theta}{\sigma(\theta)} \right) = \mathcal D_{KL}(\tau\ \|\ \theta)
\end{equation}
for every $\tau, \theta \in \Theta$, where $\mathcal L$ denotes the Cram\'er-Petrov series relative to the distribution of $\mathfrak t(X) - \theta$, with $X \sim f_{\theta}$. Finally, for a generic dimension $d$, 
putting $L(\zb) := \Psi_{\boldsymbol \theta}(\zb) - \zb \cdot \boldsymbol \theta$, we have that
\begin{equation} \label{kullback_cramerK}
\hat{\zb} \cdot \nabla L(\hat{\zb}) = \mathcal D_{KL}(\boldsymbol \tau\ \|\ \boldsymbol\theta)
\end{equation}
if $\hat{\zb}$ is defined as the solution of the equation $\nabla L(\hat{\zb}) = \boldsymbol \tau - \boldsymbol\theta$.
\end{lem}

\begin{proof}
For the main part of the lemma, we just quote any good reference on exponential families, such as \citet{BN(78)}. Here, we only prove \eqref{kullback_cramer}. Putting $\eta := \frac{\tau - \theta}{\sigma(\theta)}$,
we have by definition that 
$$
\frac 12 \eta^2 - \eta^3 \mathcal L(\eta) := \overline{z} L'(\overline{z}) - L(\overline{z})
$$
where $L(z) := -z\theta + \Psi_{\theta}(z)$ and $\overline{z}$ is the solution of the equation $\sigma(\theta) \eta = L'(\overline{z})$. Thus, the last equation can be rewritten as $\tau = \Psi'_{\theta}(\overline{z})$. 
Whence,
$$
\overline{z} L'(\overline{z}) - L(\overline{z}) = \overline{z}(-\theta + \Psi'_{\theta}(\overline{z})) + \overline{z}\theta - \Psi_{\theta}(\overline{z}) = \tau \overline{z} - \Psi_{\theta}(\overline{z}) = \Psi_{\theta}^{\ast}(\tau)
$$
completing the proof of \eqref{kullback_cramer}. The same computation in generic dimension yields \eqref{kullback_cramerK}.
\end{proof}

At this stage, noticing that the family of Bernoulli distributions is a member of the regular exponential family parametrized by the mean, we can apply \eqref{kullback_cramer} with $\theta = p_1$ and $\tau = p_0 + 
l/n + \mathfrak{Z}_{n,\pm}^{(\ep,m)}(p_0; \lambda_{\alpha})$ to conclude that
\begin{align*}
&\Hg_n\left(\frac{\sqrt{n} \Delta + l/{\sqrt{n}} + \mathfrak{Z}_{n,\pm}^{(\ep,m)}(p_0; \lambda_{\alpha})}{\sqrt{p_1(1-p_1)} } \right) \\
&\sim \frac{1}{\sqrt{n}} \exp\left\{-n \mathcal D_{KL}\left( p_0 + 
l/n + \mathfrak{Z}_{n,\pm}^{(\ep,m)}(p_0; \lambda_{\alpha})/\sqrt{n}\ \|\ p_1\right)\right\}\ .
\end{align*}
Thus, if $\Delta < 0$, we have
$$
\mathcal D_{KL}\left( p_0 + l/n + \mathfrak{Z}_{n,+}^{(\ep,m)}(p_0; \lambda_{\alpha})/\sqrt{n}\ \|\ p_1\right) < \mathcal D_{KL}\left( p_0 + l/n + \mathfrak{Z}_{n,-}^{(\ep,m)}(p_0; \lambda_{\alpha})/\sqrt{n}\ \|\ p_1\right)
$$
and we conclude that
\begin{align*}
1-\beta_n(p_1;\alpha) &\sim \frac{1}{\sqrt{n}} \frac{1}{c_{\varepsilon,m}} \sum_{l=-m}^m  e^{-\varepsilon |l|} \times \\
&\times \exp\left\{-n \mathcal D_{KL}\left( p_0 + l/n + \mathfrak{Z}_{n,+}^{(\ep,m)}(p_0; \lambda_{\alpha})/\sqrt{n}\ \|\ p_1\right)\right\}\ . 
\end{align*}
As a last step of the proof, we expand the above expression containing $\mathcal D_{KL}$ by the Taylor formula, obtaining
\begin{align*}
&\mathcal D_{KL}\left( p_0 + l/n + \mathfrak{Z}_{n,+}^{(\ep,m)}(p_0; \lambda_{\alpha})/\sqrt{n}\ \|\ p_1\right) \\
&= \mathcal D_{KL}\left( p_0\ \|\ p_1\right) + \partial_{p_0}\mathcal D_{KL}\left( p_0\ \|\ p_1\right) \left( l/n + \mathfrak{Z}_{n,+}^{(\ep,m)}(p_0; \lambda_{\alpha})/\sqrt{n}\right)\\
&+ \frac 12 \partial_{p_0}^2 \mathcal D_{KL}\left( p_0\ \|\ p_1\right)[\mathfrak{Z}_{n,+}^{(\ep,m)}(p_0; \lambda_{\alpha})]^2/n + O(n^{-3/2})\ .
\end{align*}
In conclusion, we get  
\begin{align*}
1-\beta_n(p_1;\alpha) &\sim \frac{1}{\sqrt{n}} \mathfrak M_L\left(\partial_{p_0}\mathcal D_{KL}\left( p_0\ \|\ p_1\right)\right) \times \\
&\times \exp\left\{-n \left[ \mathcal D_{KL}\left( p_0\ \|\ p_1\right) + \partial_{p_0}\mathcal D_{KL}\left( p_0\ \|\ p_1\right) \mathfrak{Z}_{n,+}^{(\ep,m)}(p_0; \lambda_{\alpha})/\sqrt{n}
\right.\right. \\ 
&\left.\left.+ \frac 12 \partial_{p_0}^2 \mathcal D_{KL}\left( p_0\ \|\ p_1\right)[\mathfrak{Z}_{n,+}^{(\ep,m)}(p_0; \lambda_{\alpha})]^2/n \right]\right\}
\end{align*}
which entails the thesis of the theorem.

%%%%%%%                                   APPENDIX  B                                    %%%%%%%%%%%%%%%%%%%%%%%%%%%%

%%%%%%%%%%%%%%%%%%%%%%%%%%%%%%%%
%%%%%%%%%%%%%%%%%%%%%%%%%%%%%%%%
%%%%%%%%%%%%%%%%%%%%%%%%%%%%%%%%
%%%%%%%%%%%%%%%%%%%%%%%%%%%%%%%%

\section{Proofs for $k>2$}\label{Appb}

In this section we will prove Theorems \ref{teo32}, \ref{teo32n} and \ref{teo33} in the case that our data set is a table with more than two cells, i.e. when $k>2$. 
Thus, the observable variable reduces to the countings $\bbb = (b_1, \dots, b_{k-1})$ contained in the
first $k-1$ cells. Here, we state some preparatory results as in Appendix A.

\subsection{Preparatory steps}

First of all, let us state a proposition that fixes the exact expression of the likelihood under the true model \eqref{true_mod}. 
\begin{lem} \label{lemma_likelihoodK}
Let $m \in \N$ and $\varepsilon > 0$ be fixed to define the Laplace distribution \eqref{Laplace}. Then, if $n > m$ and $b_i \in R(n,m) := \{m, m+1, \dots, n-m\}$ for $i=1, 2, \dots, k-1$, we have
\begin{align} 
L_{n,T}(\pb; \bbb) &= \left(\frac{1}{\cepsm}\right)^{k-1} \sum_{l_1=-m}^m \dots \sum_{l_{k-1}=-m}^m e^{-\ep |\lb|} \binom{n}{b_1-l_1, \dots, b_{k-1}-l_{k-1}} \times \label{LikTrueRegK} \\
&\times p_1^{b_1-l_1} \dots p_{k-1}^{b_{k-1}-l_{k-1}}
(1-|\pb|)^{n-|\bbb+\lb|} \qquad (\pb \in \Delta_{k-1}). \nonumber 
\end{align}
Moreover, under the same assumption, for any $\pb$ in the interior $\Delta_{k-1}^o$ of $\Delta_{k-1}$ we can write
\begin{equation} \label{LikTrueDecK} 
L_{n,T}(\pb; \bbb) = L_n^{(0)}(\pb; \bbb) \cdot H_{n,T}^{(\ep,m)}(\pb; \bbb)
\end{equation}
with
\begin{align}
L_n^{(0)}(\pb; \bbb) &:= \binom{n}{b_1, \dots, b_{k-1}} p_1^{b_1} \dots p_{k-1}^{b_{k-1}}(1-|\pb|)^{n-|\bbb|}\label{Definition_Ln0K} \\
H_{n,T}^{(\ep,m)}(\pb; \bbb) &:= \left(\frac{1}{\cepsm}\right)^{k-1} \sum_{l_1=-m}^m \dots \sum_{l_{k-1}=-m}^m e^{-\ep |\lb|} \rho(n,\bbb,\lb) \times \label{Definition_HnTK} \\
&\times \left(\frac{1-|\pb|}{p_1}\right)^{l_1} \dots  \left(\frac{1-|\pb|}{p_{k-1}}\right)^{l_{k-1}} \nonumber \\
\rho(n,\bbb,\lb) &:= \frac{b_1! \dots b_{k-1}! (n-|\bbb|)!}{(b_1-l_1)! \dots (b_{k-1} - l_{k-1})! (n - |\bbb+\lb|)!} \label{Definition_rho_nblK} \ . 
\end{align}
\end{lem}

\begin{proof}
Start from
\begin{equation} \label{ApiuLK} 
\mathbf{B}_n = \Abb_n + \mathbf{L}
\end{equation}
where the random variables $\Abb_{n}$ and $\mathbf{L}$, defined on the probability space $(\Omega,\mathscr{F},\text{Pr})$, are independent, $\Abb_n \sim Mult(n,\pb)$ and $\mathbf L = (L_1, \dots, L_{k-1})$ 
is a random vector with independent components with each $L_i$ having the Laplace distribution \eqref{Laplace}. 
By definition $L_{n,T}(\pb; \bbb) := \textrm{Pr}[\mathbf B_n = \bbb]$, so that, under the assumption of the Lemma, the independence of $\Abb_{n}$ and $\mathbf L$ entails
$$
\textrm{Pr}[\mathbf B_n = \bbb] = \sum_{l_1=-m}^m \dots \sum_{l_{k-1}=-m}^m \textrm{Pr}[\Abb_n = \bbb-\lb] \cdot \textrm{Pr}[\mathbf L = \lb]
$$  
proving \eqref{LikTrueRegK}. Finally, the decomposition \eqref{LikTrueDecK} ensues from straightforward algebraic manipulations.
\end{proof}

The next result is a multidimensional analogue of Lemma \ref{lemma_rho}, that can be obtained by considering the observable quantity $\bbb$ as itself dependent by $n$. This task proves very cumbersome for 
a general dimension $k$ so that, in the remaining part of the subsection, we will confine ourselves to dealing with the case $k=3$.  
\begin{lem} \label{lemma_rhoK}
Let $\xib \in \Delta_2^o$ be a fixed vector. Under the assumption that $\bbb = n(\xib + \epsilonb_n)$ with $\lim_{n\rightarrow 0} \epsilonb_n = \mathbf 0$, there exists $n_0 = n_0(\xib; \ep,m)$ such that 
\begin{align} \label{exp_rhonK} 
\rho_n(\xib,\lb) &:= \rho(n, \bbb, \lb) = \left(\frac{\xi_1}{1-\xi_1-\xi_2}\right)^{l_1} \left(\frac{\xi_2}{1-\xi_1-\xi_2}\right)^{l_2} \times \\
&\times \left\{1 + \alphab_1(\xib,\lb) \cdot \epsilonb_n + 
\ ^t\epsilonb_n \mathbb{A}_2(\xib,\lb) \epsilonb_n + \alpha_3(\xib,\lb) \frac 1n + \alphab_4(\xib,\lb) \cdot \frac{\epsilonb_n}{n} + R_n(\xib,\lb) \right\} \nonumber 
\end{align}
holds for any $n \geq n_0$ and $\lb \in \{-m, \dots, m\}^2$, where
\begin{align*}
\alphab_1(\xib,\lb) &:= \left(\frac{l_1}{\xi_1} +\frac{l_1+l_2}{1-\xi_1-\xi_2}, \frac{l_2}{\xi_2} +\frac{l_1+l_2}{1-\xi_1-\xi_2}\right) \\
\mathbb{A}_2(\xib,\lb) &:=  \begin{pmatrix}
a_{1,1}(\xib,\lb) & a_{1,2}(\xib,\lb) \\
a_{2,1}(\xib,\lb) & a_{2,2}(\xib,\lb) 
\end{pmatrix}
\\
a_{1,1}(\xib,\lb) &:= \binom{l_1}{2}\frac{1}{\xi_1^2} + \binom{l_1+l_2+1}{2}\frac{1}{(1-\xi_1-\xi_2)^2} + \frac{l_1(l_1+l_2)}{\xi_1(1-\xi_1-\xi_2)} \\
a_{1,2}(\xib,\lb) &:= \binom{l_1+l_2+1}{2}\frac{1}{(1-\xi_1-\xi_2)^2} + \frac{l_1l_2}{2 \xi_1\xi_2} + \frac{l_1(l_1+l_2)}{2\xi_1(1-\xi_1-\xi_2)} + \frac{l_2(l_1+l_2)}{2\xi_2(1-\xi_1-\xi_2)} \\
a_{2,1}(\xib,\lb) &:= \binom{l_1+l_2+1}{2}\frac{1}{(1-\xi_1-\xi_2)^2} + \frac{l_1l_2}{2 \xi_1\xi_2} + \frac{l_1(l_1+l_2)}{2\xi_1(1-\xi_1-\xi_2)} + \frac{l_2(l_1+l_2)}{2\xi_2(1-\xi_1-\xi_2)} \\
a_{2,2}(\xib,\lb) &:= \binom{l_2}{2}\frac{1}{\xi_2^2} + \binom{l_1+l_2+1}{2}\frac{1}{(1-\xi_1-\xi_2)^2} + \frac{l_2(l_1+l_2)}{\xi_2(1-\xi_1-\xi_2)} \\
\alpha_3(\xib,\lb) &:= -\left[\binom{l_1}{2}\frac{1}{\xi_1} +\binom{l_2}{2}\frac{1}{\xi_2} + \binom{l_1+l_2+1}{2}\frac{1}{1-\xi_1-\xi_2} \right] \\
\alphab_4(\xib,\lb) &:= -\frac{l(l^2 - 2l +1 + 2\xi^2 + 4l\xi - 2\xi)}{2\xi^2(1-\xi)^2} \\
|R_n(\xib,\lb)| &\leq C(\xib,\lb) \left[|\epsilonb_n|^3 + \frac{1}{n^2} + \frac{|\epsilonb_n|^2}{n} \right]
\end{align*}
for some constant $C(\xib,\lb)$. Therefore, under the same assumption, there holds
\begin{align}
H_{n,T}^{(\ep,m)}(\pb; \bbb) &= \mathcal H_0^{(\ep,m)}(\pb; \xib) + \boldsymbol{\mathcal H}_1^{(\ep,m)}(\pb; \xib) \cdot \epsilonb_n +\ ^t\epsilonb_n \mathbb{H}_2(\pb;\xib) \epsilonb_n\label{HnT_ep_mK}\\
&+ \mathcal H_3^{(\ep,m)}(\pb; \xib) \frac 1n + \boldsymbol{\mathcal H}_4^{(\ep,m)}(\pb; \xib) \cdot \frac{\epsilonb_n}{n} + \mathcal R_n(\pb;\xib) \nonumber
\end{align}
for any $n \geq n_0$, where
\begin{align*}
\mathcal H_i^{(\ep,m)}(\pb; \xib) &:= \left(\frac{1}{\cepsm}\right)^2 \sum_{l_1=-m}^m \sum_{l_2=-m}^m e^{-\ep (|l_1| + |l_2|)} \alpha_i(\xib,\lb) \times \\
&\times \left(\frac{\xi_1(1-p_1-p_2)}{p_1(1-\xi_1-\xi_2)}\right)^{l_1} \left(\frac{\xi_2(1-p_1-p_2)}{p_2(1-\xi_1-\xi_2)}\right)^{l_2} \qquad (i= 0,3)\\
\boldsymbol{\mathcal H}_i^{(\ep,m)}(\pb; \xib) &:= \left(\frac{1}{\cepsm}\right)^2 \sum_{l_1=-m}^m \sum_{l_2=-m}^m e^{-\ep (|l_1| + |l_2|)} \alphab_i(\xib,\lb) \times \\
&\times \left(\frac{\xi_1(1-p_1-p_2)}{p_1(1-\xi_1-\xi_2)}\right)^{l_1} \left(\frac{\xi_2(1-p_1-p_2)}{p_2(1-\xi_1-\xi_2)}\right)^{l_2} \qquad (i= 1,4)\\
\mathbb{H}_2^{(\ep,m)}(\pb;\xib) &:= \left(\frac{1}{\cepsm}\right)^2 \sum_{l_1=-m}^m \sum_{l_2=-m}^m e^{-\ep (|l_1| + |l_2|)} \mathbb{A}_2(\xib,\lb) \times \\
&\times \left(\frac{\xi_1(1-p_1-p_2)}{p_1(1-\xi_1-\xi_2)}\right)^{l_1} \left(\frac{\xi_2(1-p_1-p_2)}{p_2(1-\xi_1-\xi_2)}\right)^{l_2}  \\
|\mathcal R_n(\pb; \xib)| &\leq \mathcal C(\pb;\xib) \left[|\epsilonb_n|^3 + \frac{1}{n^2} + \frac{|\epsilonb_n|^2}{n}\right]
\end{align*}
with $\alpha_0(\xib,\lb) \equiv 1$ and 
\begin{align*}
\mathcal C(\pb;\xib) := \left(\frac{1}{\cepsm}\right)^2 \sum_{l_1=-m}^m \sum_{l_2=-m}^m & e^{-\ep (|l_1| + |l_2|)} C(\xib,\lb) \times \\
& \times \left(\frac{\xi_1(1-p_1-p_2)}{p_1(1-\xi_1-\xi_2)}\right)^{l_1} \left(\frac{\xi_2(1-p_1-p_2)}{p_2(1-\xi_1-\xi_2)}\right)^{l_2} \ .
\end{align*}
\end{lem} 

\begin{proof}
We start by dealing with \eqref{exp_rhonK}. First, we find $n_0 = n_0(\xib; \ep,m)$ in such a way that the assumptions of Lemma \ref{lemma_likelihood} are fulfilled for any $n \geq n_0$. 
Then, as in the proof of Lemma \ref{lemma_rho}, the thesis can be checked by direct computation if either $l_1 \in \{0,1\}$ or $l_2 \in \{0,1\}$.
Now, if $l_1,l_2 \in \{2, \dots, m\}$, we use \eqref{Definition_rho_nblK} to get
\begin{align*}
\rho(n, \bbb, \lb) &= \frac{(b_1)_{\downarrow l_1} (b_2)_{\downarrow l_2}}{(n-b_1-b_2+1)_{\uparrow (l_1+l_2)}} = \frac{(b_1)_{\downarrow l_1} (b_2)_{\downarrow l_2} (n-b_1-b_2)}{(n-b_1-b_2)_{\uparrow 
(l_1+l_2+1)}} \\
&= \frac{\left(\sum_{k_1=1}^{l_1} \mathfrak s(l_1,k_1) b_1^{k_1}\right) \cdot \left(\sum_{k_2=1}^{l_2} \mathfrak s(l_2,k_2) b_2^{k_2}\right)}{\sum_{k=1}^{l_1+l_2+1} |\mathfrak s(l_1+l_2+1,k)| (n-b_1-b_2)^{k-1}} 
\end{align*}
where again $\mathfrak s(l,k)$ denotes the Stirling number of first kind. Whence, 
\begin{align*}
\rho_n(\xib,\lb) &= \frac{\left(\sum_{k_1=1}^{l_1} \mathfrak s(l_1,k_1) n^{k_1-l_1} (\xi_1+\epsilon_{n,1})^{k_1}\right) \cdot \left(\sum_{k_2=1}^{l_2} \mathfrak s(l_2,k_2) n^{k_2-l_2} (\xi_2+\epsilon_{n,2})^{k_1}\right)}
{\sum_{k=1}^{l_1+l_2+1} |\mathfrak s(l_1+l_2+1,k)| n^{k-l_1-l_2-1} (1-\xi_1-\xi_2-\epsilon_{n,1}-\epsilon_{n,2})^{k-1}} \\
&= \frac{\mathfrak N_n}{\mathfrak D_n}\ .
\end{align*}
At this stage, recalling that $\mathfrak s(l,l) = 1$ and $\mathfrak s(l,l-1) = -\binom{l}{2}$, and exploiting the binomial formula, we have
\begin{align*}
\mathfrak{N}_n &:= \xi_1^{l_1} \left\{1 + l_1 \frac{\epsilon_{n,1}}{\xi_1} + \frac{l_1(l_1-1)}{2} \frac{\epsilon_{n,1}^2}{\xi_1^2} - \frac{l_1(l_1-1)}{2n} \frac{1}{\xi_1} - 
\frac{l_1(l_1-1)^2}{2n} \frac{\epsilon_{n,1}}{\xi_1^2} + R_n(\xi_1,l_1)\right\} \times \\
&\times \xi_2^{l_2} \left\{1 + l_2 \frac{\epsilon_{n,2}}{\xi_2} + \frac{l_2(l_2-1)}{2} \frac{\epsilon_{n,2}^2}{\xi_2^2} - \frac{l_2(l_2-1)}{2n} \frac{1}{\xi_2} - 
\frac{l_2(l_2-1)^2}{2n} \frac{\epsilon_{n,2}}{\xi_2^2} + R_n(\xi_2,l_2)\right\} \\
\mathfrak{D}_n &:= (1-\xi_1-\xi_2)^{l_1+l_2}  \left\{1 - (l_1+l_2) \frac{\epsilon_{n,1} + \epsilon_{n,2}}{1-\xi_1-\xi_2} + \frac{(l_1+l_2)(l_1+l_2-1)}{2} \frac{(\epsilon_{n,1} + \epsilon_{n,2})^2}{(1-\xi_1-\xi_2)^2} \right. \\
&+ \frac{(l_1+l_2)(l_1+l_2+1)}{2n} \frac{1}{1-\xi_1-\xi_2} \\
&\left. - \frac{(l_1+l_2)(l_1+l_2-1)(l_1+l_2+1)}{2n} \frac{\epsilon_{n,1} + \epsilon_{n,2}}{(1-\xi_1-\xi_2)^2} + R_n(\xi_1+\xi_2,l_1+l_2)\right\} 
\end{align*}
for suitable expressions of $R_n(\xi,l)$ (possibly different from line to line) which are, in any case, bounded by an expression like
$C(\xib,\lb) \left[|\epsilonb_n|^3 + \frac{1}{n^2} + \frac{|\epsilonb_n|^2}{n} \right]$. 
To proceed further, we exploit that $\frac{1}{1+t} = 1 - t + t^2 + o(t^2)$ as $t\rightarrow 0$, to obtain
\begin{align*}
\mathfrak{D}_n^{-1} &:= \frac{1}{(1-\xi_1-\xi_2)^{l_1+l_2}} \left\{1 + (l_1+l_2) \frac{\epsilon_{n,1}+\epsilon_{n,2}}{1-\xi_1-\xi_2} 
+ \frac{(l_1+l_2)(l_1+l_2+1)}{2} \frac{(\epsilon_{n,1}+\epsilon_{n,2})^2}{(1-\xi_1-\xi_2)^2} \right. \\
&- \frac{(l_1+l_2)(l_1+l_2+1)}{2n} \frac{1}{1-\xi_1-\xi_2} \\
&\left. - \frac{(l_1+l_2)(l_1+l_2+1)^2}{2n} \frac{\epsilon_{n,1}+\epsilon_{n,2}}{(1-\xi_1-\xi_2)^2} + R_n(\xi_1+\xi_2,l_1+l_2)\right\} \ .
\end{align*}
The thesis now follows by multiplying the last expression by that of $\mathfrak{N}_n$, neglecting all the terms which are comparable with $R_n(\xi,l)$. Thus, \eqref{exp_rhon} is proved also for all 
$l \in \{2, \dots, m\}$.
If either $l_1 \in \{-m, \dots, -1\}$ or $l_2 \in \{-m, \dots, -1\}$ the argument can be reduced to the previous case, as in the proof of Lemma \ref{lemma_rho}.
%In fact, if $l_1 \in \{-m, \dots, -1\}$, we can exploit that $\rho(n, (b_1,b_2), (l_1,l_2)) = \rho(n, (n-b_1-b_2,b_2), (l,l_2))$.
This completes the proof of \eqref{exp_rhonK}.

Finally, \eqref{HnT_ep_mK} follows immediately from the combination of \eqref{Definition_HnTK} with \eqref{exp_rhonK}.
\end{proof}

The way is now paved to state a result on the expansion of the MLE which is analogous to Lemma \ref{lem:MLE}
\begin{lem} \label{lem:MLEK}
Under the same assumption of Lemma \ref{lemma_rhoK}, there holds
\begin{equation} \label{Expansion_MLE_trueK}
\hat{\pb}_{n,T} = \xib + \epsilonb_n + \mathbf R_{n,T}(\xib; \epsilonb_n)
\end{equation}
where $\mathbf R_{n,T}(\xib; \epsilonb_n)$ is a remainder term satisfying $|\mathbf R_{n,T}(\xib; \epsilonb_n)| = O(|\epsilonb_n|^3/n)$. 
\end{lem}

\begin{proof}
We provide only a sketch of the proof, since the ensuing computations are too heavy to be fully reproduced. In any case, we start by providing a large $n$ expansion for the $\log$-likelihood, as follows
\begin{align*}
\ell_{n,T}(\pb;\bbb) := \log L_{n,T}(\pb;\bbb) &= \log \binom{n}{b_1, b_2} + n(\xi_1 + \epsilon_{n,1}) \log p_1 + n(\xi_2 + \epsilon_{n,2}) \log p_2 \\
& + n(1- \xi_1-\xi_2 - \epsilon_{n,1} - \epsilon_{n,2}) \log(1- p_1 - p_2) \\
&+ h_0^{(\ep,m)}(\pb; \xib) + \mathbf h_1^{(\ep,m)}(\pb; \xib) \cdot \epsilonb_n + 
\ ^t\epsilonb_n \mathbb{L}_2^{(\ep,m)}(\pb; \xib) \epsilonb_n \\
&+ h_3^{(\ep,m)}(\pb; \xib) \frac 1n + \mathbf h_4^{(\ep,m)}(\pb; \xib) \cdot \frac{\epsilonb_n}{n} + r_n(\pb;\xib)
\end{align*}
where, upon denoting by $\otimes$ the outer product between vectors, 
\begin{align*}
h_0^{(\ep,m)}(\pb; \xib) &:= \log \mathcal H_0^{(\ep,m)}(\pb; \xib) \\
\mathbf h_1^{(\ep,m)}(\pb; \xib) &:= \frac{\boldsymbol{\mathcal H}_1^{(\ep,m)}(\pb; \xib)}{\mathcal H_0^{(\ep,m)}(\pb; \xib)} \\
\mathbb{L}_2^{(\ep,m)}(\pb; \xib) &:= \frac{\mathbb{H}_2^{(\ep,m)}(\pb,\xib)}{\mathcal H_0^{(\ep,m)}(p; \xi)} - \frac 12 \left(\frac{\boldsymbol{\mathcal H}_1^{(\ep,m)}(p; \xi)}{\mathcal H_0^{(\ep,m)}(p; \xi)}\right)  
\otimes \left(\frac{\boldsymbol{\mathcal H}_1^{(\ep,m)}(p; \xi)}{\mathcal H_0^{(\ep,m)}(p; \xi)}\right) \\
h_3^{(\ep,m)}(\pb; \xib) &:= \frac{\mathcal H_3^{(\ep,m)}(\pb; \xib)}{\mathcal H_0^{(\ep,m)}(\pb; \xib)} \\
\mathbf h_4^{(\ep,m)}(\pb; \xib) &:= \frac{\boldsymbol{\mathcal H}_4^{(\ep,m)}(\pb; \xib)}{\mathcal H_0^{(\ep,m)}(\pb; \xib)} - \frac{\mathcal H_1^{(\ep,m)}(p; \xi)}{\mathcal H_0^{(\ep,m)}(p; \xi)}
\frac{\mathcal H_3^{(\ep,m)}(\pb; \xib)}{\mathcal H_0^{(\ep,m)}(\pb; \xib)} \\
|r_n(\pb;\xib)| &\leq C(\pb;\xib) \left[|\epsilonb_n|^3 + \frac{1}{n^2} + \frac{|\epsilonb_n|^2}{n} \right]
\end{align*}
for some constant $C(\pb;\xib)$. Then, to find the maximum point of the likelihood, we study the equation $\nabla_{\pb} \ell_{n,T}(\pb;\bbb) = \mathbf 0$, which reads
$$ 
\begin{cases}
n \dfrac{p_1 -\xi_1 -\epsilon_{n,1} + p_2(\xi_1 + \epsilon_{n,1}) - p_1(\xi_2 + \epsilon_{n,2})}{p_1(1-p_1 -p_2)} &= [\partial_{p_1} h_0^{(\ep,m)}(\pb; \xib)] + 
[\partial_{p_1}\mathbf h_1^{(\ep,m)}(\pb; \xib)] \cdot \epsilonb_n \\
&+ \ ^t\epsilonb_n [\partial_{p_1}\mathbb{L}_2^{(\ep,m)}(\pb; \xib)] \epsilonb_n +  [\partial_{p_1} h_3^{(\ep,m)}(\pb; \xib)] \frac 1n \\
&+  [\partial_{p_1}\mathbf h_4^{(\ep,m)}(\pb; \xib)] \cdot \frac{\epsilonb_n}{n} +  [\partial_{p_1}r_n(\pb;\xib)]\\
\\
n \dfrac{p_2 -\xi_2 -\epsilon_{n,2} + p_1(\xi_2 + \epsilon_{n,2}) - p_2(\xi_1 + \epsilon_{n,1})}{p_2(1-p_1 -p_2)} &= [\partial_{p_2} h_0^{(\ep,m)}(\pb; \xib)] + 
[\partial_{p_2}\mathbf h_1^{(\ep,m)}(\pb; \xib)] \cdot \epsilonb_n \\
&+ \ ^t\epsilonb_n [\partial_{p_2}\mathbb{L}_2^{(\ep,m)}(\pb; \xib)] \epsilonb_n +  [\partial_{p_2} h_3^{(\ep,m)}(\pb; \xib)] \frac 1n \\
&+  [\partial_{p_2}\mathbf h_4^{(\ep,m)}(\pb; \xib)] \cdot \frac{\epsilonb_n}{n} +  [\partial_{p_2}r_n(\pb;\xib)]\ .
\end{cases}
$$
The solution of such an equation can be obtained by inserting the expression $\xi_i + \epsilon_{n,i} + \Gamma_{0,i}(\xib) \frac{1}{n} + \boldsymbol{\Gamma}_{1,i}(\xib) \cdot \frac{\epsilonb_n}{n} + 
\frac 1n \ ^t \epsilonb_n \mathbb{V}_{2,i}(\xib) \epsilonb_n + \Delta_{n,i}(\xib; \epsilonb_n)$ in the place of $p_i$, $i=1,2$, and then expanding both members. By carefully carrying out all of these computations,
we are in a position to conclude that $\Gamma_{0,i}(\xib) = 0$, $\boldsymbol{\Gamma}_{1,i}(\xib) = \mathbf 0$ and $\mathbb{V}_{2,i}(\xib) = \mathbb O$, thus completing the proof.
\end{proof}

The last preparatory result is the multidimensional analogous of Lemma \ref{lem_gothic}.
\begin{lem} \label{lem_gothic}
Under the same assumption of Lemma \ref{lemma_rhoK} with $\xib = \pb_0$, there hold
\begin{align}
\Lambda_{k,n,T}(\pb_0; \bbb) &:= -2\log\left( \frac{L_{n,T}(\pb_0;\bbb)}{L_{n,T}(\hat{\pb}_{n,T}; \bbb)} \right) \nonumber \\
&= 2\log\left( \frac{L_n^{(0)}(\hat{\pb}_{n,T}; \bbb)}{L_n^{(0)}(\pb_0; \bbb)} \right) + 2\log\left( \frac{H_n^{(\ep,m)}(\hat{\pb}_{n,T}; \bbb)}{H_n^{(\ep,m)}(\pb_0; \bbb)} \right) \label{Dec_LambdanTK}\ .
\end{align}
Moreover, we have
\begin{equation} \label{Dec_Lambdan0K}
2\log\left( \frac{L_n^{(0)}(\hat{\pb}_{n,T}; \bbb)}{L_n^{(0)}(\pb_0; \bbb)} \right) = n\left\{\ ^t\epsilonb_n \mathbb{I}(\pb_0)\epsilonb_n  + O(|\epsilonb_n|^3)\right\} 
\end{equation}
where $\mathbb{I}(\pb_0)$ denotes the Fisher information matrix of the Multinomial model which, for $k=3$, reads
$$
\begin{pmatrix}
\dfrac{1}{p_{0,1}} + \dfrac{1}{1-p_{0,1}-p_{0,2}} & \dfrac{1}{1-p_{0,1}-p_{0,2}} \\
\dfrac{1}{1-p_{0,1}-p_{0,2}} & \dfrac{1}{p_{0,2}} + \dfrac{1}{1-p_{0,1}-p_{0,2}} 
\end{pmatrix}
$$
and
\begin{align}
2\log\left( \frac{H_n^{(\ep,m)}(\hat{\pb}_{n,T}; \bbb)}{H_n^{(\ep,m)}(\pb_0; \bbb)} \right) &= -2 \ ^t\epsilonb_n \mathbb{H}_2^{(\ep,m)}(\pb_0;\pb_0) \epsilonb_n + O(|\epsilonb_n|^3) \label{Dec_HnTK} \\
&= -\mathsf{Var}[L_1]\ ^t\epsilonb_n \mathbb{I}(\pb_0)^2 \epsilonb_n + O(|\epsilonb_n|^3) \ . \nonumber 
\end{align}
Therefore, putting $\boldsymbol{\zeta}_n := \sqrt{n} \epsilonb_n$, we get
\begin{align} 
\Lambda_{k,n,T}(\pb_0; \bbb) &= \ ^t\boldsymbol{\zeta}_n \left(\mathbb{I}(\pb_0) - \frac{\mathsf{Var}[L_1]}{n}\mathbb{I}(\pb_0)^2 \right) \boldsymbol{\zeta}_n + O\left(\frac{|\boldsymbol{\zeta}_n|^3}{\sqrt{n}}\right)
\label{final_expansion_LambdanTK} \\
\Lambda_{k,n,N}(\pb_0; \bbb) &= \ ^t\boldsymbol{\zeta}_n \mathbb{I}(\pb_0) \boldsymbol{\zeta}_n + O\left(\frac{|\boldsymbol{\zeta}_n|^3}{\sqrt{n}}\right)\ .
\label{final_expansion_LambdanNK} 
\end{align}
\end{lem}

\begin{proof}
Since 
$$
\Lambda_{k,n,T}(\pb_0; \bbb) := -2\log\left(\frac{\sup_{\pb\in \{\pb_0\}} L_{n,T}(\pb;\bbb)}{\sup_{\pb\in \Delta_{k-1}} L_{n,T}(\pb; \bbb)} \right) = -2\log\left( \frac{L_{n,T}(\pb_0;\bbb)}{L_{n,T}(\hat{\pb}_{n,T}; \bbb)} \right)
$$ 
holds by definition, identity \eqref{Dec_LambdanTK} follows immediately from \eqref{LikTrueDecK}. 

Next, we derive \eqref{Dec_Lambdan0K} by combining \eqref{Definition_Ln0K} with 
\eqref{Expansion_MLE_trueK}. In any case, a quicker argument can be based on Remark \ref{rmk:Taylor}, according to which the right-hand side of \eqref{Dec_Lambdan0K} coincides with the Taylor expansion
of the map $\epsilonb_n \mapsto \mathcal D_{KL}(\pb_0 + \epsilonb_n\| \pb_0)$. Thus, the result to be proved boils down to the application of the well-known relationship between the Kullback-Leibler divergence 
and the Fisher information matrix in a regular parametric model. 

Coming to \eqref{Dec_HnTK}, we start by noticing that
\begin{align*}
\left(\frac{1-\hat{p}_{n,T,1}}{\hat{p}_{n,T,1}}\right)^{l_1} &= \left(\frac{1-p_{0,1}-p_{0,2}}{p_{0,1}}\right)^{l_1} \left[1 - l_1\left(\frac{\epsilon_{n,1}}{p_{0,1}}  + \frac{\epsilon_{n,1} + \epsilon_{n,2}}{1 - p_{0,1} - p_{0,2}}\right) \right. \\
& + l_1\left(\frac{\epsilon_{n,1}^2}{p_{0,1}^2}  + \frac{\epsilon_{n,1}(\epsilon_{n,1} + \epsilon_{n,2})}{p_{0,1}(1 - p_{0,1} - p_{0,2})}\right) \\
&\left. + \binom{l_1}{2} \left(\frac{\epsilon_{n,1}}{p_{0,1}}  + \frac{\epsilon_{n,1} + \epsilon_{n,2}}{1 - p_{0,1} - p_{0,2}}  \right)^2 + O(|\epsilonb_n|^3) \right] \\
\left(\frac{1-\hat{p}_{n,T,2}}{\hat{p}_{n,T,2}}\right)^{l_2} &= \left(\frac{1-p_{0,1}-p_{0,2}}{p_{0,2}}\right)^{l_2} \left[1 - l_2\left(\frac{\epsilon_{n,2}}{p_{0,2}}  + \frac{\epsilon_{n,1} + \epsilon_{n,2}}{1 - p_{0,1} - p_{0,2}}\right) \right. \\
& + l_2\left(\frac{\epsilon_{n,2}^2}{p_{0,2}^2}  + \frac{\epsilon_{n,2}(\epsilon_{n,1} + \epsilon_{n,2})}{p_{0,2}(1 - p_{0,1} - p_{0,2})}\right) \\
&\left. + \binom{l_2}{2} \left(\frac{\epsilon_{n,2}}{p_{0,2}}  + \frac{\epsilon_{n,1} + \epsilon_{n,2}}{1 - p_{0,1} - p_{0,2}}  \right)^2 + O(|\epsilonb_n|^3) \right] \ .
\end{align*}
These identities, combined with \eqref{Definition_HnTK} and \eqref{exp_rhonK}, yield
$$
H_n^{(\ep,m)}(\hat{\pb}_{n,T}; \bbb) = 1 + O\left(\frac 1n + |\epsilonb_n|^3\right)\ .
$$ 
The combination of this identity with \eqref{HnT_ep_mK}, in which $\pb = \xib = \pb_0$, yields \eqref{Dec_HnTK}. Incidentally, it is crucial in what follows to notice that $\mathbb{H}_2^{(\ep,m)}(\pb_0;\pb_0) =
\mathsf{Var}[L_1]\mathbb{I}(\pb_0)^2$, but this is only a matter of direct computations.

Finally, identity \eqref{final_expansion_LambdanTK} follows immediately from \eqref{Dec_LambdanTK}-\eqref{Dec_Lambdan0K}-\eqref{Dec_HnTK}.
\end{proof}

%%%%%%%%%%%%%%%%%%%%%%%%%%%%%%%%%%%%%%%%%%%%%%%%%%%%%%%%%%%%%%%%%%%%%%%%%%%
%%%%%%%%%%%%%%%%%%%%%%%%%%%%%%%%%%%%%%%%%%%%%%%%%%%%%%%%%%%%%%%%%%%%%%%
%%%%%%%%%%%%%%%%%%%%%%%%%%%%%%%%%%%%%%%%%%%%%%%%%%%%%%%%%%%%%%%%%%%%%%%%%%%%%%%%%%%%%%
\subsection{Proof of Theorems \ref{teo32} and \ref{teo32n}}
We start again from \eqref{ApiuLK} where the random variables $\Abb_{n}$ and $\mathbf L$, defined on the probability space $(\Omega,\mathscr{F},\text{Pr})$, are independent, 
$\Abb_n \sim Mult(n,\pb)$ and $\mathbf L = (L_1, \dots, L_{k-1})$ 
is a random vector with independent components with each $L_i$ having the Laplace distribution \eqref{Laplace}. 
Thus, under the validity of $H_0$, we define the event $E_{0,n}\subset\Omega$ as
\begin{equation}\label{event}
E_{0,n,k}:=\left\{\omega\in\Omega\text{ $:$ }\left|\frac{\Abb_{n}(\omega)}{n} - \pb_0\right|\leq c(\pb_0)\left(\frac{\log n}{n}\right)^{1/2}\right\}
\end{equation}
for some $c(p_0)>0$. After recalling that each component of $\Abb_{n}$ is a Binomial random variable, we can reduce the problem to the analogous one already treated in Section \ref{sect:proof_teo31},
and we can find $c(\pb_0)$ so that
\begin{equation}\label{bound1K}
\mathrm{Pr}[E_{0,n,k}^c] \leq 2(k-1)n^{-3/2}.
\end{equation}
In this way, for $\ell=T,N$, we can write
$$
\F_{k,n,\ell}(t)= \text{Pr}[\Lambda_{k,n,\ell}(\pb_{0}) \leq t, E_{0,n,k}] + \text{Pr}[\Lambda_{k,n,\ell}(\pb_{0}) \leq t, E_{0,n,k}^c]
$$
with
$$
\mathrm{Pr}[\Lambda_{k,n,\ell}(\pb_{0}) \leq t, E_{0,n,k}^c] \leq 2(k-1)n^{-3/2}.
$$
The advantage of such a preliminary step is that, on $E_{0,n}$, the assumptions of all the Lemmata stated in the previous subsection are fulfilled and, by resorting to \eqref{final_expansion_LambdanTK}, we 
can write
\begin{align*}
&\left\{\Lambda_{k,n,T}(\pb_{0}) \leq t, E_{0,n,k}\right\} \\
&= \left\{\ ^t\Zb_n \left(\mathbb{I}(\pb_0) - \frac{\mathsf{Var}[L_1]}{n}\mathbb{I}(\pb_0)^2 \right) \Zb_n + O\left(\frac{|\Zb_n|^3}{\sqrt{n}}\right)\leq t, E_{0,n,k}\right\}
\end{align*}
where
\begin{equation} \label{eq:VnK}
\mathbf Z_n := \frac{\mathbf B_n - n\pb_0}{\sqrt{n}} =  \Sigma(\pb_0)^{1/2} \mathbf W_n + \frac{1}{\sqrt{n}} \mathbf L
\end{equation}
with
$$
\mathbf W_n := \Sigma(\pb_0)^{-1/2} \frac{\Abb_n - n\pb_0}{\sqrt{n}}
$$
and $\Sigma(\pb_0)$ is the covariance matrix of $\Abb_n$. 

As in the preliminary part of this second Appendix, we confine ourselves to dealing with the case of $k=3$. The following proposition represents the multidimensional analogous of the Berry-Esseen expansion already 
given in Proposition \ref{prp:BE}.

\begin{proposition} \label{prp:BEK}
Let $\phi_{\mathbf 0, \boldsymbol \Sigma}$ denote the density of the 2-dimensional Normal distribution with mean equal to $\mathbf 0$ and covariance matrix equal to $\Sigma(\pb_0)$, 
and let $\Phi_{\mathbf{0}, \boldsymbol{\Sigma}}$ stand for the associated distribution function. In addition, for any multi-index $\boldsymbol{\nu}$, denote by $\chi_{\boldsymbol{\nu}}$ the
$\boldsymbol{\nu}$-cumulant of the random vector $(A_{n,1} - p_{0,1}, A_{n,2} - p_{0,2})$, and put 
$\chi_m(\ub) := \sum_{|\boldsymbol{\nu}|=m} \frac{\chi_{\boldsymbol{\nu}}}{\boldsymbol{\nu}!} \ub^{\boldsymbol{\nu}}$ to set
$$
\tilde{P}_s(\ub : \{\chi_{\boldsymbol{\nu}}\}) := \sum_{(\ast)} \prod_{m=1}^s \frac{1}{k_m!} \left(\chi_{m+2}(\ub)\right)^{k_m} 
$$
where $(\ast)$ is a shorthand to indicate that the summation is extended to all the $s$-uples $(k_1, \dots, k_s)$ such that $k_1 + 2k_2 + \dots + sk_s = s$. Afterwords, define
$$
P_s(-\phi_{\mathbf{0}, \boldsymbol{\Sigma}} : \{\chi_{\boldsymbol{\nu}}\})(\yb) := \left(\frac{1}{2\pi}\right)^k \int_{\R^k} e^{-i \ub \cdot \yb} \tilde{P}_s(\ub : \{\chi_{\boldsymbol{\nu}}\}) \ud\ub
$$
and $P_s(-\Phi_{\mathbf{0}, \boldsymbol{\Sigma}} : \{\chi_{\boldsymbol{\nu}}\})(\xb) := \int_{-\infty}^{x_1}\int_{-\infty}^{x_2} P_s(-\phi_{\mathbf{0}, \boldsymbol{\Sigma}} : \{\chi_{\boldsymbol{\nu}}\})(\yb) \ud \yb$. 
Finally, set
$$
\Xi_1(\yb) := P_1(-\Phi_{\mathbf{0}, \boldsymbol{\Sigma}} : \{\chi_{\boldsymbol{\nu}}\})(\yb) - S_1(np_{0,1} + \sqrt{n}y_1) \partial_1\Phi_{\mathbf{0}, \boldsymbol{\Sigma}}(\yb) - S_1(np_{0,2} + \sqrt{n}y_2) \partial_2\Phi_{\mathbf{0}, \boldsymbol{\Sigma}}(\yb)
$$
and
\begin{align*}
\Xi_2(\yb) &= P_2(-\Phi_{\mathbf{0}, \boldsymbol{\Sigma}} : \{\chi_{\boldsymbol{\nu}}\})(\yb) \\
&- S_1(np_{0,1} + \sqrt{n}y_1) \partial_1P_1(-\Phi_{\mathbf{0}, \boldsymbol{\Sigma}} : \{\chi_{\boldsymbol{\nu}}\})(\yb)
- S_1(np_{0,2} + \sqrt{n}y_2) \partial_2P_1(-\Phi_{\mathbf{0}, \boldsymbol{\Sigma}} : \{\chi_{\boldsymbol{\nu}}\})(\yb) \\
& + S_2(np_{0,1} + \sqrt{n}y_1)\partial_1^2\Phi_{\mathbf{0}, \boldsymbol{\Sigma}}(\yb) + S_2(np_{0,2} + \sqrt{n}y_2)\partial_2^2\Phi_{\mathbf{0}, \boldsymbol{\Sigma}}(\yb) \\
&+ S_1(np_{0,1} + \sqrt{n}y_1)S_1(np_{0,2} + \sqrt{n}y_2)\partial_{1,2}^2\Phi_{\mathbf{0}, \boldsymbol{\Sigma}}(\yb)
\end{align*}
where $S_1$ and $S_2$ are the same as in Proposition \ref{prp:BE}. Then, for any $\xb \in \R^2$, there holds
\begin{equation} \label{BEK}
\G_n(\xb) = \Phi_{\mathbf{0}, \boldsymbol{\Sigma}}(\xb) +\frac{1}{\sqrt{n}} \Xi_1(\xb) + \frac{1}{2n} \mathsf{Var}(L_1) \left[\partial_1^2\Phi_{\mathbf{0}, \boldsymbol{\Sigma}}(\xb) + \partial_2^2\Phi_{\mathbf{0}, \boldsymbol{\Sigma}}(\xb) \right] + \frac{1}{n} \Xi_2(\xb) + R_n(\xb)
\end{equation}
with the remainder term $R_n$ satisfying an inequality like 
$$
\sup_{\xb \in \R^2} \{ |R_n(\xb)| (1+|\xb|)^5\} \leq \frac{C(\pb_0, m, \varepsilon)}{n^{3/2}}
$$ 
for some suitable constant $C(\pb_0, m, \varepsilon)$ independent of $n$. 
\end{proposition}

\begin{proof}
We start by putting 
$$
\F_n(\xb) := \text{Pr}\left[\sqrt{n}\left(\frac1n \Abn - \mathbf{p}_0\right) \leq \xb\right] = \text{Pr}\left[\frac{1}{\sqrt{n}}\left(\Abn - n\mathbf{p}_0\right) \leq \xb\right]\ . 
$$
By independence, we have
\begin{align*}
\G_n(\xb) &= \left(\frac{1}{\cepsm}\right)^2 \sum_{l_1 = -m}^m \sum_{l_2 = -m}^m e^{-\varepsilon (|l_1| + |l_2|)} \F_n\left(x_1 - \frac{l_1}{\sqrt{n}}, x_2 - \frac{l_2}{\sqrt{n}}\right) \\
&= \E\left[\F_n\left(\xb - \frac{1}{\sqrt{n}}\Lb\right)\right]\ .
\end{align*}
At this stage, we exploit the well-known asymptotic expansions of $\F_n$ displayed, e.g., in Section 23 of \citet{Bhatt(10)}. We put $s=5$ in Theorem 23.1 to obtain
$$
\F_n(\yb) = \Phi_{\mathbf{0}, \boldsymbol{\Sigma}}(\yb) + \frac{1}{\sqrt{n}} \Xi_1(\yb) +  \frac{1}{n} \Xi_2(\yb) + R_n(\xb)
$$
where $R_n$ is a remainder term satisfying an inequality like 
$$
\sup_{\xb \in \R^2} \{ |R_n(\xb)| (1+|\xb|)^5\} \leq \frac{C(\pb_0, m, \varepsilon)}{n^{3/2}}
$$ 
for some suitable constant $C(\pb_0, m, \varepsilon)$ independent of $n$. 
The key remark is about the modification of the terms of the type $S_k(np_{0,i} + \sqrt{n}y_i)$ after the substitution $y_i = x_i - \frac{l_i}{\sqrt{n}}$. Indeed, we have
$$
S_k(np_{0,i} + \sqrt{n}y_i) = S_k(np_{0,i} + \sqrt{n}x_i - l_i) = S_k(np_{0,i} + \sqrt{n}x_i)
$$
because of the periodicity of the $S_k$'s. Therefore, the substitution $y_i = x_i - \frac{l_i}{\sqrt{n}}$ affects the functions $\Xi_1$ and  $\Xi_2$ only in the terms involving $\Phi_{\mathbf{0}, \boldsymbol{\Sigma}}$, $P_1(-\Phi_{\mathbf{0}, \boldsymbol{\Sigma}} : \{\chi_{\boldsymbol{\nu}}\})$, $P_2(-\Phi_{\mathbf{0}, \boldsymbol{\Sigma}} : \{\chi_{\boldsymbol{\nu}}\})$, $P_3(-\Phi_{\mathbf{0}, \boldsymbol{\Sigma}} : \{\chi_{\boldsymbol{\nu}}\})$
and their derivatives. Since these functions are smooth, we can apply the Taylor formula to get
\begin{align*}
D^{\boldsymbol{\alpha}}\Phi_{\mathbf{0}, \boldsymbol{\Sigma}}\left(\xb - \frac{1}{\sqrt{n}}\mathbf{l}\right) &= D^{\boldsymbol{\alpha}}\Phi_{\mathbf{0}, \boldsymbol{\Sigma}}(\xb) - \frac{l_1}{\sqrt{n}} \partial_1D^{\boldsymbol{\alpha}}\Phi_{\mathbf{0}, \boldsymbol{\Sigma}}(\xb) - \frac{l_2}{\sqrt{n}} \partial_2D^{\boldsymbol{\alpha}}\Phi_{\mathbf{0}, \boldsymbol{\Sigma}}(\xb) \\
&+ \frac{l_1^2}{2n} \partial_1^2D^{\boldsymbol{\alpha}}\Phi_{\mathbf{0}, \boldsymbol{\Sigma}}(\xb) + \frac{l_2^2}{2n} \partial_2^2D^{\boldsymbol{\alpha}}\Phi_{\mathbf{0}, \boldsymbol{\Sigma}}(\xb)
+ \frac{l_1 l_2}{n} \partial_{1,2}^2D^{\boldsymbol{\alpha}}\Phi_{\mathbf{0}, \boldsymbol{\Sigma}}(\xb) \\
& + O\left(\frac{1}{n^{3/2}}\right)
\end{align*}
for any multi-index $\boldsymbol{\alpha}$, and analogous expansions for $D^{\boldsymbol{\alpha}}P_k(-\Phi_{\mathbf{0}, \boldsymbol{\Sigma}} : \{\chi_{\boldsymbol{\nu}}\})\left(\xb - \frac{1}{\sqrt{n}}\mathbf{l}\right)$.
Now, the assumption on the distribution of $\Lb$ entails that $\E[\Lb^{\boldsymbol{\alpha}}] = \E[L_1^{\alpha_1} L_2^{\alpha_2}] = 0$ as soon as either $\alpha_1$ or $\alpha_2$ is odd.
Whence,
\begin{align*}
&\E\left[D^{\boldsymbol{\alpha}}\Phi_{\mathbf{0}, \boldsymbol{\Sigma}}\left(\xb - \frac{1}{\sqrt{n}}\Lb\right)\right] \\
&= D^{\boldsymbol{\alpha}}\Phi_{\mathbf{0}, \boldsymbol{\Sigma}}(\xb) + \frac{1}{2n} \mathsf{Var}(L_1) 
\left[\partial_1^2D^{\boldsymbol{\alpha}}\Phi_{\mathbf{0}, \boldsymbol{\Sigma}}(\xb) + \partial_2^2D^{\boldsymbol{\alpha}}\Phi_{\mathbf{0}, \boldsymbol{\Sigma}}(\xb) \right] + O\left(\frac{1}{n^{3/2}}\right)
\end{align*}
for any multi-index $\boldsymbol{\alpha}$, and analogous expansions for the term
$$
\E\left[D^{\boldsymbol{\alpha}}P_k(-\Phi_{\mathbf{0}, \boldsymbol{\Sigma}} : \{\chi_{\boldsymbol{\nu}}\})\left(\xb - \frac{1}{\sqrt{n}}\Lb\right)\right]\ .
$$
This completes the proof. 
\end{proof}

At this stage, like in the previous proof of Theorem \ref{teo31}, the thesis of Theorems \ref{teo32} and \ref{teo32n} follow from the combination of the Berry-Esseen expansion \eqref{BEK} with identity \eqref{final_expansion_LambdanTK}. Indeed, \eqref{BEK} shows that the probability distribution of $\mathbf Z_n$ has both an absolutely continuous part as well as a singular part. As before, the core of the 
proof hinges on the study of the absolutely continuous part, which is significantly affected by the Laplace perturbation because of the term $\frac{1}{2n} \mathsf{Var}[L_1] \left[\partial_1^2\Phi_{\mathbf{0}, \boldsymbol{\Sigma}}(\xb) + \partial_2^2\Phi_{\mathbf{0}, \boldsymbol{\Sigma}}(\xb)\right]$. Thus, the sum 
$$
\Phi_{\mathbf{0}, \boldsymbol{\Sigma}}(\xb) + \frac{1}{2n} \mathsf{Var}[L_1] \left[\partial_1^2\Phi_{\mathbf{0}, \boldsymbol{\Sigma}}(\xb) + \partial_2^2\Phi_{\mathbf{0}, \boldsymbol{\Sigma}}(\xb)\right]
$$
yields the following density (with respect to the 2-dimensional Lebesgue measure)
\begin{align*}
g_n(\xb) &= \varphi_{\boldsymbol \Sigma(\pb_0)}(\xb) + \frac{\mathsf{Var}[L_1]}{2n} \Delta_{\xb} \varphi_{\boldsymbol \Sigma(\pb_0)}(\xb) \\
&= \varphi_{\boldsymbol \Sigma(\pb_0)}(\xb) \left[1 - \frac{\mathsf{Var}[L_1]}{2n}\mathrm{tr}\left(\boldsymbol \Sigma(\pb_0)^{-1}\right) + \frac{\mathsf{Var}[L_1]}{2n} \ ^t \xb \Sigma(\pb_0)^{-2} \xb\right]
\end{align*}
where $\varphi_{\boldsymbol \Sigma}$ denotes the density of the 2-dimensional Normal distribution with mean $\mathbf 0$ and covariance matrix $\boldsymbol \Sigma$, $\boldsymbol  \Sigma(\pb_0)$ 
is the covariance matrix of the random vector $\Abb_n \sim Mult(n, \pb_0)$, $\Delta_{\xb}$ stands for the Laplacian operator, and $\mathrm{tr}$ is the trace operator. At this stage, it is enough to notice that
the principal term of the expression of $\F_{k,n,T}(t)$ is given by the following integral
\begin{align*}
&\int_{\left\{ ^t \xb\left(\mathbb{I}(\pb_0) - \frac{\mathsf{Var}[L_1]}{n}\mathbb{I}(\pb_0)^2 \right) \xb \leq t \right\}} g_n(\xb) \ddr\xb \\
&= \int_{\{ |\yb|^2 \leq t \}} g_n\left(\left(\mathbb{I}(\pb_0) - \frac{\mathsf{Var}[L_1]}{n}\mathbb{I}(\pb_0)^2 \right)^{-1/2} \yb \right) \frac{\ddr\yb}{\sqrt{\mathrm{det}\left(\mathbb{I}(\pb_0) - \frac{\mathsf{Var}[L_1]}{n}\mathbb{I}(\pb_0)^2\right)}}\ .
\end{align*}
To handle the last integral, we start by noticing that 
$$
\left(\mathbb{I}(\pb_0) - \frac{\mathsf{Var}[L_1]}{n}\mathbb{I}(\pb_0)^2 \right)^{-1/2} = \mathbb{I}(\pb_0)^{-1/2} + \frac{\mathsf{Var}[L_1]}{2n}\mathbb{I}(\pb_0)^{1/2} + O\left(\frac{1}{n^2} \right)\ . 
$$
Thus, combining the last identities, we get
\begin{align*}
&g_n\left(\left(\mathbb{I}(\pb_0) - \frac{\mathsf{Var}[L_1]}{n}\mathbb{I}(\pb_0)^2 \right)^{-1/2} \yb \right) \\
&= g_n\left(\mathbb{I}(\pb_0)^{-1/2}\yb + \frac{\mathsf{Var}[L_1]}{2n}\mathbb{I}(\pb_0)^{1/2}\yb\right) + O\left(\frac{1}{n^2} \right)\\
&= \varphi_{\boldsymbol \Sigma(\pb_0)}\left(\mathbb{I}(\pb_0)^{-1/2}\yb + \frac{\mathsf{Var}[L_1]}{2n}\mathbb{I}(\pb_0)^{1/2}\yb\right) \times \\
&\times \left[1 - \frac{\mathsf{Var}[L_1]}{2n}\mathrm{tr}\left(\boldsymbol \Sigma(\pb_0)^{-1}\right) + \frac{\mathsf{Var}[L_1]}{2n} \ ^t \left(\mathbb{I}(\pb_0)^{-1/2}\yb \right) \boldsymbol{\Sigma}(\pb_0)^{-2} 
\left(\mathbb{I}(\pb_0)^{-1/2}\yb \right)\right] + O\left(\frac{1}{n^2} \right)\ .
\end{align*}
In view of the identity $\nabla_{\xb}\varphi_{\boldsymbol \Sigma}(\xb) = -\varphi_{\boldsymbol \Sigma}(\xb) \boldsymbol \Sigma^{-1}\xb$, we can use the Taylor formula to write
\begin{align*}
&\varphi_{\boldsymbol \Sigma(\pb_0)}\left(\mathbb{I}(\pb_0)^{-1/2}\yb + \frac{\mathsf{Var}[L_1]}{2n}\mathbb{I}(\pb_0)^{1/2}\yb\right) \\
&=\varphi_{\boldsymbol \Sigma(\pb_0)}\left(\mathbb{I}(\pb_0)^{-1/2}\yb\right)
+ \frac{\mathsf{Var}[L_1]}{2n}\mathbb{I}(\pb_0)^{1/2}\yb \cdot \nabla \varphi_{\boldsymbol \Sigma(\pb_0)}\left(\mathbb{I}(\pb_0)^{-1/2}\yb\right) + O\left(\frac{1}{n^2} \right)\\
&=\varphi_{\boldsymbol \Sigma(\pb_0)}\left(\mathbb{I}(\pb_0)^{-1/2}\yb\right) \left[1 -  \frac{\mathsf{Var}[L_1]}{2n}\mathbb{I}(\pb_0)^{1/2}\yb \cdot \left(\boldsymbol \Sigma^{-1}(\pb_0)\mathbb{I}(\pb_0)^{-1/2}\yb \right)\right]\ .
\end{align*}
Now, it is crucial to observe that, for the Multinomial model, it holds $\mathbb{I}(\pb_0)^{-1} = \Sigma(\pb_0)$. Whence,
\begin{align*}
\varphi_{\boldsymbol \Sigma(\pb_0)}\left(\mathbb{I}(\pb_0)^{-1/2}\yb\right) &= \frac{1}{\sqrt{\mathrm{det}(\boldsymbol \Sigma(\pb_0))}} \varphi(\yb) \\
^t \left(\mathbb{I}(\pb_0)^{-1/2}\yb \right) \boldsymbol{\Sigma}(\pb_0)^{-2} \left(\mathbb{I}(\pb_0)^{-1/2}\yb \right) &=\ ^t\yb \mathbb{I}(\pb_0) \yb \\
\mathbb{I}(\pb_0)^{1/2}\yb \cdot \left(\boldsymbol \Sigma^{-1}(\pb_0)\mathbb{I}(\pb_0)^{-1/2}\yb \right) &=\ ^t\yb \mathbb{I}(\pb_0) \yb
\end{align*}
where $\varphi$ stands for the density of the standard Normal distribution. In conclusion, we have
\begin{align*}
&g_n\left(\left(\mathbb{I}(\pb_0) - \frac{\mathsf{Var}[L_1]}{n}\mathbb{I}(\pb_0)^2 \right)^{-1/2} \yb \right) \\
&= \frac{\varphi(\yb) }{\sqrt{\mathrm{det}(\boldsymbol \Sigma(\pb_0))}} \left[1 -  \frac{\mathsf{Var}[L_1]}{2n}\ ^t\yb \mathbb{I}(\pb_0) \yb\right]  \times \\
&\times \left[1 - \frac{\mathsf{Var}[L_1]}{2n}\mathrm{tr}\left(\boldsymbol \Sigma(\pb_0)^{-1}\right) + \frac{\mathsf{Var}[L_1]}{2n}\ ^t\yb \mathbb{I}(\pb_0) \yb \right] + O\left(\frac{1}{n^2} \right) \\
&= \frac{\varphi(\yb) }{\sqrt{\mathrm{det}(\boldsymbol \Sigma(\pb_0))}} \left[1 - \frac{\mathsf{Var}[L_1]}{2n}\mathrm{tr}\left(\boldsymbol \Sigma(\pb_0)^{-1}\right) \right] + O\left(\frac{1}{n^2} \right)\ .
\end{align*}
Then, invoking the Taylor formula for the determinant operator, we get
$$
\frac{1}{\sqrt{\mathrm{det}\left(\mathbb{I}(\pb_0) - \frac{\mathsf{Var}[L_1]}{n}\mathbb{I}(\pb_0)^2\right)}} = \sqrt{\mathrm{det}(\boldsymbol \Sigma(\pb_0))} 
\left[1 + \frac{\mathsf{Var}[L_1]}{2n}\mathrm{tr}\left(\boldsymbol \Sigma(\pb_0)^{-1}\right) \right] + O\left(\frac{1}{n^2} \right)
$$
which entails 
\begin{align*}
\int_{\left\{ ^t \xb\left(\mathbb{I}(\pb_0) - \frac{\mathsf{Var}[L_1]}{n}\mathbb{I}(\pb_0)^2 \right) \xb \leq t \right\}} g_n(\xb) \ddr\xb &= \int_{\{ |\yb|^2 \leq t \}} \varphi(\yb)\ddr\yb + O\left(\frac{1}{n^2} \right)\\
&= \K_2(t) + O\left(\frac{1}{n^2} \right)\ .
\end{align*}
This argument proves rigorously the presence pf the terms $\K_{k-1}(t)$ in \eqref{BE_LRTk}. Now, we take cognizance that the expressions of $c_{k,1}(t;\pb_{0})$, $c_{k,\ast}(t;\pb_0; m,\varepsilon)$,
$c_{k,2}(t;\mathbf{p}_{0})$ and $R_{k,n,T}(t)$ in \eqref{BE_LRTk} ensue from the various terms in \eqref{BEK}, according to the above line of reasoning. In particular, $c_{k,1}(t;\pb_{0})$ ensues from
the term $P_1(-\Phi_{\mathbf{0}, \boldsymbol{\Sigma}} : \{\chi_{\boldsymbol{\nu}}\})(\yb)$ that figures in the expression of $\Xi_1$, while $c_{k,\ast}(t;\pb_0; m,\varepsilon)$ corresponds to the manipulation of the
quantity $- S_1(np_{0,1} + \sqrt{n}y_1) \partial_1\Phi_{\mathbf{0}, \boldsymbol{\Sigma}}(\yb) - S_1(np_{0,2} + \sqrt{n}y_2) \partial_2\Phi_{\mathbf{0}, \boldsymbol{\Sigma}}(\yb)$ which also figures in the 
expression of $\Xi_1$. Furthermore, $c_{k,2}(t;\mathbf{p}_{0})$ ensues from the expression of $\Xi_2$. However, what is crucial is just to remark once again is that the additional term
$$
\frac{1}{2n} \mathsf{Var}[L_1] \left[\partial_1^2\Phi_{\mathbf{0}, \boldsymbol{\Sigma}}(\xb) + \partial_2^2\Phi_{\mathbf{0}, \boldsymbol{\Sigma}}(\xb) \right]
$$
that appears in \eqref{BEK}, which depends significantly on the Laplace perturbation, is no more present in \eqref{BE_LRTk}.

It remains to justify the expression of $\F_{k,n,N}$, which follows from the combination of the Berry-Esseen expansion \eqref{BEK} with identity \eqref{final_expansion_LambdanNK}. Indeed, 
the core of the argument consists in the study of the integral $\int_{\{ ^t \xb\mathbb{I}(\pb_0)\xb \leq t \}} g_n(\xb) \ddr\xb$, which equal to
\begin{align*}
&\int_{\{ |\yb|^2 \leq t \}} g_n\left(\mathbb{I}(\pb_0)^{-1/2} \yb \right) \frac{\ddr\yb}{\sqrt{\mathrm{det}\left(\mathbb{I}(\pb_0)\right)}}\\
&= \int_{\{ |\yb|^2 \leq t \}} \varphi_{\boldsymbol \Sigma(\pb_0)}\left(\mathbb{I}(\pb_0)^{-1/2} \yb \right) \left[1 - \frac{\mathsf{Var}[L_1]}{2n}\mathrm{tr}\left(\boldsymbol \Sigma(\pb_0)^{-1}\right) \right. \\
&\left. + \frac{\mathsf{Var}[L_1]}{2n} \ ^t\left(\mathbb{I}(\pb_0)^{-1/2} \yb \right)  \Sigma(\pb_0)^{-2} \left(\mathbb{I}(\pb_0)^{-1/2} \yb \right) \right]\frac{\ddr\yb}{\sqrt{\mathrm{det}\left(\mathbb{I}(\pb_0)\right)}}+ O\left(\frac{1}{n^2} \right) \\
&= \int_{\{ |\yb|^2 \leq t \}} \varphi(\yb) \left[1 - \frac{\mathsf{Var}[L_1]}{2n}\mathrm{tr}\left(\boldsymbol \Sigma(\pb_0)^{-1}\right) + \frac{\mathsf{Var}[L_1]}{2n}\ ^t\yb \mathbb{I}(\pb_0) \yb \right] \ddr\yb
+ O\left(\frac{1}{n^2} \right)\ .
\end{align*}
Then, we can pass to polar coordinates and notice that
$$
\int_{\mathbb S(k-2)}\ ^t \boldsymbol \sigma \mathbb{I}(\pb_0) \boldsymbol \sigma \ddr \boldsymbol\sigma = \mathrm{tr}(\mathbb{I}(\pb_0)) \frac{|\mathbb S(k-2)|}{k-1}
$$
with $|\mathbb S(k-2)| = \frac{2\pi^{(k-1)/2}}{\Gamma\left(\frac{k-1}{2}\right)}$. Finally, we have
\begin{align*}
&\frac{\mathsf{Var}[L_1]}{2n} \mathrm{tr}(\mathbb{I}(\pb_0)) \int_0^t \frac{\left(\frac 12\right)^{(k-1)/2}}{\Gamma\left(\frac{k-1}{2}\right)} e^{-u/2} u^{(k-3)/2} \left(-1 + \frac{u}{k-1}\right) \ddr u \\
&= - \frac{\mathsf{Var}[L_1]}{n} \mathrm{tr}(\mathbb I(\pb_0)) \frac{\left(\frac 12\right)^{(k+1)/2}}{\Gamma\left(\frac{k+1}{2}\right)}\text{e}^{-t/2} t^{(k-1)/2}
\end{align*}
which completes the proof of \eqref{BE_LRNk}.

%%%%%%%%%%%%%%%%%%%%%%%%%%%%%%%%%%%%%%%%%%%%%%%%%%%%%%%%%%%%
%%%%%%%%%%%%%%%%%%%%%%%%%%%%%%%%%%%%%%%%%%%%%%%%%%%%%%%%%%%%%
%%%%%%%%%%%%%%%%%%%%%%%%%%%%%%%%%%%%%%%%%%%%%%%%%%%%%%%%%%%%%
%%%                                            LARGE   DEVIATIONS.    MULTID                     %%%%%%%%%%%%%%%%%%%%%
%%%%%%%%%%%%%%%%%%%%%%%%%%%%%%%%%%%%%%%%%%%%%%%%%%%%%%%%%%%%%%
%%%%%%%%%%%%%%%%%%%%%%%%%%%%%%%%%%%%%%%%%%%%%%%%%%%%%%%%%%%%%%
%%%%%%%%%%%%%%%%%%%%%%%%%%%%%%%%%%%%%%%%%%%%%%%%%%%%%%%%%%%%%%%%

\subsection{Proof of \eqref{main_ld}} Upon putting $\mu_n(\cdot) := \mathrm{Pr} \left[\frac{1}{\sqrt{n}} \sum_{i=1}^n \mathbf X_i \in \cdot \right]$, we prove that 
\begin{align}
& \mu_n( \sqrt{n} \boldsymbol \xi + \mathscr E)\nonumber \\
&= \exp\left\{-n[\hat{\mathbf z} \cdot \nabla L(\hat{\mathbf z}) - L(\hat{\mathbf z})] + \sqrt{n} 
\min_{\mathbf v \in \mathscr E} \hat{\mathbf z} \cdot \mathbf v\right\} n^{-(d+1)/4}
[1 + o(1)] \label{AppB1} 
\end{align}
is valid as $n \rightarrow +\infty$, where $d$ is the dimension of the random vectors $\mathbf X_i$. The proof is a small variation of a classical argument developed, e.g., in Chapter VIII of  
\citet{Pet(75)}, in \citet{vBa(67)}, or in \citet{Ale(83)}.

Letting the distribution of $\Xb_1$ be denoted by $\mu$, we introduce the new probability measure $\nu_{\zb}(A) := e^{-L(\zb)} \int_A e^{\zb \cdot \yb} \mu(\ddr\yb)$ for some $|\zb| < H$ and, then, a sequence
$\{\Yb_i\}_{i \geq 1}$ of i.i.d. random vectors with distribution $\nu_{\zb}$. We also put $\mb_{\zb} := \E[\Yb_1] = \nabla L(\zb)$ and $\Cb_{\zb} := Cov[\Yb_1] = \mathrm{Hess}[L](\zb)$. Hence, setting
$\nu_{\zb,n}(\cdot) := \mathrm{Pr}[\sum_{i=1}^n \Yb_i \in \cdot]$, we find that 
$$
\mathrm{Pr} \left[\sum_{i=1}^n \mathbf X_i \in A\right] = e^{nL(\zb)} \int_A e^{-\zb \cdot \yb} \nu_{\zb,n}(\ddr\yb)
$$
yielding in turn, after some manipulations, that
\begin{equation} \label{AppB2}
\mu_n( \sqrt{n} \boldsymbol \xi + \mathscr E) = e^{nL(\zb) - n\zb\cdot\mb_{\zb}} \int_{\Cb_{\zb}^{-1/2}[\sqrt{n}(\boldsymbol \xi - \mb_{\zb}) + \mathscr E]} 
e^{-\sqrt{n}\ ^t\zb \Cb_{\zb}^{1/2}\yb} \gamma_{\zb,n}(\ddr\yb)
\end{equation}
with
$$
\gamma_{\zb,n}(\cdot) := \mathrm{Pr} \left[\frac{1}{\sqrt{n}} \Cb_{\zb}^{-1/2} \sum_{i=1}^n (\mathbf Y_i - \mb_{\zb}) \in \cdot\right]\ .
$$
At this stage, exploiting the assumption of the existence of $\hat{\zb}$ for which $\nabla L(\hat{\zb}) = \mb_{\hat{\zb}} = \boldsymbol \xi$, we choose $\zb = \hat{\zb}$ so that \eqref{AppB2} becomes
$$
\mu_n( \sqrt{n} \boldsymbol \xi + \mathscr E) = e^{nL(\hat{\zb}) - n\hat{\zb}\cdot\mb_{\hat{\zb}}} \int_{\Cb_{\hat{\zb}}^{-1/2}\mathscr E} 
e^{-\sqrt{n}\ ^t\hat{\zb} \Cb_{\hat{\zb}}^{1/2}\yb} \gamma_{\hat{\zb},n}(\ddr\yb)\ .
$$
Therefore, to prove \eqref{AppB1}, it is enough to show that 
$$
\int_{\Cb_{\hat{\zb}}^{-1/2}\mathscr E} e^{-\sqrt{n}\ ^t\hat{\zb} \Cb_{\hat{\zb}}^{1/2}\yb} \gamma_{\hat{\zb},n}(\ddr\yb) = 
\exp\left\{\sqrt{n} \min_{\mathbf v \in \mathscr E} \hat{\mathbf z} \cdot \mathbf v\right\} n^{-(d+1)/4}[1 + o(1)] 
$$
as $n \rightarrow +\infty$. Indeed, upon denoting by $\gamma$ the standard Normal distribution on $\R^d$, we can write
\begin{align}
&\int_{\Cb_{\hat{\zb}}^{-1/2}\mathscr E} e^{-\sqrt{n}\ ^t\hat{\zb} \Cb_{\hat{\zb}}^{1/2}\yb} \gamma_{\hat{\zb},n}(\ddr\yb) = \left(\frac{1}{2\pi}\right)^{d/2}
\int_{\Cb_{\hat{\zb}}^{-1/2}\mathscr E} e^{-\sqrt{n}\ ^t\hat{\zb} \Cb_{\hat{\zb}}^{1/2}\yb - \frac 12 |\yb|^2}\ddr\yb \nonumber \\
&+ \int_{\R^d} \mathds{1}_{\Cb_{\hat{\zb}}^{-1/2}\mathscr E}(\yb)
e^{-\sqrt{n}\ ^t\hat{\zb} \Cb_{\hat{\zb}}^{1/2}\yb} [\gamma_{\hat{\zb},n}(\ddr\yb) - \gamma(\ddr\yb)]\ . \label{AppB3} 
\end{align}
For the former term on the right-hand side of \eqref{AppB3}, we have
\begin{align*}
&\left(\frac{1}{2\pi}\right)^{d/2}\int_{\Cb_{\hat{\zb}}^{-1/2}\mathscr E} e^{-\sqrt{n}\ ^t\hat{\zb} \Cb_{\hat{\zb}}^{1/2}\yb - \frac 12 |\yb|^2}\ddr\yb \\
&=\left(\frac{1}{2\pi}\right)^{d/2} \frac{1}{\sqrt{\mathrm{det}(\Cb_{\hat{\zb}})}} \int_{\mathscr E} e^{-\sqrt{n} \hat{\zb} \cdot \ub - \frac 12 |\Cb_{\hat{\zb}}^{-1/2} \ub|^2}\ddr\ub \\
&= C(d,\boldsymbol \xi, \mathscr E) \exp\left\{-\sqrt{n} \min_{\ub \in \mathscr E} \hat{\zb} \cdot \ub \right\} n^{-(d+1)/4}[1 + o(1)] 
\end{align*}
where, in the last identity, $C(d,\boldsymbol \xi, \mathscr E)$ is a constant depending solely on $(d,\boldsymbol \xi, \mathscr E)$. For completeness, the validity of such an identity follows from a 
direct application of the multidimensional Laplace method displayed, e.g., in Theorem 46 of \citet{Bre(94)}. It remains to show that the absolute value of the
latter term on the right-hand side of \eqref{AppB3} is even less significant, as it 
can be bounded by an expression like 
\begin{equation}\label{AppB5} 
C'(d,\boldsymbol \xi, \mathscr E) \exp\left\{-\sqrt{n} \min_{\ub \in \mathscr E} \hat{\zb} \cdot \ub \right\} n^{-(d+3)/4}[1 + o(1)]
\end{equation}
with some constant $C'(d,\boldsymbol \xi, \mathscr E)$ depending solely on $(d,\boldsymbol \xi, \mathscr E)$. This task can be carried out by resorting to the Plancherel identity, i.e.
$$
\int_{\R^d} \phi_n(\yb) [\gamma_{\hat{\zb},n}(\ddr\yb) - \gamma(\ddr\yb)] = \left(\frac{1}{2\pi}\right)^d 
\int_{\R^d} \hat{\phi}_n(\tb) [\hat{\gamma}_{\hat{\zb},n}(\tb) - \hat{\gamma}(\tb)] \ddr\tb
$$
where 
\begin{align*}
\phi_n(\yb) &:= \mathds{1}_{\Cb_{\hat{\zb}}^{-1/2}\mathscr E}(\yb) e^{-\sqrt{n}\ ^t\hat{\zb} \Cb_{\hat{\zb}}^{1/2}\yb} \\
\hat{\phi}_n(\tb) &:= \int_{\R^d} e^{i \tb \cdot \yb} \phi_n(\yb) \ddr \yb \\
\hat{\gamma}_{\hat{\zb},n}(\tb) &:= \int_{\R^d} e^{i \tb \cdot \yb} \gamma_{\hat{\zb},n}(\ddr\yb) \\
\hat{\gamma}(\tb) &:= \int_{\R^d} e^{i \tb \cdot \yb} \gamma(\ddr\yb)\ . 
\end{align*}
Therefore, we can write
\begin{align}
&\Big{|} \int_{\R^d} \phi_n(\yb) [\gamma_{\hat{\zb},n}(\ddr\yb) - \gamma(\ddr\yb)] \Big{|} \nonumber \\
&\leq \left(\frac{1}{2\pi}\right)^d \Big\{ \|\phi_n\|_{L^1} \Big( \int_{\{|\tb| \leq A(\hat{\zb}) \sqrt{n}\}} \!\!\!\!
|\hat{\gamma}_{\hat{\zb},n}(\tb) - \hat{\gamma}(\tb)| \ddr\tb + \int_{\{|\tb| > A(\hat{\zb}) \sqrt{n} \}} \!\!\!\! |\hat{\gamma}(\tb)| \ddr\tb \Big) \nonumber \\
&+ \Big{|} \int_{\{|\tb| > A(\hat{\zb}) \sqrt{n} \}} \!\!\!\! \hat{\phi}_n(\tb)\hat{\gamma}_{\hat{\zb},n}(\tb) \ddr\tb \Big{|} \Big\} \label{AppB4} 
\end{align}
with
\begin{align*}
\|\phi_n\|_{L^1} &:= \int_{\R^d} \phi_n(\yb) \ddr\yb = C'(d,\boldsymbol \xi, \mathscr E) \exp\left\{-\sqrt{n} \min_{\ub \in \mathscr E} \hat{\zb} \cdot \ub \right\} 
n^{-(d+1)/4}[1 + o(1)] \\
A(\hat{\zb}) &:= (\E[|\Yb_1 - \boldsymbol \xi|^2])^{3/2}/\E[|\Yb_1 - \boldsymbol \xi|^3] \ ,
\end{align*}
the former of the above identities following once again from the multidimensional Laplace method. Thus, for the first integral on the right-hand side of \eqref{AppB4}, an application of inequalities (8.22)-(8.23)
of \citet{Bhatt(10)}---which constitute a multidimensional generalization of the well-known Berry-Esseen inequalities---shows that
$$
\int_{\{|\tb| \leq A(\hat{\zb}) \sqrt{n}\}} \!\!\!\! |\hat{\gamma}_{\hat{\zb},n}(\tb) - \hat{\gamma}(\tb)| \ddr\tb \leq \frac{C''(d,\boldsymbol \xi, \mathscr E)}{\sqrt{n}}
$$ 
holds for some constant $C''(d,\boldsymbol \xi, \mathscr E)$ depending solely on $(d,\boldsymbol \xi, \mathscr E)$. Then, the second integral on the right-hand side of \eqref{AppB4} is asymptotically equivalent to
the integral 
$$
\int_{A(\hat{\zb}) \sqrt{n}}^{+\infty} e^{-\frac 12 \rho^2} \rho^{d-1} \ddr\rho \sim \exp\{-\frac 12 A(\hat{\zb})^2 n\} n^{d/2 - 1}
$$
as $n \rightarrow +\infty$. Lastly, the last integral on the right-hand side of \eqref{AppB4} has different behaviors according on whether the distribution $\mu$ is lattice or not. If $\limsup_{|\tb| \rightarrow +\infty}
|\hat{\mu}(\tb)| < 1$, then the integral at issue is exponentially small like the second integral on the right-hand side of \eqref{AppB4} described above. Otherwise, in the lattice case (which is of interest here),
we deduce the expansion \eqref{AppB5} by using the expression of $\hat{\gamma}_{\hat{\zb},n}$ (explicitly available in the lattice case) and by resorting once again to the multidimensional Laplace method.

\subsection{Proof of Theorem \ref{teo33} for $k>2$}

We start again from \eqref{ApiuLK} where the random variables $\Abb_{n}$ and $\Lb$, defined on the probability space $(\Omega,\mathscr{F},\text{Pr})$, are independent, $\Abb_n \sim Mult(n,\pb_1)$ and 
$\mathbf L = (L_1, \dots, L_{k-1})$ is a random vector with independent components with each $L_i$ having the Laplace distribution \eqref{Laplace}. 
After fixing $\lambda_{\alpha}$ such that $\K_{k-1}(\lambda_{\alpha}) = 1-\alpha$, according to Theorem \ref{teo32}, we have that
$$
1-\beta_n(\pb_1;\alpha) = \mathrm{Pr}[\Lambda_{k,n,T}(\pb_0) \leq \lambda_{\alpha}]\ . 
$$
First, for $\delta \in (0, 1/2)$, we define the event $E_{n,k}(\delta)\subset\Omega$ as
\begin{equation}\label{event_delta}
E_{n,k}(\delta) :=\left\{\omega\in\Omega\text{ $:$ } \frac{A_{n,i}(\omega)}{n} \in [\delta, 1-\delta], \text{for}\ i= 1, \dots, k-1 \right\} \ .
\end{equation}
Then, recalling that 
$$
\mathcal D_{KL}(\pb_0\ \|\ \pb_1) := \sum_{i=1}^{k-1} p_{0,i}\log\left(\frac{p_{0,i}}{p_{1,i}}\right) + \left(1- \sum_{i=1}^{k-1} p_{0,i} \right) \log\left(\frac{1- \sum_{i=1}^{k-1} p_{0,i}}{1-\sum_{i=1}^{k-1} p_{1,i}}\right)
$$
is a fixed quantity, we resort once again on a large deviation argument to choose $\delta$ sufficiently small so that
\begin{equation}\label{bound_largedevK}
\mathrm{Pr}[E_{n,k}(\delta)^c] \leq e^{-nC(\delta)}
\end{equation}
holds for some $C(\delta) > \mathcal D_{KL}(\pb_0\ \|\ \pb_1)$. In this way,  we can write
$$
1-\beta_n(\pb_1;\alpha) = \mathrm{Pr}[\Lambda_{k,n,T}(\pb_0) \leq \lambda_{\alpha}, E_{n,k}(\delta)] + \mathrm{Pr}[\Lambda_{k,n,T}(\pb_0) \leq \lambda_{\alpha}, E_{n,k}(\delta)^c]
$$
with
$$
\mathrm{Pr}[\Lambda_{k,n,T}(\pb_0) \leq \lambda_{\alpha}, E_{n,k}(\delta)^c] = o\left(e^{-n\mathcal D_{KL}(\pb_0\ \|\ \pb_1)}\right)\ .
$$
Then, for $M>0$, we define another event $E'_{n,k}(M)\subset\Omega$ as
\begin{equation}\label{event_M}
E'_{n,k}(M) :=\left\{\omega\in\Omega\text{ $:$ } \Big|\frac{\Abb_{n}(\omega)}{n} - \pb_0\Big| \leq \frac{M}{\sqrt{n}}\right\} 
\end{equation}
and we write
\begin{align*}
\mathrm{Pr}[\Lambda_{k,n,T}(\pb_0) \leq \lambda_{\alpha}, E_{n,k}(\delta)] &= \mathrm{Pr}[\Lambda_{k,n,T}(\pb_0) \leq \lambda_{\alpha}, E_{n,k}(\delta) \cap E'_{n,k}(M)] \\
&+\mathrm{Pr}[\Lambda_{k,n,T}(\pb_0) \leq \lambda_{\alpha}, E_{n,k}(\delta) \cap E'_{n,k}(M)^c]\ .
\end{align*}
The advantage of such a preliminary step is that, on $E_{n,k}(\delta) \cap E'_{n,k}(M)$, the assumptions of all the preparatory Lemmata contained in Appendix B are fulfilled and, by resorting to 
\eqref{final_expansion_LambdanTK}, we can write
\begin{align*}
&\left\{\Lambda_{k,n,T}(\pb_0) \leq \lambda_{\alpha}, E_{n,k}(\delta) \cap E'_{n,k}(M)\right\} \\
&= \left\{\ ^t\Zb_n \left(\mathbb{I}(\pb_0) - \frac{\mathsf{Var}[L_1]}{n}\mathbb{I}(\pb_0)^2 \right) \Zb_n + O\left(\frac{|\Zb_n|^3}{\sqrt{n}}\right) \leq \lambda_{\alpha}, E_{n,k}(\delta) \cap E'_{n,k}(M)\right\}
\end{align*}
where $\Zb_n$ is the same random variable as in \eqref{eq:VnK}. Moreover, if $M > \lambda_{\alpha}$, as we will choose, we have eventually that
\begin{align*}
&\left\{\Lambda_{k,n,T}(\pb_0) \leq \lambda_{\alpha}, E_{n,k}(\delta) \cap E'_{n,k}n(M)\right\} \\
&= \left\{\ ^t\Zb_n \left(\mathbb{I}(\pb_0) - \frac{\mathsf{Var}[L_1]}{n}\mathbb{I}(\pb_0)^2 \right) \Zb_n + O\left(\frac{|\Zb_n|^3}{\sqrt{n}}\right) \leq \lambda_{\alpha}, E_{n,k}(\delta)\right\}\ .
\end{align*}
It remains to show that, for a suitable choice of $M$, we have
\begin{equation} \label{emptysetK}
\left\{\Lambda_{k,n,T}(\pb_0) \leq \lambda_{\alpha}, E_{n,k}(\delta) \cap E'_{n,k}(M)^c\right\} = \emptyset
\end{equation}
eventually, yielding that 
\begin{align} \label{Power_1K}
1-\beta_n(\pb_1;\alpha) &= \mathrm{Pr}\left[\ ^t\Zb_n \left(\mathbb{I}(\pb_0) - \frac{\mathsf{Var}[L_1]}{n}\mathbb{I}(\pb_0)^2 \right) \Zb_n + O\left(\frac{|\Zb_n|^3}{\sqrt{n}}\right) \leq \lambda_{\alpha} \right] \\
&+ o\left(e^{-n\mathcal D_{KL}(p_0\ \|\ p_1)}\right)\ . \nonumber 
\end{align}
The proof of \eqref{emptysetK}, along with the proper choice of $M$, follows as an application of Lemma \ref{lem:emptyset} component-wise.

Afterwords, we come back to \eqref{Power_1K} by writing
\begin{equation} \label{eq:VnK}
\mathbf Z_n := \frac{\mathbf B_n - n\pb_0}{\sqrt{n}} =  \boldsymbol \Sigma(\pb_1)^{1/2} \mathbf V_n + \frac{1}{\sqrt{n}} \mathbf L - \sqrt{n} \boldsymbol \Delta
\end{equation}
where $\boldsymbol  \Delta := \pb_0 - \pb_1$,
$$
\mathbf V_n := \boldsymbol \Sigma(\pb_1)^{-1/2} \frac{\Abb_n - n\pb_1}{\sqrt{n}}
$$
and $\boldsymbol \Sigma(\pb_1)$ denotes the covariance matrix of $\Abb_n$. Exploiting the independence between $\Abb_n$ and $\Lb$, upon putting
$$
\mathbb{I}_n(\pb_0):= \mathbb{I}(\pb_0) - \frac{\mathsf{Var}[L_1]}{n}\mathbb{I}(\pb_0)^2\ ,
$$
we get
\begin{align} \label{Power_2K}
1-\beta_n(\pb_1;\alpha) &= \left(\frac{1}{\cepsm}\right)^{k-1} \sum_{\mathbf l \in\{-m, \dots, m\}^{k-1}} e^{-\varepsilon |\mathbf l|} \times \\
&\times \mathrm{Pr}\left[\ ^t\left(\Sigma(\pb_1)^{1/2} \mathbf V_n + \frac{1}{\sqrt{n}} \mathbf l - \sqrt{n} \boldsymbol \Delta\right) \mathbb{I}_n(\pb_0)
\left(\Sigma(\pb_1)^{1/2} \mathbf V_n + \frac{1}{\sqrt{n}} \mathbf l - \sqrt{n} \boldsymbol \Delta\right) \right. \nonumber \\
&\left. + O\left(\frac{\left|\Sigma(\pb_1)^{1/2} \mathbf V_n + \frac{1}{\sqrt{n}} \mathbf l - \sqrt{n} \boldsymbol \Delta\right|^3}{\sqrt{n}}\right) \leq \lambda_{\alpha} \right] + 
o\left(e^{-n\mathcal D_{KL}(\pb_0\ \|\ \pb_1)}\right)\ . \nonumber 
\end{align}
Now, we apply \eqref{main_ld} to a sequence $\{\Xb_n\}_{n \geq 1}$ of i.i.d. random vectors taking values in $\mathbb X_k :=
\{\xb = (x_1, \dots, x_{k-1}) \in \{0,1\}^{k-1}\ |\ x_1+ \dots + x_{k-1} \leq 1\}$, in such a way that 
\begin{equation} \label{MultiBernoulli}
\mathrm{Pr}\left[ \Xb_1 = \underbrace{(0, \dots, 0, 1, 0, \dots, 0)}_\text{1\ at\ the\ i-th\ position} \right] = p_{1,i} \qquad (i=1, \dots, k-1)\ .
\end{equation}
Moreover, we can put $\xib = \boldsymbol \Delta - \frac{1}{n} \mathbf l$ and
$$
\mathscr E = \{\xb \in \mathbb R^{k-1}\ |\ \ ^t\xb \mathbb{I}_n(\pb_0) \xb + O(|\xb|^3/\sqrt{n}) \lambda_{\alpha}\}\ . 
$$
Actually, the equation that defines $\mathscr E$ could be specified in a more precise way, as we have done for $k=2$. In fact, exploiting the analogy with the case $k=2$, we could write the big-$O$ term as
$$
\frac{1}{\sqrt{n}} \sum_{\substack{\nub \in \mathbb N_0^{k-1}\ :\\ |\nub| \leq 3}} \mathfrak B_{n,\nub}^{(\varepsilon, m)}(\pb_0) \xb^{\nub} + \frac{1}{n} \sum_{\substack{\nub \in \mathbb N_0^{k-1}\ :\\ |\nub| \leq 4}} \mathfrak C_{n,\nub}^{(\varepsilon, m)}(\pb_0) \xb^{\nub} + O(n^{-3/2})
$$
for suitable tensors $\mathfrak B_{n,\nub}^{(\varepsilon, m)}(\pb_0)$ and $\mathfrak C_{n,\nub}^{(\varepsilon, m)}(\pb_0)$ of third and fourth order respectively, both satisfying analogous expansions similar to
\eqref{B_gothic}--\eqref{C_gothic}. However, despite the cumbersome computation that would have needed to derive such quantities, we have already learnt from the case $k=2$ that they do not play any active role 
in the final result encapsulated in \eqref{LD1}. More precisely, we know that the 0-order terms in $\mathfrak B_{n,\nub}^{(\varepsilon, m)}(\pb_0)$ and $\mathfrak C_{n,\nub}^{(\varepsilon, m)}(\pb_0)$ would ensue 
from the Taylor expansion of the map $\epsilonb_n \mapsto \mathcal D_{KL}(\pb_0 + \epsilonb_n\ \|\ \pb_0)$, while the successive terms in their expansion---which explicitly depend on the Laplace perturbation---do
not affect the expressions of $c_1(\mathbf p_0, \mathbf p_{1}; \lambda_{\alpha})$ and $c_2(\mathbf p_0, \mathbf p_{1}; \lambda_{\alpha})$ in \eqref{LD1}. Indeed, after defining the vector $\hat{\mathbf z}$ by means 
of the identity $\nabla L(\hat{\mathbf z}) = \boldsymbol \xi$, where $L(\mathbf z) := \log \E[e^{\mathbf z\mathbf X_1}] - \pb_1 \cdot \zb$, we are now able to derive also the term $-\frac{1}{n}\log \mathfrak M_L\left(\nabla_{\mathbf p_0}\mathcal D_{KL}(\mathbf p_0\,\|\, \mathbf p_{1})\right)$ in \eqref{LD1}. After noticing that the above distribution \eqref{MultiBernoulli} is a member of the regular exponential family parametrized by the mean, we can resort to Lemma \ref{lem:exp_fam} and apply \eqref{kullback_cramerK} with $\boldsymbol\theta = \pb_1$ and $\boldsymbol\tau = \pb_0 + 
\frac 1n \lb + \boldsymbol \rho_n(\pb_0, \pb_1; \lambda_{\alpha})$, where $\boldsymbol \rho_n(\pb_0, \pb_1; \lambda_{\alpha}) = O(1/n)$, to conclude that
\begin{align*}
\hat{\mathbf z} \cdot \nabla L(\hat{\mathbf z}) - L(\hat{\mathbf z}) &=  \mathcal D_{KL}\left( \pb_0 + \frac 1n \lb + \boldsymbol \rho_n(\pb_0, \pb_1; \lambda_{\alpha})\ \|\ \pb_1\right)\\
&= \mathcal D_{KL}\left( \pb_0\ \|\ \pb_1\right) + \frac 1n \lb \cdot \nabla_{\pb_0} \mathcal D_{KL}\left( \pb_0\ \|\ \pb_1\right) \\
&+ \boldsymbol \rho_n(\pb_0, \pb_1; \lambda_{\alpha}) \cdot \nabla_{\pb_0} \mathcal D_{KL}\left( \pb_0\ \|\ \pb_1\right) \ .
\end{align*}
Therefore, combining this last equation with \eqref{main_ld} and \eqref{Power_2K} we get 
\begin{align*}
1-\beta_n(\pb_1;\alpha) &\sim n^{-k/4} \mathfrak M_L\left(\nabla_{\pb_0} \mathcal D_{KL}\left( \pb_0\ \|\ \pb_1\right)\right) \times \\
&\times \exp\Big\{-n \left[ 
\mathcal D_{KL}\left( \pb_0\ \|\ \pb_1\right) + \boldsymbol \rho_n(\pb_0, \pb_1; \lambda_{\alpha}) \cdot \nabla_{\pb_0} \mathcal D_{KL}\left( \pb_0\ \|\ \pb_1\right) \right] \\
& + \sqrt{n} \min_{\mathbf v \in \mathscr E} \hat{\mathbf z} \cdot \mathbf v \Big\}
\end{align*}
which entails the thesis of the theorem, upon noticing that the quantity $\min_{\mathbf v \in \mathscr E} \hat{\mathbf z} \cdot \mathbf v$ contains more explicit terms that depend on the Laplace perturbation only at 
the level $O(1/n)$. This ends the proof. 

\subsection{Proof of Proposition \ref{prop:large_k}}

First, we notice that 
\begin{equation} \label{CFEB}
\log \mathfrak M_L(\zb) = \sum_{i=1}^k \log \E[\exp\{L_i z_i\}] \qquad (\zb \in \mathbb R^k)
\end{equation}
and $\E[\exp\{L_i z\}] = \frac{1}{c_{\varepsilon,m}} \sum_{l=-m}^m \exp\{-\varepsilon |l| + lz\}$. In particular, if $z>0$, the main contribution comes from the sum over $l \in \{0, \dots, m\}$, so that 
$\E[\exp\{L_i z\}] = \frac{1}{c_{\varepsilon,m}} \sum_{l=0}^m \exp\{(z-\varepsilon) l\} + R_{\varepsilon,m}(z)$, where $R_{\varepsilon,m}(z)$ denotes a small remainder term. In particular,
if $z \geq \varepsilon(1+\eta)$, then
$$
\E[\exp\{L_i z\}] \geq \frac{1}{c_{\varepsilon,m}} \sum_{l=0}^m e^{\varepsilon\eta l} + R_{\varepsilon,m}(z)
$$
holds. If $\nu_k(\mathbf p_{0}, \mathbf p_{1}; \varepsilon, \eta) > 0$, then the relevant part of the sum on the right-hand side of \eqref{CFEB} is given by the sum of 
$\nu_k(\mathbf p_{0}, \mathbf p_{1}; \varepsilon, \eta)$ summands, each of which greater than $\log \E[\exp\{L_1 \varepsilon(1+\eta)\}]$. Finally, if $\nu_k(\mathbf p_{0}, \mathbf p_{1}; \varepsilon, \eta) = 0$ and $\varepsilon^2(1+\eta)^2\mathsf{Var}[L] < 1$, we use the Taylor expansion of the logarithm to obtain that 
$$
\log \E[\exp\{L_i z\}] = \frac 12 \mathsf{Var}[L] z^2 + o(z^2)
$$
holds as $z \to 0$. Thus, if all the components of $\nabla_{\mathbf p_0}\mathcal D_{KL}(\mathbf p_0\,\|\, \mathbf p_{1})$ are less than $\varepsilon (1+\eta)$, we get that the main contribution 
in the sum on the right-hand side of \eqref{CFEB} is given by $\frac k2 \mathsf{Var}[L] [\varepsilon(1+\eta)]^2$. This completes the proof. 

%%%%%%%%%%%%%%%%%%%%%%%%%%%%%%%%
%%%%%%%%%%%%%%%%%%%%%%%%%%%%%%%%
%%%%%%%%%%%%%%%%%%%%%%%%%%%%%%%%
%%%%%%%%%%%%%%%%%%%%%%%%%%%%%%%%

\section*{Acknowledgement}

The authors are very grateful to an Associate Editor and three Referees for their comments and suggestions that improved remarkably the paper. Emanuele Dolera and Stefano Favaro are grateful to Professor Yosef Rinott for suggesting the problem and for the numerous stimulating conversations and valuable suggestions. Emanuele Dolera and Stefano Favaro received funding from the European Research Council (ERC) under the European Union's Horizon 2020 research and innovation programme under grant agreement No 817257. Emanuele Dolera and Stefano Favaro gratefully acknowledge the financial support from the Italian Ministry of Education, University and Research (MIUR), ``Dipartimenti di Eccellenza" grant agreement 2018-2022.

%%%%%%%%%%%%%%%%%%%%%%%%%%%%%%%%
%%%%%%%%%%%%%%%%%%%%%%%%%%%%%%%%
%%%%%%%%%%%%%%%%%%%%%%%%%%%%%%%%
%%%%%%%%%%%%%%%%%%%%%%%%%%%%%%%%

%%%%%%%%%%%%%%%%%%%%%%%%%%%%%%%%
%%%%%%%%%%%%%%%%%%%%%%%%%%%%%%%%
%%%%%%%%%%%%%%%%%%%%%%%%%%%%%%%%
%%%%%%%%%%%%%%%%%%%%%%%%%%%%%%%%

\end{document}